\documentclass[11pt]{article}

\usepackage{tikz}
\usepackage{tikz-cd}
\usepackage[T1]{fontenc}
\usetikzlibrary{matrix,calc}
\usepackage[utf8]{inputenc}
\usepackage{amsmath}
\usepackage{amssymb}
\usepackage{amsthm}
\usepackage[backend=biber]{biblatex}
\usepackage[nottoc]{tocbibind}
\usepackage[english]{babel}
\usepackage{mathtools}
\usepackage{eufrak}
\usepackage{fancyhdr}
\usepackage{lastpage}
\usepackage{amssymb}
\usepackage{graphicx}
\usepackage{dsfont}
\usepackage{listings}
\usepackage{todonotes}
\usepackage{adjustbox}
\usepackage[all]{xy}
\usepackage{mathtools}
\usetikzlibrary{shapes.geometric, arrows}
\tikzstyle{startstop} = [rectangle, rounded corners, minimum width=3cm, minimum height=1cm,text centered, draw=black, fill=red!30]
\tikzstyle{io} = [rectangle, rounded corners, minimum width=3cm, minimum height=1cm,text centered, draw=black, fill=cyan!30]
\tikzstyle{process} = [rectangle, rounded corners, minimum width=3cm, minimum height=1cm,text centered, draw=black, fill=green!30]
\tikzstyle{decision} = [rectangle, rounded corners, minimum width=3cm, minimum height=1cm, text centered, draw=black, fill=pink!30]
\tikzstyle{arrow} = [thick,->,>=stealth]

\usetikzlibrary{shapes.geometric}

\graphicspath{ {./images/} }
\usepackage[a4paper, total={14cm, 8in}]{geometry}
\usepackage [pdftitle = {{Titre document}}, pdfstartview = Fit, pdfpagelayout = SinglePage,, pdfnewwindow = true, bookmarksnumbered = true, breaklinks, colorlinks, linkcolor = black, urlcolor = black, citecolor = cyan, linktoc = all]{hyperref}
\newtheorem{theorem}{Theorem}[section]
\newtheorem{proposition}[theorem]{Proposition}
\newtheorem{lemma}[theorem]{Lemma}
\newtheorem{remark}[theorem]{Remark}
\newtheorem{corollary}[theorem]{Corollary}
\newtheorem{definition}[theorem]{Definition}
\newtheorem{example}[theorem]{Example}

\theoremstyle{remark}

\renewcommand*{\proofname}{Proof}
\newcommand{\Z}{\mathbb{Z}}
\newcommand{\N}{\mathbb{N}}
\newcommand{\U}{\mathbb{U}}
\newcommand{\R}{\mathbb{R}}
\newcommand{\C}{\mathbb{C}}
\newcommand{\V}{\mathbb{V}}
\newcommand{\W}{\mathbb{W}}
\newcommand{\I}{\mathbb{I}}
\newcommand{\X}{\mathbb{X}}
\newcommand{\Y}{\mathbb{Y}}
\newcommand{\A}{\mathbb{A}}
\newcommand{\B}{\mathbb{B}}
\newcommand{\E}{\mathbb{E}}
\newcommand{\D}{\mathbb{D}}
\newcommand{\M}{\mathbb{M}}
\newcommand{\F}{\mathbb{F}}

\renewcommand{\S}{\mathcal{S}}

\makeatletter
\newcommand{\extp}{\@ifnextchar^\@extp{\@extp^{\,}}}
\def\@extp^#1{\mathop{\bigwedge\nolimits^{\!#1}}}
\makeatother

\renewcommand{\varprojlim}{%
  \mathop{\mathpalette\varlim@{\leftarrowfill@\scriptscriptstyle}}\nmlimits@
}

\begin{document}

\begin{titlepage}
   \begin{center}
       \vspace*{1cm}
      
      \Large{ONE DIAMOND TO RULE THEM ALL\\}
      
       \Large{Old and new topics about zigzag, levelsets and extended persistence}

     \bigskip

       \normalsize{Nicolas Berkouk$^*$ and Luca Nyckees}
       
       \medskip 
       
       EPFL, Laboratory for Topology and Neuroscience
       
       \bigskip 
       
       \bigskip 
       
       $^*$Corresponding author E-mail: nicolas.berkouk@epfl.ch
       
       Contributing author: luca.nyckees99@gmail.com
       
       \begin{abstract}
        Extended and zigzag persistence were introduced more than ten years ago, as generalizations of ordinary persistence. While overcoming certain limitations of ordinary persistence, they both enjoy nice computational properties, which make them an intermediate between ordinary and multi-parameter persistence, with already existing efficient software implementations. Nevertheless, their algebraic theory is more intricate, and in the case of extended persistence, was formulated only very recently. In this context, this paper presents a richly illustrated self-contained introduction to the foundational aspects of the topic, with an eye towards recent applications in which they are involved, such as computational sheaf theory and multi-parameter persistence.  
        \end{abstract}


      \bigskip
      
      \bigskip
      
      \small{\textbf{Conflict of interest:} On behalf of all authors, the corresponding author states that there is no conflict of interest. }

       \end{center}
\end{titlepage}

\tableofcontents
\newpage



\section{Introduction}

Topological Data Analysis (TDA) arose in the early 2000s with the purpose of developing techniques to estimate and analyze the shape of datasets (point clouds, networks, images…). Its first and main construction is so-called ordinary persistence, which allows tracking topological changes across a filtration of topological spaces or simplicial complexes. These changes are efficiently stored in a topological summary called barcodes. Ordinary persistence had major success in a varied range of applications (material science, neuroscience, network analysis…), because barcodes can be computed efficiently thanks to several open source software packages, and because they have a nice algebraic stability theory, through persistence modules and the interleaving distance. We refer the reader to \cite{oudot:hal-01247501} and \cite{dey2022computational} for self-contained introductions to the topic. Nevertheless, ordinary persistence has several limitations. First, it can only handle filtrations, which exclude the study of time varying data. Second, it is very easy to produce topological features in a dataset that remain undetected by ordinary persistence (see figure \ref{fig:cesub}).

To overcome these issues, the TDA community has introduced two enhancements of ordinary persistence in the one-parameter case: zigzag persistence \cite{carlsson2008zigzag} and extended persistence \cite{article}. The first one allows dealing with non-increasing filtrations, such as time-varying data, while the second one can detect previously unseen topological features. Moreover, a certain type of zigzag filtration, called the levelsets zigzag filtration, has been shown to be equivalent to extended persistence.

Both theories, like in the ordinary case, admit a notion of barcode that can be computed by already existing software packages \cite{gudhi:urm,Dionysus}. However, their algebraic stability counterpart is far more intricate than in the ordinary case, and has been achieved for extended persistence only very recently. 

There is more than ten years of literature on the developments of both extended and zigzag persistence, though it seems to us that these techniques remain less popular than ordinary persistence, while being in theory more powerful. This richly illustrated paper intend to give a comprehensive overview of the theoretical foundations of these theories, starting from their very beginnings, to their recent developments. We hope that it will help reignite the appeal for these methods. 

To illustrate the rich techniques enabled by these flavors of persistence, we end the paper by an exposition of some recent TDA methods relying on extended or zigzag persistence, together with a survey of the algorithmic advances for its computation.

\bigskip

\noindent \textbf{WebApp.} This work is accompanied with an open source webapp, illustrating the different constructions exposed in the paper, and showcasing their equivalence. The webapp is accessible at the following url: \url{https://github.com/LucaNyckees/zigzag-homology}.

\bigskip

\noindent \textbf{Notation.} If unspecified, a vector space is a $K$-vector space where $K$ is a fixed field. We denote by $\text{Vect}_K$ the category of $K$-vector spaces and linear maps. Moreover, homology functors $H_p(\cdot)$ are taken over $K$, \textit{i.e.} $H_p(\cdot)=H_p(\cdot,K)$. Whenever there is a diagram of topological spaces $\mathcal{X}$ or simplicial complexes, we write $H_p(\mathcal{X})$ for the diagram of vector spaces induced by applying the functor $H_p$. The same goes for cohomology functors $H^p(\cdot)$.

\begin{figure}
\begin{center}

\tikzset{every picture/.style={line width=0.75pt}} 

\begin{tikzpicture}[x=0.75pt,y=0.75pt,yscale=-1,xscale=1]

\draw  [draw opacity=0][fill={rgb, 255:red, 74; green, 144; blue, 226 }  ,fill opacity=0.2 ] (143,201.4) .. controls (143,198.42) and (145.42,196) .. (148.4,196) -- (237.6,196) .. controls (240.58,196) and (243,198.42) .. (243,201.4) -- (243,217.6) .. controls (243,220.58) and (240.58,223) .. (237.6,223) -- (148.4,223) .. controls (145.42,223) and (143,220.58) .. (143,217.6) -- cycle ;
\draw  [draw opacity=0][fill={rgb, 255:red, 74; green, 144; blue, 226 }  ,fill opacity=0.2 ] (96,116) .. controls (96,110.48) and (100.48,106) .. (106,106) -- (280,106) .. controls (285.52,106) and (290,110.48) .. (290,116) -- (290,146) .. controls (290,151.52) and (285.52,156) .. (280,156) -- (106,156) .. controls (100.48,156) and (96,151.52) .. (96,146) -- cycle ;
\draw  [draw opacity=0][fill={rgb, 255:red, 74; green, 144; blue, 226 }  ,fill opacity=0.51 ] (121,278) .. controls (121,272.48) and (125.48,268) .. (131,268) -- (253,268) .. controls (258.52,268) and (263,272.48) .. (263,278) -- (263,308) .. controls (263,313.52) and (258.52,318) .. (253,318) -- (131,318) .. controls (125.48,318) and (121,313.52) .. (121,308) -- cycle ;
\draw  [draw opacity=0][fill={rgb, 255:red, 74; green, 144; blue, 226 }  ,fill opacity=0.51 ] (99,376) .. controls (99,370.48) and (103.48,366) .. (109,366) -- (277,366) .. controls (282.52,366) and (287,370.48) .. (287,376) -- (287,406) .. controls (287,411.52) and (282.52,416) .. (277,416) -- (109,416) .. controls (103.48,416) and (99,411.52) .. (99,406) -- cycle ;
\draw    (193,320) -- (193,364) ;
\draw [shift={(193,366)}, rotate = 270] [color={rgb, 255:red, 0; green, 0; blue, 0 }  ][line width=0.75]    (10.93,-3.29) .. controls (6.95,-1.4) and (3.31,-0.3) .. (0,0) .. controls (3.31,0.3) and (6.95,1.4) .. (10.93,3.29)   ;
\draw [shift={(193,318)}, rotate = 90] [color={rgb, 255:red, 0; green, 0; blue, 0 }  ][line width=0.75]    (10.93,-3.29) .. controls (6.95,-1.4) and (3.31,-0.3) .. (0,0) .. controls (3.31,0.3) and (6.95,1.4) .. (10.93,3.29)   ;
\draw  [draw opacity=0][fill={rgb, 255:red, 144; green, 19; blue, 254 }  ,fill opacity=0.43 ] (363,376) .. controls (363,370.48) and (367.48,366) .. (373,366) -- (513,366) .. controls (518.52,366) and (523,370.48) .. (523,376) -- (523,406) .. controls (523,411.52) and (518.52,416) .. (513,416) -- (373,416) .. controls (367.48,416) and (363,411.52) .. (363,406) -- cycle ;
\draw  [draw opacity=0][fill={rgb, 255:red, 144; green, 19; blue, 254 }  ,fill opacity=0.43 ] (388,276) .. controls (388,270.48) and (392.48,266) .. (398,266) -- (488,266) .. controls (493.52,266) and (498,270.48) .. (498,276) -- (498,306) .. controls (498,311.52) and (493.52,316) .. (488,316) -- (398,316) .. controls (392.48,316) and (388,311.52) .. (388,306) -- cycle ;
\draw  [draw opacity=0][fill={rgb, 255:red, 144; green, 19; blue, 254 }  ,fill opacity=0.2 ] (349,116) .. controls (349,110.48) and (353.48,106) .. (359,106) -- (528,106) .. controls (533.52,106) and (538,110.48) .. (538,116) -- (538,146) .. controls (538,151.52) and (533.52,156) .. (528,156) -- (359,156) .. controls (353.48,156) and (349,151.52) .. (349,146) -- cycle ;
\draw    (193,223) -- (193,266) ;
\draw [shift={(193,268)}, rotate = 270] [color={rgb, 255:red, 0; green, 0; blue, 0 }  ][line width=0.75]    (10.93,-3.29) .. controls (6.95,-1.4) and (3.31,-0.3) .. (0,0) .. controls (3.31,0.3) and (6.95,1.4) .. (10.93,3.29)   ;
\draw    (193,156) -- (193,194) ;
\draw [shift={(193,196)}, rotate = 270] [color={rgb, 255:red, 0; green, 0; blue, 0 }  ][line width=0.75]    (10.93,-3.29) .. controls (6.95,-1.4) and (3.31,-0.3) .. (0,0) .. controls (3.31,0.3) and (6.95,1.4) .. (10.93,3.29)   ;
\draw    (443,157) -- (443,264) ;
\draw [shift={(443,266)}, rotate = 270] [color={rgb, 255:red, 0; green, 0; blue, 0 }  ][line width=0.75]    (10.93,-3.29) .. controls (6.95,-1.4) and (3.31,-0.3) .. (0,0) .. controls (3.31,0.3) and (6.95,1.4) .. (10.93,3.29)   ;
\draw    (443,318) -- (443,364) ;
\draw [shift={(443,366)}, rotate = 270] [color={rgb, 255:red, 0; green, 0; blue, 0 }  ][line width=0.75]    (10.93,-3.29) .. controls (6.95,-1.4) and (3.31,-0.3) .. (0,0) .. controls (3.31,0.3) and (6.95,1.4) .. (10.93,3.29)   ;
\draw [shift={(443,316)}, rotate = 90] [color={rgb, 255:red, 0; green, 0; blue, 0 }  ][line width=0.75]    (10.93,-3.29) .. controls (6.95,-1.4) and (3.31,-0.3) .. (0,0) .. controls (3.31,0.3) and (6.95,1.4) .. (10.93,3.29)   ;
\draw    (295,390) -- (353,390) ;
\draw [shift={(355,390)}, rotate = 180] [color={rgb, 255:red, 0; green, 0; blue, 0 }  ][line width=0.75]    (10.93,-3.29) .. controls (6.95,-1.4) and (3.31,-0.3) .. (0,0) .. controls (3.31,0.3) and (6.95,1.4) .. (10.93,3.29)   ;
\draw    (273,293) -- (377,293) ;
\draw [shift={(271,293)}, rotate = 0] [color={rgb, 255:red, 0; green, 0; blue, 0 }  ][line width=0.75]    (10.93,-3.29) .. controls (6.95,-1.4) and (3.31,-0.3) .. (0,0) .. controls (3.31,0.3) and (6.95,1.4) .. (10.93,3.29)   ;
\draw  [dash pattern={on 0.84pt off 2.51pt}]  (320,50) -- (320,263) ;
\draw  [dash pattern={on 0.84pt off 2.51pt}]  (320,310) -- (320,373) ;

\draw (135,66) node [anchor=north west][inner sep=0.75pt]  [font=\small] [align=left] {Discrete framework};
\draw (375,66) node [anchor=north west][inner sep=0.75pt]  [font=\small] [align=left] {Continuous framework};
\draw (102,113) node [anchor=north west][inner sep=0.75pt]  [font=\small] [align=left] {Filtration of simplicial complex};
\draw (150,203) node [anchor=north west][inner sep=0.75pt]  [font=\small] [align=left] {Critical values};
\draw (127,278) node [anchor=north west][inner sep=0.75pt]  [font=\small] [align=left] {Extended persistence};
\draw (105,374) node [anchor=north west][inner sep=0.75pt]  [font=\small] [align=left] {Levelsets zigzag persistence};
\draw (364,115) node [anchor=north west][inner sep=0.75pt]  [font=\small] [align=left] {Continuous function with \\pfd levelsets co-homology};
\draw (404,278) node [anchor=north west][inner sep=0.75pt]  [font=\small] [align=left] {RISC functor};
\draw (377,374) node [anchor=north west][inner sep=0.75pt]  [font=\small] [align=left] {Levelsets persistence};
\draw (275,275) node [anchor=north west][inner sep=0.75pt]  [font=\scriptsize] [align=left] {Block decomposition};
\draw (304,398) node [anchor=north west][inner sep=0.75pt]  [font=\scriptsize] [align=left] {Inclusion};
\draw (133,136) node [anchor=north west][inner sep=0.75pt]  [font=\small] [align=left] {Morse type function};
\draw (99,327) node [anchor=north west][inner sep=0.75pt]  [font=\scriptsize] [align=left] {Pyramid theorem};
\draw (158,393) node [anchor=north west][inner sep=0.75pt]  [font=\footnotesize] [align=left] {(Section 2.3)};
\draw (158,296) node [anchor=north west][inner sep=0.75pt]  [font=\footnotesize] [align=left] {(Section 2.4)};
\draw (408,393) node [anchor=north west][inner sep=0.75pt]  [font=\footnotesize] [align=left] {(Section 4)};
\draw (111,342) node [anchor=north west][inner sep=0.75pt]  [font=\footnotesize] [align=left] {(Section 3)};
\draw (413,295) node [anchor=north west][inner sep=0.75pt]  [font=\footnotesize] [align=left] {(Section 5)};
\draw (303,412) node [anchor=north west][inner sep=0.75pt]  [font=\scriptsize] [align=left] {of posets};

\end{tikzpicture}

\end{center}
\caption{General pipeline illustration. }
\label{pipeline}
\end{figure}
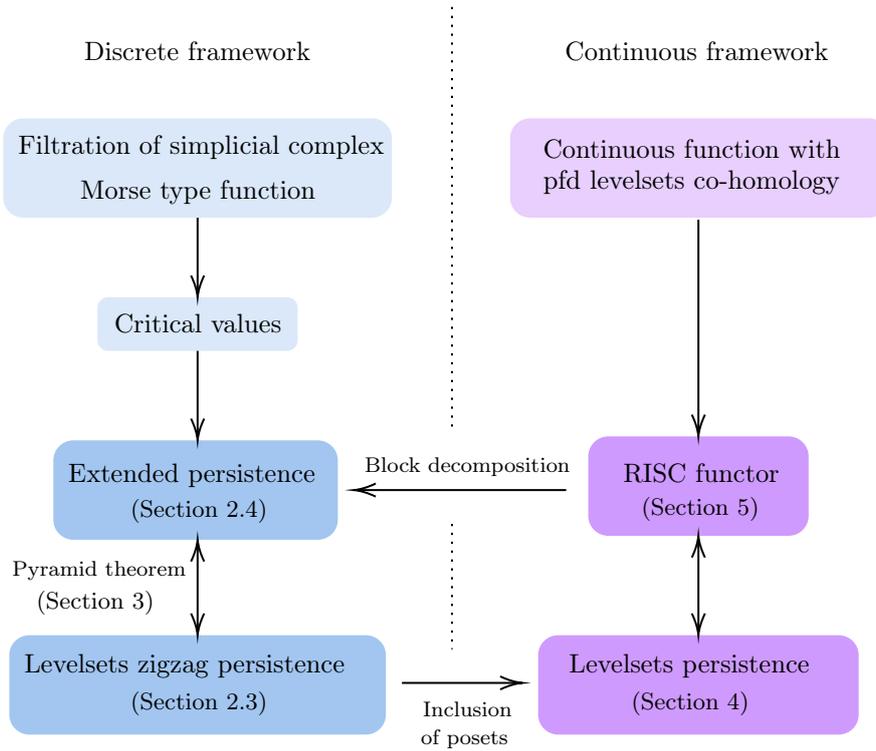

\section{Discrete flavors of  persistence}

This section is devoted to define the discrete versions of extended and zigzag persistence, as they were first introduced in \cite{carlsson2008zigzag} and \cite{article}. The existence of barcodes in both case relies on Gabriel's theorem for type A quivers, which we start by reviewing. 

\subsection{A short introduction to Gabriel's theorem}

One of the cornerstone stone of the different flavors of persistence (ordinary, extended, zigzag), that we will present in the next sections, is Gabriel's theorem on representation of quivers. For a detailed exposition, we refer to \cite[Appendix A]{oudot:hal-01247501}.

\begin{definition} A quiver $Q$ is an oriented graph, that is, a tuple $(Q_0,Q_1)$ where $Q_0$ and $Q_1$ are sets, equipped with two maps $h,t : Q_1 \longrightarrow Q_0$. $Q_0$ is the set of \emph{vertices} of $Q$, $Q_1$ the set of \emph{edges}, and $h$ (resp. $t$) associates to an edge its \emph{head} (resp. its \emph{tail}).
\end{definition} 


The following are examples of quivers:

\bigskip

\begin{center}
     $R = \xymatrix{ &   &  &  & &\bullet \\
    \bullet \ar[r] & \bullet \ar[r] & \bullet \ar[r] & \bullet \ar[r] & \bullet \ar[rd] \ar[ru] \\ 
     &   &  &  & &\bullet}$
     
     \bigskip

  $S = \xymatrix{\bullet \ar@(ur,ul)}$
  
  \bigskip
    
    $T = \xymatrix{
    \bullet \ar[r] & \bullet & \bullet \ar[l] \ar[r] & \bullet \ar[r] & \bullet }$

\end{center}

\begin{definition}
A quiver $Q$ is said to be of type $A$, if there exists an integer $n$ such that $Q_0 = \{0,...,n\}$, and there is exactly one edge between $i$ and $i+1$ (in any direction), and no other edges.
\end{definition}

\begin{remark}
The quiver $T$ in the last line of the above examples is a quiver of type $A$.
\end{remark}

\begin{definition}
Let $Q$ be a quiver. A representation of $Q$ over the field $K$ is the data, for all $q \in Q_0$, of a finite dimensional $K$-vector space $\V_q$, and for all $e \in Q_1$ of a $K$-linear map $\V_e : \V_{t(e)} \longrightarrow \V_{h(e)}$.

\noindent We shall simply denote the corresponding representation by $\V$.
\end{definition}

The following is a representation of the quiver $R$:

\begin{center}
     $\xymatrix{ &   &  &  & & K \\
    K \ar[r]^{\tiny \begin{pmatrix} 1 \\ 0 \end{pmatrix}} & K^2 \ar[r]^{\text{id}_{K^2}} & K^2 \ar[r]^{(1~1)} & K \ar[r]\ar[r]^{\text{id}_K} & K \ar[rd]_{\text{id}_K} \ar[ru]^{\text{id}_K} \\ 
    &   &  &  & & K}$
    
\end{center}


\begin{definition}[Interval representation]
Let $Q$ be a quiver of type $A$. The interval representation of $Q$ with birth $b \in Q_0$ and death $d \in Q_0$ ($b \leq d$), is defined by:
\begin{equation*}
\I(b,d)_i= \begin{cases}
             K  & \text{if } b\leq i\leq d \\
             0  & \text{otherwise,}
       \end{cases}
\end{equation*}
and equipped with identity maps between two adjacent copies of $K$ and zero maps elsewhere.
\end{definition}

The interval representation of the quiver $T$ with birth $1$ and death $3$ is:

\[ \I(1,3)  = \xymatrix{
    0 \ar[r]^0 & K & K \ar[l]_{\text{id}_K} \ar[r]^{\text{id}_K} & K \ar[r]^{\text{id}_K} & 0 }
    \]

\begin{definition}
Let $Q$ be a quiver, and $\V,\W$ be two $K$-representations of $Q$. A morphism of representation from $\V$ to $\W$ is the data, for all $q \in Q_0$, of a $K$-linear map $\varphi_q : \V_q \to \W_q$, such that for all $e \in Q_1$, the following diagram is commutative:

$$\xymatrix{\V_{t(e)} \ar[r]^{\V_e} \ar[d]_{\varphi_{t(e)}} & \V_{h(e)} \ar[d]^{\varphi_{h(e)}} \\
\W_{t(e)} \ar[r]_{\W_e} & \W_{h(e)} }.$$

\noindent We denote this data by $\varphi : \V \longrightarrow \W$. We say that $\varphi$ is an isomorphism if and only if $\varphi_q$ is an isomorphism for all $q \in Q_0$.
\end{definition}

\begin{definition}
Let $Q$ be a quiver, and let $(\V_i)_{i\in I}$ be a family of representations of $Q$. The direct sum of $(\V_i)_{i\in I}$, denoted by $\bigoplus_{i\in I} \V_i$, is defined pointwise, for all $q \in Q_0$ and all $e \in Q_1$, by:

\[\left( \bigoplus_{i\in I} \V_i \right )_q := \bigoplus_{i\in I} (\V_i)_q, ~~~\textnormal{and}~~~ \left( \bigoplus_{i\in I} \V_i \right )_e := \bigoplus_{i\in I} (\V_i)_e. \]
\end{definition}

We can now state Gabriel's theorem for type $A$ quivers.

\begin{theorem}
Let $Q$ be a type $A$ quiver, and $\V$ be a representation of $Q$. Then, there exists $N \in \Z_{\geq 0}$, and pairs of integers $(b_1,d_1),...,(b_N,d_N)\in Q_0^2$, such that: $$\V \simeq \I(b_1,d_1)\oplus\cdots\oplus\I(b_N,d_N).$$ 

\noindent Moreover, the pairs of integers $(b_i,d_i)$ are unique up to reordering. 
\end{theorem}

Note that since a pair $(b_i,d_i)$ can appear several times in the decomposition, the collection of the $(b_i,d_i)$ carries the structure of a multi-set, that is, a set where elements have a multiplicity. This multi-set, denoted $\B(\V)$, is called the \emph{barcode} of $\V$.

\subsection{Zigzag Persistence}
\label{zigzag_theory}

While ordinary persistence studies filtrations of topological spaces, that induces at the homological level  representations of equi-oriented (all arrows go in the same direction) type A quivers, \textit{zigzag persistence} \cite{carlsson2008zigzag} (Carlsson and De Silva) aims to generalize the setting of ordinary persistent homology to all type A quivers, regardless of the orientation of their edges. Indeed, Gabriel's theorem assures the existence of barcodes for all type A quivers. Beyond the fact that zigzag persistence successfully generalizes ordinary persistence, it is initially motivated by some concrete problems arising in computational topology (such as topological bootstrapping, probability density estimation on point cloud data, or time varying data). 

For a collection of simplicial complexes $\mathcal{X}=\{\X_i\}_{i=0}^n$ and a fixed $p\in\N$, one can consider a diagram of embeddings associated to $\mathcal{X}$ by taking pairwise unions as follows.
\begin{center}
\begin{tikzcd}
                    & \X_0\cup \X_1 &                                & \X_1\cup \X_2 &                                & ... &                     \\
\X_0 \arrow[ru] &                  & \X_1 \arrow[lu] \arrow[ru] &                  & \X_2 \arrow[lu] \arrow[ru] &     & \X_n \arrow[lu]
\end{tikzcd}
\end{center}

Applying simplicial homology to the above, we obtain a diagram of vector spaces:

\[
H_p(\X_0) \longrightarrow H_p(\X_0 \cup \X_1) \longleftarrow H_p(\X_1)  \longrightarrow  H_p(\X_1 \cup \X_2) \longleftarrow H_p(\X_2) \longrightarrow \cdots \longleftarrow H_p(\X_n).
\]

This simple example, that allows to compare a sequence a simplicial complexes that is not a filtration, motivates the following definition.

\begin{definition}[Zigzag module]
A \emph{zigzag persistence module} is a diagram in the category of finite dimensional vector spaces (each arrow is a linear map), of the form 
\[
V_0\longrightarrow V_1\longleftarrow V_2\longrightarrow \cdots\longleftarrow V_n,
\]
alternating between arrows $\longrightarrow$ and $\longleftarrow$.

\end{definition} 

\noindent In particular, a zigzag persistence module is a representation of a type A quiver. Therefore, it admits a barcode, by Gabriel's theorem. In section \ref{LZZ_subsection}, we will study a particular type of zigzag persistence module, associated to Morse type functions.

\subsection{Levelsets Zigzag Persistence}

\label{LZZ_subsection}

\begin{definition}[Morse type]
Let $\X$ be a topological space. A continuous function $f:\X\rightarrow \R$ is of \emph{Morse type} if there are finitely many so-called \emph{critical values} $a_1<...<a_n$ such that for any interval $I$ of the form $(-\infty,a_1)$, $(a_n,\infty)$ or $(a_i,a_{i+1})$ for some $i\in\{1,...,n-1\}$, there exists a topological space $\Y_I$ and an homeomorphism $f^{-1}(I) \stackrel{\sim}{\longrightarrow} \Y_I \times I$ such that, with $p_2$ the second coordinate projection, the following diagram is commutative:

\[ \xymatrix{f^{-1}(I) \ar[rr]^{\sim} \ar[rd]_{f} & & \Y_I \times I \ar[ld]^{p_2} \\
& \R &}. \]

Moreover, we ask that the spaces $\Y_I$ have finitely generated homology groups.
\end{definition}

\noindent Consider a topological space $\X$ paired with a Morse type function $f:\X\rightarrow\R$, with critical values $a_1<...<a_n$. Choose in-between values $s_0,...,s_n\in\R$ that satisfy $$-\infty<s_0<a_1<s_1<\cdots<s_{n-1}<a_n<s_n<\infty,$$ and denote by $\X_i^j$ the space $f^{-1}([s_i,s_j])$ for any pair of indices $i\leq j$. The \textit{levelsets zigzag filtration} of the pair $(\X,f)$ is defined by the sequence of inclusions of topological spaces: 
$$ZZ(f) : \X_0^0\rightarrow\X_0^1\leftarrow\X_1^1\rightarrow\cdots\leftarrow\X_{n-1}^{n-1}\rightarrow\X_{n-1}^n\leftarrow\X_n^n.$$ 

\begin{definition}
The $p$-th levelsets persistence module associated to the Morse type function $f : \X \longrightarrow \R$ is, with the previous notations, the sequence of vector spaces:

$$ZZ_p(f) : H_p(\X_0^0)\rightarrow H_p(\X_0^1)\leftarrow H_p(\X_1^1) \rightarrow\cdots\leftarrow H_p(\X_{n-1}^{n-1}) \rightarrow H_p(\X_{n-1}^n) \leftarrow H_p(\X_n^n).$$ 
\end{definition}

The $p$-th levelsets persistence module associated to $f$ is a representation of a type A quiver, therefore, it admits a barcode, that we shall denote by $\B(ZZ_p(f))$.

\medskip

\noindent \textbf{Notation.} Intervals of the form $\I(b,d)$ appearing in the levelsets zigzag barcode are denoted $[\Y,\Z]$, where $\Y$ and $\Z$ are the subspaces at position $b$ and $d$ in the above sequence, respectively. For example, $\I(0,2)$ would be denoted by $[\X_0^0,\X_1^1]$.\\




\subsection{Extended Persistent Homology}
\label{extended}

On a manifold $M$ with Morse function $f$, persistent homology works by pairing critical points that correspond to the birth and death of a homology class. However, there are homology classes that never die, called \textit{essential homology classes}, for which the pairing does not apply. The goal of introducing extended persistent homology is to extend the pairing to all homology classes, including essential ones. This unlocks more information about topological features represented by essential homology. This section focuses on introducing extended persistence by following \cite{article}.

\subsubsection{On Manifolds}
\label{manifolds}

\noindent Let $M$ be a $d$-dimensional manifold with Morse function $f$ and critical values $\infty<a_1<...<a_n<\infty$. Choose a set of regular values $s_0<...<s_n$ satisfying $$\infty<s_0<a_1<s_1<...<a_n<s_n<\infty.$$ Define sub-levelsets $M_k=f^{-1}((-\infty,s_k])$ for any $k\in\{0,...,n\}$. Let $H_p(\cdot)$ (resp. $H^p(\cdot)$) denote the $p$-th singular homology (resp. cohomology) functor. Poincaré duality (\autocite[Section~3.3]{Hatcher:478079}) tells us that if the $d$-dimensional manifold $M$ is oriented and closed (\textit{i.e.} compact and without boundary), then there are group isomorphisms $$H_p(M)\cong H^{d-p}(M).$$ Together with inclusion-induced maps, this gives the homology-cohomology sequence $$H_p(M_0)\rightarrow \cdots \rightarrow H_p(M_m)\rightarrow H^{d-p}(M_m)\rightarrow \cdots\rightarrow H^{d-p}(M_0).$$ 

\noindent Note that groups in the second half of the sequence get progressively shrunk down to $0$, and thus all homology classes, including essential ones, die at some point. Now, to turn the problem into a purely homological one, the trick is to use Lefschetz duality (\autocite[Theorem~3.43]{Hatcher:478079}), that translates cohomology to relative homology by providing group isomorphisms $$H^{d-p}(M_p)\cong H_p(M_k,\partial M_k).$$ 
\noindent Using excision, one may show that $H_p(M_k,\partial M_k)\cong H_p(M,M^{m-k})$, leading to 
$$H_p(M_0)\rightarrow \cdots \rightarrow H_p(M_m)\rightarrow H_p(M,M^{0})\rightarrow \cdots\rightarrow H_p(M,M^{m}).$$ 
\noindent The advantage of this version of the sequence is that it can generalize to the case of simplicial complexes, which is the subject of the next subsection.

\subsubsection{On Simplicial Complexes}

Let $K$ be the triangulation of a manifold $M$, and $f$ a real-valued injective function defined on the vertex set of $K$. Let $\{u_1,...,u_n\}$ be the unique ordered set of vertices of $K$ satisfying $f(u_1)<...<f(u_n)$. Moreover, let $K_i$ denote the sub-complex of $K$ spanned by the vertices $u_1,...,u_i$, and $L_{n-i}$ the sub-complex of $K$ spanned by the vertices $u_{i+1},...,u_n$. Assuming that $M$ is compact, closed and oriented, the duality results of section \ref{manifolds} still hold at the simplicial homology level, and we can consider the following extended sequence:
$$H_p(K_0)\rightarrow \cdots \rightarrow H_p(K_n)\rightarrow H_p(K,L_0)\rightarrow \cdots\rightarrow H_p(K,L_n).$$ 
\noindent  \\

\noindent \textbf{Notation.} For a fixed $p\in\N$, one can adopt an indexing convention by assigning $H_p(K_i)\mapsto i$ and  $H_p(K,L_i)\mapsto n+i$.\\

\noindent Poincaré and Lefschetz duality do not hols for general simplicial complex. Though, the above extended homological sequence still make sens for any abstract simplicial complex endowed with an injective real-valued function.

\begin{definition}
Let $K$ be a simplicial complex with an injective function $f:K\rightarrow\R$. The extended persistence module of the pair $(K,f)$ is the diagram $$EP_p(f) : H_p(K_0)\rightarrow \cdots \rightarrow H_p(K_n)\rightarrow H_p(K,L_0)\rightarrow \cdots\rightarrow H_p(K,L_n)$$ induced by an ordering $\{u_1,...,u_n\}$ of the vertices of $K$ as above.

\end{definition}

The extended persistence module of a pair $(K,f)$ is a representation of a quiver of type A, therefore, it admits a barcode.

\begin{definition}[Interval Types]
Four types of intervals arise from the extended persistence barcode of a pair $(K,f)$. \textit{Type I} or \textit{Ord} (resp. \textit{Type II} or \textit{Rel}) stands for intervals whose lifespan is contained within the first (resp. second) part of the sequence. \textit{Type III} or \textit{EP}($+$) (resp. \textit{Type IV} or \textit{EP}($-$)) stands for intervals whose birth appears in the first part of the sequence (resp. death) and whose death (resp. birth) appears in the second one. 
\end{definition}

\begin{figure}
    \centering

\tikzset{every picture/.style={line width=0.75pt}} 

\begin{tikzpicture}[x=0.75pt,y=0.75pt,yscale=-1,xscale=1]

\draw    (171,250) -- (475,250) ;
\draw [shift={(477,250)}, rotate = 180] [color={rgb, 255:red, 0; green, 0; blue, 0 }  ][line width=0.75]    (10.93,-3.29) .. controls (6.95,-1.4) and (3.31,-0.3) .. (0,0) .. controls (3.31,0.3) and (6.95,1.4) .. (10.93,3.29)   ;
\draw  [dash pattern={on 0.84pt off 2.51pt}]  (316,145) -- (316,326) ;
\draw    (210,247) -- (210,254) ;
\draw    (240,247) -- (240,254) ;
\draw    (270,247) -- (270,254) ;
\draw    (300,247) -- (300,254) ;
\draw    (330,247) -- (330,254) ;
\draw    (360,247) -- (360,254) ;
\draw    (390,247) -- (390,254) ;
\draw    (420,247) -- (420,254) ;
\draw [color={rgb, 255:red, 46; green, 78; blue, 167 }  ,draw opacity=1 ]   (210,247) .. controls (237,226) and (273,225) .. (300,247) ;
\draw [shift={(255.47,230.87)}, rotate = 177.4] [fill={rgb, 255:red, 46; green, 78; blue, 167 }  ,fill opacity=1 ][line width=0.08]  [draw opacity=0] (10.72,-5.15) -- (0,0) -- (10.72,5.15) -- (7.12,0) -- cycle    ;
\draw [color={rgb, 255:red, 144; green, 19; blue, 254 }  ,draw opacity=0.55 ]   (300,254) .. controls (326.2,277.6) and (395.2,275.6) .. (420,254) ;
\draw [shift={(360.53,270.95)}, rotate = 0.63] [fill={rgb, 255:red, 144; green, 19; blue, 254 }  ,fill opacity=0.55 ][line width=0.08]  [draw opacity=0] (10.72,-5.15) -- (0,0) -- (10.72,5.15) -- (7.12,0) -- cycle    ;
\draw [color={rgb, 255:red, 77; green, 158; blue, 235 }  ,draw opacity=1 ]   (330,247) .. controls (355.2,228.6) and (394.2,226.6) .. (420,247) ;
\draw [shift={(375.26,232.44)}, rotate = 357.4] [fill={rgb, 255:red, 77; green, 158; blue, 235 }  ,fill opacity=1 ][line width=0.08]  [draw opacity=0] (10.72,-5.15) -- (0,0) -- (10.72,5.15) -- (7.12,0) -- cycle    ;
\draw [color={rgb, 255:red, 116; green, 104; blue, 207 }  ,draw opacity=1 ]   (210,254) .. controls (241.2,275.6) and (303.2,274.6) .. (330,254) ;
\draw [shift={(270.28,269.83)}, rotate = 181.41] [fill={rgb, 255:red, 116; green, 104; blue, 207 }  ,fill opacity=1 ][line width=0.08]  [draw opacity=0] (10.72,-5.15) -- (0,0) -- (10.72,5.15) -- (7.12,0) -- cycle    ;
\draw  [color={rgb, 255:red, 46; green, 78; blue, 167 }  ,draw opacity=1 ][fill={rgb, 255:red, 46; green, 78; blue, 167 }  ,fill opacity=1 ] (402.04,299.55) -- (414.96,299.55) .. controls (416.64,299.55) and (418,300.48) .. (418,301.63) .. controls (418,302.78) and (416.64,303.71) .. (414.96,303.71) -- (402.04,303.71) .. controls (400.36,303.71) and (399,302.78) .. (399,301.63) .. controls (399,300.48) and (400.36,299.55) .. (402.04,299.55) -- cycle ;
\draw  [color={rgb, 255:red, 77; green, 158; blue, 235 }  ,draw opacity=1 ][fill={rgb, 255:red, 77; green, 158; blue, 235 }  ,fill opacity=1 ] (402.04,313.41) -- (414.96,313.41) .. controls (416.64,313.41) and (418,314.35) .. (418,315.49) .. controls (418,316.64) and (416.64,317.57) .. (414.96,317.57) -- (402.04,317.57) .. controls (400.36,317.57) and (399,316.64) .. (399,315.49) .. controls (399,314.35) and (400.36,313.41) .. (402.04,313.41) -- cycle ;
\draw  [color={rgb, 255:red, 116; green, 104; blue, 207 }  ,draw opacity=1 ][fill={rgb, 255:red, 116; green, 104; blue, 207 }  ,fill opacity=1 ] (402.04,327.28) -- (414.96,327.28) .. controls (416.64,327.28) and (418,328.21) .. (418,329.36) .. controls (418,330.51) and (416.64,331.44) .. (414.96,331.44) -- (402.04,331.44) .. controls (400.36,331.44) and (399,330.51) .. (399,329.36) .. controls (399,328.21) and (400.36,327.28) .. (402.04,327.28) -- cycle ;
\draw  [color={rgb, 255:red, 144; green, 19; blue, 254 }  ,draw opacity=0.55 ][fill={rgb, 255:red, 144; green, 19; blue, 254 }  ,fill opacity=0.55 ] (402.04,341.84) -- (414.96,341.84) .. controls (416.64,341.84) and (418,342.77) .. (418,343.92) .. controls (418,345.07) and (416.64,346) .. (414.96,346) -- (402.04,346) .. controls (400.36,346) and (399,345.07) .. (399,343.92) .. controls (399,342.77) and (400.36,341.84) .. (402.04,341.84) -- cycle ;

\draw (202,256.4) node [anchor=north west][inner sep=0.75pt]    {$a_{1}$};
\draw (292,257.4) node [anchor=north west][inner sep=0.75pt]    {$a_{n}$};
\draw (320,222.4) node [anchor=north west][inner sep=0.75pt]    {$\overline{ \begin{array}{l}
a_{n}\\
\end{array}}$};
\draw (410,222.4) node [anchor=north west][inner sep=0.75pt]    {$\overline{ \begin{array}{l}
a_{1}\\
\end{array}}$};
\draw (422.89,295.9) node [anchor=north west][inner sep=0.75pt]  [font=\scriptsize]  {$ \begin{array}{l}
Type\ I\\
Type\ II\\
Type\ III\\
Type\ IV
\end{array}$};

\end{tikzpicture}

    \caption{Types of intervals appearing in the extended persistence barcode of $(K,f)$, where $\{a_1,...,a_n\}$ is the filtration and $\{a_1,...,a_n,\bar{a_n},...,\bar{a_1}\}$ is the extended filtration.}
    \label{fig:interval_types}
\end{figure}
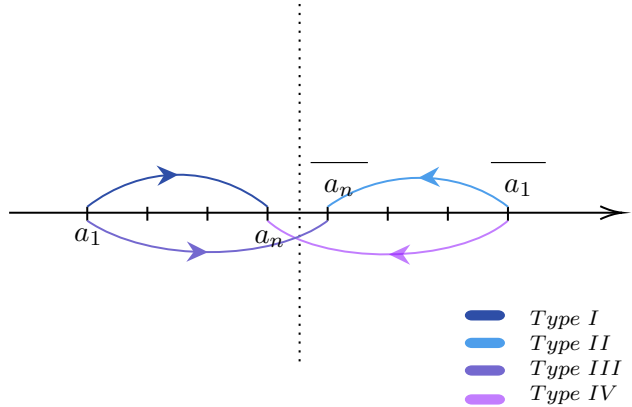

\section{The Pyramid Theorem}
\label{pyr1}

In \cite{10.1145/1542362.1542408}, the authors associate to a Morse type function, a poset of pair of spaces that is pyramid shaped. This pyramid allows to intermediate between extended and levelsets zigzag persistence, thanks to so-called diamond moves. In the end, this yields a diagram bijection between the levelsets zigzag persistence and the extended persistence of a pair $(\X,f)$ of Morse type having finitely generated levelsets homology.  This section reviews these results. We consider the same mathematical context as Section \ref{LZZ_subsection}.

\subsection{Mayer-Vietoris Diamonds}

The central algebraic concept at the core of the pyramidal construction is the one of Mayer-Vietoris diamonds.

\begin{definition}[Exact Square]
An \emph{exact square} is a diagram of vector spaces 
\begin{center}
\begin{tikzcd}
V_3 \arrow[r, "g_2"]                  & V_4                  \\
V_1 \arrow[r, "f_1"] \arrow[u, "f_2"] & V_2 \arrow[u, "g_1"]
\end{tikzcd}
\end{center}
that satisfies the condition $\mathrm{Ker}(V_2\oplus V_3\rightarrow V_4)=\mathrm{Im}(V_1\rightarrow V_2\oplus V_3)$ in the sequence 

$$V_1\longrightarrow V_2\oplus V_3\longrightarrow V_4,$$
where $(V_1\rightarrow V_2\oplus V_3)=f_1\oplus f_2$ and $(V_2\oplus V_3\rightarrow V_4)=g_1-g_2$.
\end{definition}

Two zigzag modules $\V^+$ and $\V^-$ are said to \textit{differ by an exact square at $k$} if one has $$\V^+:V_0\longleftrightarrow \cdots\longleftrightarrow V_{k-1}\longrightarrow V^+\longleftarrow V_{k+1}\longleftrightarrow\cdots\longleftrightarrow V_n,$$
$$\V^-:V_0\longleftrightarrow \cdots\longleftrightarrow V_{k-1}\longleftarrow V^-\longrightarrow V_{k+1}\longleftrightarrow\cdots\longleftrightarrow V_n$$ such that the square  
\begin{center}
    \begin{tikzcd}
V_{k+1} \arrow[r]       & V^+               \\
V^- \arrow[r] \arrow[u] & V_{k-1} \arrow[u]
\end{tikzcd}
\end{center}
is exact. 
\begin{theorem}[Diamond Principle](\autocite[Theorem~ 5.6]{carlsson2008zigzag})
For two zigzag modules $\V^+$ and $\V^-$ differing by an exact square at index $k$, there is a partial bijection between $\B({\V^+})$ and $\B({\V^-})$ given by the correspondence table 
\begin{center}
\begin{tabular}{|c c c|} 
 \hline
Condition & $\B({\V^+})$ & $\B({\V^-})$\\ [0.5ex] 
 \hline
 $b\leq k-1$ & $[b,k]$ & $[b,k-1]$ \\ 
\hline
 $d\geq k+1$ & $[k,d]$ & $[k+1,d]$\\
 \hline
 $b\geq k\lor d\leq k$ & $[b, d]$ & $[b, d]$ \\
 \hline
\end{tabular}
\end{center}
\end{theorem}

\noindent Note that the bijection above is \textit{partial}, as no matching is provided for intervals $[k,k]$. Now, the Diamond Principle admits a stronger version, when exact squares are given by taking the homology of Mayer-Vietoris diamonds.\\

\noindent Mayer-Vietoris diamonds are squares of topological spaces of the form 

\begin{center}
\begin{tikzcd}
V_2 \arrow[r, hook]                       & V_1\cup V_2           \\
V_1\cap V_2 \arrow[u, hook] \arrow[r, hook] & V_1 \arrow[u, hook]
\end{tikzcd}.
\end{center}

\noindent Applying a homology functor $H_p(\cdot)$ to this square provides an exact square 

\begin{center}
\begin{tikzcd}
H_p(V_2) \arrow[r]                       & H_p(V_1\cup V_2)          \\
H_p(V_1\cap V_2) \arrow[u] \arrow[r] & H_p(V_1) \arrow[u]
\end{tikzcd},
\end{center}
as a direct consequence of exactness of the corresponding Mayer-Vietoris sequence $$\cdots\longrightarrow H_p(V_1\cap V_2)\longrightarrow H_p(V_1)\oplus H_p(V_2) \longrightarrow H_p(V_1\cup V_2)\longrightarrow\cdots.$$

\begin{figure}
    \centering

\tikzset{every picture/.style={line width=0.75pt}} 

\begin{tikzpicture}[x=0.75pt,y=0.75pt,yscale=-1,xscale=1]

\draw   (169.5,49.14) -- (185.68,65.32) -- (169.5,81.5) -- (153.32,65.32) -- cycle ;
\draw [line width=1.5]    (142.32,65.32) -- (153.32,65.32) ;
\draw    (185.68,65.32) -- (196.68,65.32) ;
\draw   (259.5,49.14) -- (275.68,65.32) -- (259.5,81.5) -- (243.32,65.32) -- cycle ;
\draw [line width=1.5]    (232.32,65.32) -- (243.32,65.32) ;
\draw    (275.68,65.32) -- (286.68,65.32) ;
\draw   (169.5,119.14) -- (185.68,135.32) -- (169.5,151.5) -- (153.32,135.32) -- cycle ;
\draw [line width=1.5]    (142.32,135.32) -- (153.32,135.32) ;
\draw    (185.68,135.32) -- (196.68,135.32) ;
\draw   (259.5,119.14) -- (275.68,135.32) -- (259.5,151.5) -- (243.32,135.32) -- cycle ;
\draw [line width=1.5]    (232.32,135.32) -- (243.32,135.32) ;
\draw    (275.68,135.32) -- (286.68,135.32) ;
\draw   (169.5,219.14) -- (185.68,235.32) -- (169.5,251.5) -- (153.32,235.32) -- cycle ;
\draw    (142.32,235.32) -- (153.32,235.32) ;
\draw [line width=1.5]    (185.68,235.32) -- (196.68,235.32) ;
\draw   (259.5,219.14) -- (275.68,235.32) -- (259.5,251.5) -- (243.32,235.32) -- cycle ;
\draw    (232.32,235.32) -- (243.32,235.32) ;
\draw [line width=1.5]    (275.68,235.32) -- (286.68,235.32) ;
\draw   (169.5,289.14) -- (185.68,305.32) -- (169.5,321.5) -- (153.32,305.32) -- cycle ;
\draw    (142.32,305.32) -- (153.32,305.32) ;
\draw [line width=1.5]    (185.68,305.32) -- (196.68,305.32) ;
\draw   (259.5,289.14) -- (275.68,305.32) -- (259.5,321.5) -- (243.32,305.32) -- cycle ;
\draw    (232.32,305.32) -- (243.32,305.32) ;
\draw [line width=1.5]    (275.68,305.32) -- (286.68,305.32) ;
\draw   (169.5,379.14) -- (185.68,395.32) -- (169.5,411.5) -- (153.32,395.32) -- cycle ;
\draw    (142.32,395.32) -- (153.32,395.32) ;
\draw    (185.68,395.32) -- (196.68,395.32) ;
\draw   (259.5,379.14) -- (275.68,395.32) -- (259.5,411.5) -- (243.32,395.32) -- cycle ;
\draw    (232.32,395.32) -- (243.32,395.32) ;
\draw    (275.68,395.32) -- (286.68,395.32) ;
\draw   (169.5,459.14) -- (185.68,475.32) -- (169.5,491.5) -- (153.32,475.32) -- cycle ;
\draw [line width=1.5]    (142.32,475.32) -- (153.32,475.32) ;
\draw [line width=1.5]    (185.68,475.32) -- (196.68,475.32) ;
\draw   (259.5,459.14) -- (275.68,475.32) -- (259.5,491.5) -- (243.32,475.32) -- cycle ;
\draw [line width=1.5]    (232.32,475.32) -- (243.32,475.32) ;
\draw [line width=1.5]    (275.68,475.32) -- (286.68,475.32) ;
\draw  [draw opacity=0][fill={rgb, 255:red, 128; green, 181; blue, 243 }  ,fill opacity=1 ] (182.68,235.32) .. controls (182.68,233.67) and (184.02,232.32) .. (185.68,232.32) .. controls (187.33,232.32) and (188.68,233.67) .. (188.68,235.32) .. controls (188.68,236.98) and (187.33,238.32) .. (185.68,238.32) .. controls (184.02,238.32) and (182.68,236.98) .. (182.68,235.32) -- cycle ;
\draw  [draw opacity=0][fill={rgb, 255:red, 128; green, 181; blue, 243 }  ,fill opacity=1 ] (169.53,414.5) .. controls (167.91,414.48) and (166.6,413.15) .. (166.6,411.5) .. controls (166.6,409.85) and (167.91,408.52) .. (169.54,408.5) -- cycle ;
\draw  [draw opacity=0][fill={rgb, 255:red, 219; green, 125; blue, 238 }  ,fill opacity=1 ] (169.54,408.5) .. controls (171.16,408.52) and (172.47,409.85) .. (172.47,411.5) .. controls (172.47,413.15) and (171.16,414.48) .. (169.53,414.5) -- cycle ;
\draw  [draw opacity=0][fill={rgb, 255:red, 128; green, 181; blue, 243 }  ,fill opacity=1 ] (259,382) .. controls (257.37,381.98) and (256.06,380.65) .. (256.06,379) .. controls (256.06,377.35) and (257.37,376.02) .. (259,376) -- cycle ;
\draw  [draw opacity=0][fill={rgb, 255:red, 219; green, 125; blue, 238 }  ,fill opacity=1 ] (259,376) .. controls (260.63,376.02) and (261.94,377.35) .. (261.94,379) .. controls (261.94,380.65) and (260.63,381.98) .. (259,382) -- cycle ;
\draw  [draw opacity=0][fill={rgb, 255:red, 219; green, 125; blue, 238 }  ,fill opacity=1 ] (150.32,65.32) .. controls (150.32,63.67) and (151.67,62.32) .. (153.32,62.32) .. controls (154.98,62.32) and (156.32,63.67) .. (156.32,65.32) .. controls (156.32,66.98) and (154.98,68.32) .. (153.32,68.32) .. controls (151.67,68.32) and (150.32,66.98) .. (150.32,65.32) -- cycle ;
\draw  [draw opacity=0][fill={rgb, 255:red, 219; green, 125; blue, 238 }  ,fill opacity=1 ] (240.32,135.32) .. controls (240.32,133.67) and (241.67,132.32) .. (243.32,132.32) .. controls (244.98,132.32) and (246.32,133.67) .. (246.32,135.32) .. controls (246.32,136.98) and (244.98,138.32) .. (243.32,138.32) .. controls (241.67,138.32) and (240.32,136.98) .. (240.32,135.32) -- cycle ;
\draw  [draw opacity=0][fill={rgb, 255:red, 128; green, 181; blue, 243 }  ,fill opacity=1 ] (272.68,305.32) .. controls (272.68,303.67) and (274.02,302.32) .. (275.68,302.32) .. controls (277.33,302.32) and (278.68,303.67) .. (278.68,305.32) .. controls (278.68,306.98) and (277.33,308.32) .. (275.68,308.32) .. controls (274.02,308.32) and (272.68,306.98) .. (272.68,305.32) -- cycle ;
\draw [line width=1.5]    (243.32,65.32) -- (259.5,49.14) ;
\draw  [draw opacity=0][fill={rgb, 255:red, 219; green, 125; blue, 238 }  ,fill opacity=1 ] (256.5,49.14) .. controls (256.5,47.49) and (257.84,46.14) .. (259.5,46.14) .. controls (261.16,46.14) and (262.5,47.49) .. (262.5,49.14) .. controls (262.5,50.8) and (261.16,52.14) .. (259.5,52.14) .. controls (257.84,52.14) and (256.5,50.8) .. (256.5,49.14) -- cycle ;
\draw [line width=1.5]    (169.5,321.5) -- (185.68,305.32) ;
\draw  [draw opacity=0][fill={rgb, 255:red, 128; green, 181; blue, 243 }  ,fill opacity=1 ] (166.5,321.5) .. controls (166.5,319.84) and (167.84,318.5) .. (169.5,318.5) .. controls (171.16,318.5) and (172.5,319.84) .. (172.5,321.5) .. controls (172.5,323.16) and (171.16,324.5) .. (169.5,324.5) .. controls (167.84,324.5) and (166.5,323.16) .. (166.5,321.5) -- cycle ;
\draw [line width=1.5]    (275.68,235.32) -- (259.5,219.14) ;
\draw  [draw opacity=0][fill={rgb, 255:red, 128; green, 181; blue, 243 }  ,fill opacity=1 ] (256.5,219.14) .. controls (256.5,217.49) and (257.84,216.14) .. (259.5,216.14) .. controls (261.16,216.14) and (262.5,217.49) .. (262.5,219.14) .. controls (262.5,220.8) and (261.16,222.14) .. (259.5,222.14) .. controls (257.84,222.14) and (256.5,220.8) .. (256.5,219.14) -- cycle ;
\draw [line width=1.5]    (169.5,151.5) -- (153.32,135.32) ;
\draw  [draw opacity=0][fill={rgb, 255:red, 219; green, 125; blue, 238 }  ,fill opacity=1 ] (166.5,151.5) .. controls (166.5,149.84) and (167.84,148.5) .. (169.5,148.5) .. controls (171.16,148.5) and (172.5,149.84) .. (172.5,151.5) .. controls (172.5,153.16) and (171.16,154.5) .. (169.5,154.5) .. controls (167.84,154.5) and (166.5,153.16) .. (166.5,151.5) -- cycle ;
\draw [line width=1.5]    (169.5,491.5) -- (185.68,475.32) ;
\draw [line width=1.5]    (243.32,475.32) -- (259.5,459.14) ;
\draw [line width=1.5]    (169.5,491.5) -- (153.32,475.32) ;
\draw [line width=1.5]    (275.68,475.32) -- (259.5,459.14) ;

\draw (199,55.4) node [anchor=north west][inner sep=0.75pt]    {$\leftrightarrow $};
\draw (199,125.4) node [anchor=north west][inner sep=0.75pt]    {$\leftrightarrow $};
\draw (199,225.4) node [anchor=north west][inner sep=0.75pt]    {$\leftrightarrow $};
\draw (199,295.4) node [anchor=north west][inner sep=0.75pt]    {$\leftrightarrow $};
\draw (199,465.4) node [anchor=north west][inner sep=0.75pt]    {$\leftrightarrow $};
\draw (199,385.4) node [anchor=north west][inner sep=0.75pt]    {$\leftrightarrow $};
\draw (16,12.4) node [anchor=north west][inner sep=0.75pt]    {$\text{For }\mathbb{B} \in \{\mathbb{X}_{1} ,...,\mathbb{X}_{k-2} ,\mathbb{U}\} :$};
\draw (16,182.4) node [anchor=north west][inner sep=0.75pt]    {$\text{For }\mathbb{D} \in \{\mathbb{V,X}_{k+2} ,...,\mathbb{X}_{n}\} :\ $};
\draw (16,341.4) node [anchor=north west][inner sep=0.75pt]    {$\text{Exceptional case} :$};
\draw (16,434.4) node [anchor=north west][inner sep=0.75pt]    {$\text{Otherwise :}$};
\draw (86,56.4) node [anchor=north west][inner sep=0.75pt]    {$[\mathbb{B} ,\mathbb{U}]$};
\draw (294,56.4) node [anchor=north west][inner sep=0.75pt]    {$[\mathbb{B} ,\mathbb{U} \cup \mathbb{V}]$};
\draw (56,126.4) node [anchor=north west][inner sep=0.75pt]    {$[\mathbb{B} ,\mathbb{U} \cap \mathbb{V}]$};
\draw (294,126.4) node [anchor=north west][inner sep=0.75pt]    {$[\mathbb{B} ,\mathbb{U}]$};
\draw (86,226.4) node [anchor=north west][inner sep=0.75pt]    {$[\mathbb{V} ,\mathbb{D}]$};
\draw (56,296.4) node [anchor=north west][inner sep=0.75pt]    {$[\mathbb{U} \cap \mathbb{V} ,\mathbb{D}]$};
\draw (294,226.4) node [anchor=north west][inner sep=0.75pt]    {$[\mathbb{U} \cup \mathbb{V} ,\mathbb{D}]$};
\draw (296,296.4) node [anchor=north west][inner sep=0.75pt]    {$[\mathbb{V} ,\mathbb{D}]$};
\draw (21,386.4) node [anchor=north west][inner sep=0.75pt]    {$[\mathbb{U} \cap \mathbb{V} ,\mathbb{U} \cap \mathbb{V}]$};
\draw (289,386.4) node [anchor=north west][inner sep=0.75pt]    {$[\mathbb{U\cup V} ,\mathbb{U\cup V}]$};

\end{tikzpicture}

    \caption{The (strong) diamond principle illustrated. The Mayer-Vietoris diamond involved is $\U\xhookleftarrow{}\U\cap\V\xhookrightarrow{}\V$, used to go from the zigzag $\mathcal{X}_\cap$ to $\mathcal{X}_\cup$. }
    \label{fig:diamond_moves}
\end{figure}

\begin{theorem}[Strong Diamond Principle](\autocite[Theorem~ 5.9]{carlsson2008zigzag}) 
Let $\V^+=H_*(\mathcal{X}_{\cup})$ and $\V^-=H_*(\mathcal{X}_{\cap})$, where $\mathcal{X}_{\cup}$ and $\mathcal{X}_{\cap}$ are zigzag diagrams 
$$\mathcal{X}_{\cup}:\X_0\longleftrightarrow \cdots\longleftrightarrow \X_{k-1}\xhookrightarrow{} \X_{k-1}\cup \X_{k+1}\xhookleftarrow{} \X_{k+1}\longleftrightarrow\cdots\longleftrightarrow \X_n,$$
$$\mathcal{X}_{\cap}:\X_0\longleftrightarrow \cdots\longleftrightarrow \X_{k-1}\xhookleftarrow{} \X_{k-1}\cap \X_{k+1} \xhookrightarrow{} \X_{k+1}\longleftrightarrow\cdots\longleftrightarrow \X_n.$$
There is a bijection between intervals of $\B({\V^+})$ and $\B({\V^-})$ given by the table below.
\begin{center}
\begin{tabular}{|c c c|} 
 \hline
Condition & $\B({\V^+})$ & $\B({\V^-})$\\ [0.5ex] 
 \hline
 $b\leq k-1$ & $[b,k]$ & $[b,k-1]$ \\ 
\hline
 $d\geq k+1$ & $[k,d]$ & $[k+1,d]$\\
 \hline
 $b\geq k\lor d\leq k$ & $[b, d]$ & $[b, d]$ \\
 \hline
  & $[k, k]^{p+1}$ & $[k, k]^{p}$\\
 \hline
\end{tabular}
\end{center}
where the superscripts $p+1$ and $p$ in the last row indicate a $\pm 1$ shift of homological dimension, \textit{i.e.} $[k,k]\in \B({\V_{p+1}^+})$ is matched with $[k,k]\in \B({\V_p^-})$, where we write  $$\V_{p+1}^+=H_{p+1}(\mathcal{X}_{\cup})\text{ and } \V_{p}^-=H_{p}(\mathcal{X}_{\cap}).$$
\end{theorem}

\begin{proof}
Let $\X_\cap=\X_{k-1}\cap \X_{k+1}$ and $\X_\cup=\X_{k-1}\cup \X_{k+1}$. For any $p\in\N$, the square \begin{center}
\begin{tikzcd}
H_p(\X_{k+1}) \arrow[r]          & H_p(\X_\cup)            \\
H_p(\X_\cap) \arrow[r] \arrow[u] & H_p(\X_{k-1}) \arrow[u]
\end{tikzcd}    
\end{center}
is exact, and thus the Diamond Principle applies, providing the three first rows of the bijection table. It remains to prove that any interval $[k,k]\in \B({\V_{p+1}^+})$ is matched with $[k,k]\in \B({\V_p^-})$. This follows from the exactness of the Mayer-Vietoris sequence

\[ \xymatrix{ \cdots \ar[r]& H_{p+1}(\X_\cap) \ar[r]& H_{p+1}(\X_{k-1})\oplus H_{p+1}(\X_{k+1})\ar[r] & H_{p+1}(\X_\cup) \ar[dll]  \\
&   H_{p}(\X_\cap)  \ar[r]  & H_{p}(\X_{k-1})\oplus H_{p}(\X_{k+1}) \ar[r]& H_{p}(\X_\cup) \ar[r]& \cdots \\
}
\]

Indeed, the first isomorphism theorem, together with exactness, implies that 
$$\mathrm{Coker}(D_2)=H_{p+1}(\X_\cup)/\mathrm{Im}(D_2)= H_{p+1}(\X_\cup)/\mathrm{Ker}(\partial)\cong \mathrm{Im}(\partial)=\mathrm{Ker}(D_1),$$ where $\partial$ denotes the connecting homomorphism $H_{p+1}(\X_\cup)\rightarrow H_{p}(\X_\cap)$, and with $$D_1:H_p(\X_\cap)\rightarrow H_{p}(\X_{k-1})\oplus H_{p}(\X_{k+1}),$$ 
$$D_2:H_{p+1}(\X_{k-1})\oplus H_{p+1}(\X_{k+1})\rightarrow H_{p+1}(\X_\cup).$$ This proves the last row of the matching table, as $\mathrm{Coker}(D_2)$ is precisely spanned by intervals $[k,k]\in\B({\V_{p+1}^+})$ and $\mathrm{Ker}(D_1)$ is spanned by intervals $[k,k]\in\B({\V_{p}^-})$.
\end{proof}

In the following, a pair of subspaces of $\mathbb{X}$ is a tuple $(\mathbb{A}, \mathbb{B})$ where $\mathbb{B} \subset \mathbb{A} \subset \mathbb{X}$. For two pairs $(\mathbb{A}, \mathbb{B})$, $(\mathbb{C}, \mathbb{D})$, such that $\mathbb{B} \subset \mathbb{D}$ and $\mathbb{A} \subset \mathbb{C}$, we will denote $(\mathbb{A}, \mathbb{B}) \longrightarrow (\mathbb{C}, \mathbb{D}) $ to designate the inclusion of pairs. A pair of type $(\mathbb{A}, \emptyset)$ will simply be denoted $\mathbb{A}$. To a pair $(\mathbb{A}, \mathbb{B})$ of subspaces of $\mathbb{X}$, we can associate their relative $p$-th homology groups with coefficients in $K$, denoted by $H_p(\mathbb{A},\mathbb{B})$. It is functorial with respect to inclusion of pairs. 

\begin{definition}[Mayer-Vietoris Pyramid] \label{d:MVpyramid}
The \emph{pyramid associated to $(\X,f)$} is defined as the diagram of pairs of subspaces of $\X$ built inductively as follows. One starts by defining the southern edge as the levelsets zigzag sequence $$(\emptyset, \emptyset) \longleftarrow (\mathbb{X}_0^0,  \emptyset) \longrightarrow (\mathbb{X}_0^1,  \emptyset)\longleftarrow \cdots\longrightarrow (\mathbb{X}_{n-1}^n,  \emptyset)\longleftarrow (\mathbb{X}_n^n,  \emptyset)\longrightarrow (\emptyset, \emptyset).$$
\noindent Then, one can build the pyramid from the southern edge by creating relative Mayer-Vietoris diamonds (the order in which one completes the diamonds has no importance). More precisely, consider any subspace $\Y$ as a pair $(\Y,\emptyset)$ and, given a $3$-tuple of the form $$(\A,\B)\longleftarrow (\A\cap\C,\B\cap\D)\longrightarrow (\C,\D),$$
complete it into a square of the form \begin{center}
\begin{tikzcd}
                   & {(\A\cup \C,\B\cup \D)}                       &                    \\
{(\A,\B)} \arrow[ru] &                                           & {(\C,\D)} \arrow[lu] \\
                   & {(\A\cap \C,\B\cap \D)} \arrow[lu] \arrow[ru] &                   
\end{tikzcd}.
\end{center}

\noindent \textbf{Detailed method : building the pyramid by completing diagrams.} 
\begin{enumerate}
    \item A diagram $\A\longleftarrow \A\cap\C\longrightarrow \C$ with $\A,\C\neq\emptyset$ is completed with $\A\cup\C$.
    \item A diagram $\emptyset\longleftarrow \E\longrightarrow \C$ is completed with $(\C,\E)$ (south-west corner).
    \item A diagram $\A\longleftarrow \E\longrightarrow \emptyset$ is completed with $(\A,\E)$ (south-east corner).
    \item A diagram $\emptyset\longleftarrow (\E,\F)\longrightarrow (\C,\D)$ is completed with $(\C,\E)$ (western edge).
    \item A diagram $(\A,\B)\longleftarrow (\E,\F)\longrightarrow \emptyset$ is completed with $(\A,\E)$ (eastern edge).
    \item A diagram $\A\longleftarrow \E\longrightarrow (\C,\D)$ is completed with $(\A,\D)$.
    \item A diagram $(\A,\B)\longleftarrow \E\longrightarrow \C$ is completed with $(\C,\B)$.
    \item A diagram $(\A,\B)\longleftarrow (\A\cap\C,\B\cap\D)\longrightarrow (\C,\D)$ is completed with $(\A\cup\C,\B\cup\D)$.
    \item The diagram $(\A,\B)\longleftarrow \A\longrightarrow (\A,\D)$ is completed with $(\A,\B\cup\D)$ (center).
    
\end{enumerate}

\noindent Whenever creating a diamond gives rise to a pair of the form $(\Y,\Y)$, stop considering $3$-tuples containing this pair. This process defines the northern edge of the pyramid. Moreover, pairs of this form are considered as the empty subspace $\emptyset$. 


\end{definition}

\begin{figure}
    \centering

\tikzset{every picture/.style={line width=0.75pt}} 

\begin{tikzpicture}[x=0.60pt,y=0.60pt,yscale=-1,xscale=1]

\draw    (422.41,278.22) -- (455.71,256.3) ;
\draw [shift={(457.38,255.2)}, rotate = 146.65] [color={rgb, 255:red, 0; green, 0; blue, 0 }  ][line width=0.75]    (10.93,-3.29) .. controls (6.95,-1.4) and (3.31,-0.3) .. (0,0) .. controls (3.31,0.3) and (6.95,1.4) .. (10.93,3.29)   ;
\draw    (346.83,325.89) -- (385.76,300.68) ;
\draw [shift={(387.44,299.59)}, rotate = 147.07] [color={rgb, 255:red, 0; green, 0; blue, 0 }  ][line width=0.75]    (10.93,-3.29) .. controls (6.95,-1.4) and (3.31,-0.3) .. (0,0) .. controls (3.31,0.3) and (6.95,1.4) .. (10.93,3.29)   ;
\draw    (306.43,328.36) -- (261.82,299.85) ;
\draw [shift={(260.14,298.77)}, rotate = 32.59] [color={rgb, 255:red, 0; green, 0; blue, 0 }  ][line width=0.75]    (10.93,-3.29) .. controls (6.95,-1.4) and (3.31,-0.3) .. (0,0) .. controls (3.31,0.3) and (6.95,1.4) .. (10.93,3.29)   ;
\draw    (469.82,328.36) -- (425.21,299.85) ;
\draw [shift={(423.53,298.77)}, rotate = 32.59] [color={rgb, 255:red, 0; green, 0; blue, 0 }  ][line width=0.75]    (10.93,-3.29) .. controls (6.95,-1.4) and (3.31,-0.3) .. (0,0) .. controls (3.31,0.3) and (6.95,1.4) .. (10.93,3.29)   ;
\draw    (143.05,328.36) -- (112.04,308.08) ;
\draw [shift={(110.37,306.99)}, rotate = 33.19] [color={rgb, 255:red, 0; green, 0; blue, 0 }  ][line width=0.75]    (10.93,-3.29) .. controls (6.95,-1.4) and (3.31,-0.3) .. (0,0) .. controls (3.31,0.3) and (6.95,1.4) .. (10.93,3.29)   ;
\draw    (103.16,173.83) -- (125.21,158.53) ;
\draw [shift={(126.85,157.39)}, rotate = 145.24] [color={rgb, 255:red, 0; green, 0; blue, 0 }  ][line width=0.75]    (10.93,-3.29) .. controls (6.95,-1.4) and (3.31,-0.3) .. (0,0) .. controls (3.31,0.3) and (6.95,1.4) .. (10.93,3.29)   ;
\draw    (272.38,173.83) -- (294.42,158.53) ;
\draw [shift={(296.07,157.39)}, rotate = 145.24] [color={rgb, 255:red, 0; green, 0; blue, 0 }  ][line width=0.75]    (10.93,-3.29) .. controls (6.95,-1.4) and (3.31,-0.3) .. (0,0) .. controls (3.31,0.3) and (6.95,1.4) .. (10.93,3.29)   ;
\draw [color={rgb, 255:red, 74; green, 144; blue, 226 }  ,draw opacity=1 ]   (441.59,173.83) -- (463.64,158.53) ;
\draw [shift={(465.28,157.39)}, rotate = 145.24] [color={rgb, 255:red, 74; green, 144; blue, 226 }  ,draw opacity=1 ][line width=0.75]    (10.93,-3.29) .. controls (6.95,-1.4) and (3.31,-0.3) .. (0,0) .. controls (3.31,0.3) and (6.95,1.4) .. (10.93,3.29)   ;
\draw [color={rgb, 255:red, 74; green, 144; blue, 226 }  ,draw opacity=1 ]   (520.56,124.51) -- (542.6,109.21) ;
\draw [shift={(544.25,108.07)}, rotate = 145.24] [color={rgb, 255:red, 74; green, 144; blue, 226 }  ,draw opacity=1 ][line width=0.75]    (10.93,-3.29) .. controls (6.95,-1.4) and (3.31,-0.3) .. (0,0) .. controls (3.31,0.3) and (6.95,1.4) .. (10.93,3.29)   ;
\draw    (434.31,75.19) -- (456.36,59.89) ;
\draw [shift={(458,58.75)}, rotate = 145.24] [color={rgb, 255:red, 0; green, 0; blue, 0 }  ][line width=0.75]    (10.93,-3.29) .. controls (6.95,-1.4) and (3.31,-0.3) .. (0,0) .. controls (3.31,0.3) and (6.95,1.4) .. (10.93,3.29)   ;
\draw    (272.38,75.19) -- (294.42,59.89) ;
\draw [shift={(296.07,58.75)}, rotate = 145.24] [color={rgb, 255:red, 0; green, 0; blue, 0 }  ][line width=0.75]    (10.93,-3.29) .. controls (6.95,-1.4) and (3.31,-0.3) .. (0,0) .. controls (3.31,0.3) and (6.95,1.4) .. (10.93,3.29)   ;
\draw    (103.16,75.19) -- (125.21,59.89) ;
\draw [shift={(126.85,58.75)}, rotate = 145.24] [color={rgb, 255:red, 0; green, 0; blue, 0 }  ][line width=0.75]    (10.93,-3.29) .. controls (6.95,-1.4) and (3.31,-0.3) .. (0,0) .. controls (3.31,0.3) and (6.95,1.4) .. (10.93,3.29)   ;
\draw    (215.97,173.83) -- (196.13,158.6) ;
\draw [shift={(194.54,157.39)}, rotate = 37.49] [color={rgb, 255:red, 0; green, 0; blue, 0 }  ][line width=0.75]    (10.93,-3.29) .. controls (6.95,-1.4) and (3.31,-0.3) .. (0,0) .. controls (3.31,0.3) and (6.95,1.4) .. (10.93,3.29)   ;
\draw    (385.19,173.83) -- (365.34,158.6) ;
\draw [shift={(363.75,157.39)}, rotate = 37.49] [color={rgb, 255:red, 0; green, 0; blue, 0 }  ][line width=0.75]    (10.93,-3.29) .. controls (6.95,-1.4) and (3.31,-0.3) .. (0,0) .. controls (3.31,0.3) and (6.95,1.4) .. (10.93,3.29)   ;
\draw    (543.12,173.83) -- (523.27,158.6) ;
\draw [shift={(521.68,157.39)}, rotate = 37.49] [color={rgb, 255:red, 0; green, 0; blue, 0 }  ][line width=0.75]    (10.93,-3.29) .. controls (6.95,-1.4) and (3.31,-0.3) .. (0,0) .. controls (3.31,0.3) and (6.95,1.4) .. (10.93,3.29)   ;
\draw    (464.15,124.51) -- (444.31,109.28) ;
\draw [shift={(442.72,108.07)}, rotate = 37.49] [color={rgb, 255:red, 0; green, 0; blue, 0 }  ][line width=0.75]    (10.93,-3.29) .. controls (6.95,-1.4) and (3.31,-0.3) .. (0,0) .. controls (3.31,0.3) and (6.95,1.4) .. (10.93,3.29)   ;
\draw    (303.35,124.51) -- (283.5,109.28) ;
\draw [shift={(281.91,108.07)}, rotate = 37.49] [color={rgb, 255:red, 0; green, 0; blue, 0 }  ][line width=0.75]    (10.93,-3.29) .. controls (6.95,-1.4) and (3.31,-0.3) .. (0,0) .. controls (3.31,0.3) and (6.95,1.4) .. (10.93,3.29)   ;
\draw    (137.01,124.51) -- (117.16,109.28) ;
\draw [shift={(115.57,108.07)}, rotate = 37.49] [color={rgb, 255:red, 0; green, 0; blue, 0 }  ][line width=0.75]    (10.93,-3.29) .. controls (6.95,-1.4) and (3.31,-0.3) .. (0,0) .. controls (3.31,0.3) and (6.95,1.4) .. (10.93,3.29)   ;
\draw    (362.62,124.51) -- (384.67,109.21) ;
\draw [shift={(386.31,108.07)}, rotate = 145.24] [color={rgb, 255:red, 0; green, 0; blue, 0 }  ][line width=0.75]    (10.93,-3.29) .. controls (6.95,-1.4) and (3.31,-0.3) .. (0,0) .. controls (3.31,0.3) and (6.95,1.4) .. (10.93,3.29)   ;
\draw    (193.41,124.51) -- (215.46,109.21) ;
\draw [shift={(217.1,108.07)}, rotate = 145.24] [color={rgb, 255:red, 0; green, 0; blue, 0 }  ][line width=0.75]    (10.93,-3.29) .. controls (6.95,-1.4) and (3.31,-0.3) .. (0,0) .. controls (3.31,0.3) and (6.95,1.4) .. (10.93,3.29)   ;
\draw    (217.1,75.19) -- (197.25,59.96) ;
\draw [shift={(195.67,58.75)}, rotate = 37.49] [color={rgb, 255:red, 0; green, 0; blue, 0 }  ][line width=0.75]    (10.93,-3.29) .. controls (6.95,-1.4) and (3.31,-0.3) .. (0,0) .. controls (3.31,0.3) and (6.95,1.4) .. (10.93,3.29)   ;
\draw    (387.44,75.19) -- (367.6,59.96) ;
\draw [shift={(366.01,58.75)}, rotate = 37.49] [color={rgb, 255:red, 0; green, 0; blue, 0 }  ][line width=0.75]    (10.93,-3.29) .. controls (6.95,-1.4) and (3.31,-0.3) .. (0,0) .. controls (3.31,0.3) and (6.95,1.4) .. (10.93,3.29)   ;
\draw    (552.53,75.19) -- (532.68,59.96) ;
\draw [shift={(531.09,58.75)}, rotate = 37.49] [color={rgb, 255:red, 0; green, 0; blue, 0 }  ][line width=0.75]    (10.93,-3.29) .. controls (6.95,-1.4) and (3.31,-0.3) .. (0,0) .. controls (3.31,0.3) and (6.95,1.4) .. (10.93,3.29)   ;
\draw    (193.41,223.15) -- (215.46,207.85) ;
\draw [shift={(217.1,206.71)}, rotate = 145.24] [color={rgb, 255:red, 0; green, 0; blue, 0 }  ][line width=0.75]    (10.93,-3.29) .. controls (6.95,-1.4) and (3.31,-0.3) .. (0,0) .. controls (3.31,0.3) and (6.95,1.4) .. (10.93,3.29)   ;
\draw    (137.01,223.15) -- (117.16,207.92) ;
\draw [shift={(115.57,206.71)}, rotate = 37.49] [color={rgb, 255:red, 0; green, 0; blue, 0 }  ][line width=0.75]    (10.93,-3.29) .. controls (6.95,-1.4) and (3.31,-0.3) .. (0,0) .. controls (3.31,0.3) and (6.95,1.4) .. (10.93,3.29)   ;
\draw    (103.16,271.64) -- (125.21,256.34) ;
\draw [shift={(126.85,255.2)}, rotate = 145.24] [color={rgb, 255:red, 0; green, 0; blue, 0 }  ][line width=0.75]    (10.93,-3.29) .. controls (6.95,-1.4) and (3.31,-0.3) .. (0,0) .. controls (3.31,0.3) and (6.95,1.4) .. (10.93,3.29)   ;
\draw    (223.87,276.57) -- (191.7,255.48) ;
\draw [shift={(190.03,254.38)}, rotate = 33.26] [color={rgb, 255:red, 0; green, 0; blue, 0 }  ][line width=0.75]    (10.93,-3.29) .. controls (6.95,-1.4) and (3.31,-0.3) .. (0,0) .. controls (3.31,0.3) and (6.95,1.4) .. (10.93,3.29)   ;
\draw    (314.12,227.26) -- (281.95,206.16) ;
\draw [shift={(280.27,205.06)}, rotate = 33.26] [color={rgb, 255:red, 0; green, 0; blue, 0 }  ][line width=0.75]    (10.93,-3.29) .. controls (6.95,-1.4) and (3.31,-0.3) .. (0,0) .. controls (3.31,0.3) and (6.95,1.4) .. (10.93,3.29)   ;
\draw    (381.8,276.57) -- (349.63,255.48) ;
\draw [shift={(347.96,254.38)}, rotate = 33.26] [color={rgb, 255:red, 0; green, 0; blue, 0 }  ][line width=0.75]    (10.93,-3.29) .. controls (6.95,-1.4) and (3.31,-0.3) .. (0,0) .. controls (3.31,0.3) and (6.95,1.4) .. (10.93,3.29)   ;
\draw    (520.56,223.15) -- (542.6,207.85) ;
\draw [shift={(544.25,206.71)}, rotate = 145.24] [color={rgb, 255:red, 0; green, 0; blue, 0 }  ][line width=0.75]    (10.93,-3.29) .. controls (6.95,-1.4) and (3.31,-0.3) .. (0,0) .. controls (3.31,0.3) and (6.95,1.4) .. (10.93,3.29)   ;
\draw    (464.15,223.15) -- (444.31,207.92) ;
\draw [shift={(442.72,206.71)}, rotate = 37.49] [color={rgb, 255:red, 0; green, 0; blue, 0 }  ][line width=0.75]    (10.93,-3.29) .. controls (6.95,-1.4) and (3.31,-0.3) .. (0,0) .. controls (3.31,0.3) and (6.95,1.4) .. (10.93,3.29)   ;
\draw [color={rgb, 255:red, 74; green, 144; blue, 226 }  ,draw opacity=1 ]   (343.45,228.9) -- (376.75,206.98) ;
\draw [shift={(378.42,205.88)}, rotate = 146.65] [color={rgb, 255:red, 74; green, 144; blue, 226 }  ,draw opacity=1 ][line width=0.75]    (10.93,-3.29) .. controls (6.95,-1.4) and (3.31,-0.3) .. (0,0) .. controls (3.31,0.3) and (6.95,1.4) .. (10.93,3.29)   ;
\draw [color={rgb, 255:red, 74; green, 144; blue, 226 }  ,draw opacity=1 ]   (184.13,325.89) -- (223.06,300.68) ;
\draw [shift={(224.74,299.59)}, rotate = 147.07] [color={rgb, 255:red, 74; green, 144; blue, 226 }  ,draw opacity=1 ][line width=0.75]    (10.93,-3.29) .. controls (6.95,-1.4) and (3.31,-0.3) .. (0,0) .. controls (3.31,0.3) and (6.95,1.4) .. (10.93,3.29)   ;
\draw    (509.28,328.36) -- (548.21,303.14) ;
\draw [shift={(549.89,302.06)}, rotate = 147.07] [color={rgb, 255:red, 0; green, 0; blue, 0 }  ][line width=0.75]    (10.93,-3.29) .. controls (6.95,-1.4) and (3.31,-0.3) .. (0,0) .. controls (3.31,0.3) and (6.95,1.4) .. (10.93,3.29)   ;
\draw [color={rgb, 255:red, 74; green, 144; blue, 226 }  ,draw opacity=1 ]   (104.04,376.86) -- (142.97,351.64) ;
\draw [shift={(144.65,350.55)}, rotate = 147.07] [color={rgb, 255:red, 74; green, 144; blue, 226 }  ,draw opacity=1 ][line width=0.75]    (10.93,-3.29) .. controls (6.95,-1.4) and (3.31,-0.3) .. (0,0) .. controls (3.31,0.3) and (6.95,1.4) .. (10.93,3.29)   ;
\draw [color={rgb, 255:red, 144; green, 19; blue, 254 }  ,draw opacity=0.6 ]   (585.99,376.04) -- (624.92,350.82) ;
\draw [shift={(626.6,349.73)}, rotate = 147.07] [color={rgb, 255:red, 144; green, 19; blue, 254 }  ,draw opacity=0.6 ][line width=0.75]    (10.93,-3.29) .. controls (6.95,-1.4) and (3.31,-0.3) .. (0,0) .. controls (3.31,0.3) and (6.95,1.4) .. (10.93,3.29)   ;
\draw [color={rgb, 255:red, 144; green, 19; blue, 254 }  ,draw opacity=0.6 ]   (428.05,376.86) -- (466.99,351.64) ;
\draw [shift={(468.66,350.55)}, rotate = 147.07] [color={rgb, 255:red, 144; green, 19; blue, 254 }  ,draw opacity=0.6 ][line width=0.75]    (10.93,-3.29) .. controls (6.95,-1.4) and (3.31,-0.3) .. (0,0) .. controls (3.31,0.3) and (6.95,1.4) .. (10.93,3.29)   ;
\draw [color={rgb, 255:red, 144; green, 19; blue, 254 }  ,draw opacity=0.6 ]   (265.61,376.04) -- (304.54,350.82) ;
\draw [shift={(306.22,349.73)}, rotate = 147.07] [color={rgb, 255:red, 144; green, 19; blue, 254 }  ,draw opacity=0.6 ][line width=0.75]    (10.93,-3.29) .. controls (6.95,-1.4) and (3.31,-0.3) .. (0,0) .. controls (3.31,0.3) and (6.95,1.4) .. (10.93,3.29)   ;
\draw [color={rgb, 255:red, 74; green, 144; blue, 226 }  ,draw opacity=1 ]   (264.35,278.22) -- (303.29,253) ;
\draw [shift={(304.96,251.91)}, rotate = 147.07] [color={rgb, 255:red, 74; green, 144; blue, 226 }  ,draw opacity=1 ][line width=0.75]    (10.93,-3.29) .. controls (6.95,-1.4) and (3.31,-0.3) .. (0,0) .. controls (3.31,0.3) and (6.95,1.4) .. (10.93,3.29)   ;
\draw [color={rgb, 255:red, 144; green, 19; blue, 254 }  ,draw opacity=0.6 ]   (552.14,378.5) -- (513.21,353.29) ;
\draw [shift={(511.53,352.2)}, rotate = 32.93] [color={rgb, 255:red, 144; green, 19; blue, 254 }  ,draw opacity=0.6 ][line width=0.75]    (10.93,-3.29) .. controls (6.95,-1.4) and (3.31,-0.3) .. (0,0) .. controls (3.31,0.3) and (6.95,1.4) .. (10.93,3.29)   ;
\draw [color={rgb, 255:red, 144; green, 19; blue, 254 }  ,draw opacity=0.6 ]   (382.93,378.5) -- (344,353.29) ;
\draw [shift={(342.32,352.2)}, rotate = 32.93] [color={rgb, 255:red, 144; green, 19; blue, 254 }  ,draw opacity=0.6 ][line width=0.75]    (10.93,-3.29) .. controls (6.95,-1.4) and (3.31,-0.3) .. (0,0) .. controls (3.31,0.3) and (6.95,1.4) .. (10.93,3.29)   ;
\draw [color={rgb, 255:red, 144; green, 19; blue, 254 }  ,draw opacity=0.6 ]   (225,378.5) -- (186.07,353.29) ;
\draw [shift={(184.39,352.2)}, rotate = 32.93] [color={rgb, 255:red, 144; green, 19; blue, 254 }  ,draw opacity=0.6 ][line width=0.75]    (10.93,-3.29) .. controls (6.95,-1.4) and (3.31,-0.3) .. (0,0) .. controls (3.31,0.3) and (6.95,1.4) .. (10.93,3.29)   ;
\draw    (625.47,328.36) -- (586.54,303.14) ;
\draw [shift={(584.86,302.06)}, rotate = 32.93] [color={rgb, 255:red, 0; green, 0; blue, 0 }  ][line width=0.75]    (10.93,-3.29) .. controls (6.95,-1.4) and (3.31,-0.3) .. (0,0) .. controls (3.31,0.3) and (6.95,1.4) .. (10.93,3.29)   ;
\draw    (625.47,232.19) -- (586.54,206.97) ;
\draw [shift={(584.86,205.88)}, rotate = 32.93] [color={rgb, 255:red, 0; green, 0; blue, 0 }  ][line width=0.75]    (10.93,-3.29) .. controls (6.95,-1.4) and (3.31,-0.3) .. (0,0) .. controls (3.31,0.3) and (6.95,1.4) .. (10.93,3.29)   ;
\draw [color={rgb, 255:red, 144; green, 19; blue, 254 }  ,draw opacity=0.6 ]   (64.08,377.68) -- (33.08,357.4) ;
\draw [shift={(31.4,356.31)}, rotate = 33.19] [color={rgb, 255:red, 144; green, 19; blue, 254 }  ,draw opacity=0.6 ][line width=0.75]    (10.93,-3.29) .. controls (6.95,-1.4) and (3.31,-0.3) .. (0,0) .. controls (3.31,0.3) and (6.95,1.4) .. (10.93,3.29)   ;
\draw    (26.45,328.36) -- (56.38,308.11) ;
\draw [shift={(58.04,306.99)}, rotate = 145.92] [color={rgb, 255:red, 0; green, 0; blue, 0 }  ][line width=0.75]    (10.93,-3.29) .. controls (6.95,-1.4) and (3.31,-0.3) .. (0,0) .. controls (3.31,0.3) and (6.95,1.4) .. (10.93,3.29)   ;
\draw    (26.45,229.72) -- (56.38,209.47) ;
\draw [shift={(58.04,208.35)}, rotate = 145.92] [color={rgb, 255:red, 0; green, 0; blue, 0 }  ][line width=0.75]    (10.93,-3.29) .. controls (6.95,-1.4) and (3.31,-0.3) .. (0,0) .. controls (3.31,0.3) and (6.95,1.4) .. (10.93,3.29)   ;
\draw    (26.45,131.08) -- (56.38,110.83) ;
\draw [shift={(58.04,109.71)}, rotate = 145.92] [color={rgb, 255:red, 0; green, 0; blue, 0 }  ][line width=0.75]    (10.93,-3.29) .. controls (6.95,-1.4) and (3.31,-0.3) .. (0,0) .. controls (3.31,0.3) and (6.95,1.4) .. (10.93,3.29)   ;
\draw    (599.52,173.83) -- (621.57,158.53) ;
\draw [shift={(623.21,157.39)}, rotate = 145.24] [color={rgb, 255:red, 0; green, 0; blue, 0 }  ][line width=0.75]    (10.93,-3.29) .. controls (6.95,-1.4) and (3.31,-0.3) .. (0,0) .. controls (3.31,0.3) and (6.95,1.4) .. (10.93,3.29)   ;
\draw    (53.53,173) -- (30.4,158.45) ;
\draw [shift={(28.71,157.39)}, rotate = 32.18] [color={rgb, 255:red, 0; green, 0; blue, 0 }  ][line width=0.75]    (10.93,-3.29) .. controls (6.95,-1.4) and (3.31,-0.3) .. (0,0) .. controls (3.31,0.3) and (6.95,1.4) .. (10.93,3.29)   ;
\draw    (53.53,74.36) -- (30.4,59.81) ;
\draw [shift={(28.71,58.75)}, rotate = 32.18] [color={rgb, 255:red, 0; green, 0; blue, 0 }  ][line width=0.75]    (10.93,-3.29) .. controls (6.95,-1.4) and (3.31,-0.3) .. (0,0) .. controls (3.31,0.3) and (6.95,1.4) .. (10.93,3.29)   ;
\draw    (53.53,271.64) -- (30.4,257.09) ;
\draw [shift={(28.71,256.02)}, rotate = 32.18] [color={rgb, 255:red, 0; green, 0; blue, 0 }  ][line width=0.75]    (10.93,-3.29) .. controls (6.95,-1.4) and (3.31,-0.3) .. (0,0) .. controls (3.31,0.3) and (6.95,1.4) .. (10.93,3.29)   ;
\draw [color={rgb, 255:red, 74; green, 144; blue, 226 }  ,draw opacity=1 ]   (606.29,75.19) -- (629.39,58.19) ;
\draw [shift={(631,57)}, rotate = 143.65] [color={rgb, 255:red, 74; green, 144; blue, 226 }  ,draw opacity=1 ][line width=0.75]    (10.93,-3.29) .. controls (6.95,-1.4) and (3.31,-0.3) .. (0,0) .. controls (3.31,0.3) and (6.95,1.4) .. (10.93,3.29)   ;
\draw    (624.34,129.44) -- (597.8,111.65) ;
\draw [shift={(596.14,110.53)}, rotate = 33.84] [color={rgb, 255:red, 0; green, 0; blue, 0 }  ][line width=0.75]    (10.93,-3.29) .. controls (6.95,-1.4) and (3.31,-0.3) .. (0,0) .. controls (3.31,0.3) and (6.95,1.4) .. (10.93,3.29)   ;
\draw    (597.56,270.15) -- (619.6,254.85) ;
\draw [shift={(621.25,253.71)}, rotate = 145.24] [color={rgb, 255:red, 0; green, 0; blue, 0 }  ][line width=0.75]    (10.93,-3.29) .. controls (6.95,-1.4) and (3.31,-0.3) .. (0,0) .. controls (3.31,0.3) and (6.95,1.4) .. (10.93,3.29)   ;
\draw    (544.15,273.15) -- (524.31,257.92) ;
\draw [shift={(522.72,256.71)}, rotate = 37.49] [color={rgb, 255:red, 0; green, 0; blue, 0 }  ][line width=0.75]    (10.93,-3.29) .. controls (6.95,-1.4) and (3.31,-0.3) .. (0,0) .. controls (3.31,0.3) and (6.95,1.4) .. (10.93,3.29)   ;

\draw (78.18,375.12) node [anchor=north west][inner sep=0.75pt]  [color={rgb, 255:red, 74; green, 144; blue, 226 }  ,opacity=1 ]  {$\mathbb{X}_{0}^{0}$};
\draw (157.01,326.62) node [anchor=north west][inner sep=0.75pt]  [color={rgb, 255:red, 74; green, 144; blue, 226 }  ,opacity=1 ]  {$\mathbb{X}_{0}^{1}$};
\draw (235.24,376.12) node [anchor=north west][inner sep=0.75pt]  [color={rgb, 255:red, 144; green, 19; blue, 254 }  ,opacity=0.55 ]  {$\mathbb{X}_{1}^{1}$};
\draw (395.31,374.3) node [anchor=north west][inner sep=0.75pt]  [color={rgb, 255:red, 144; green, 19; blue, 254 }  ,opacity=0.55 ]  {$\mathbb{X}_{2}^{2}$};
\draw (318.4,326.62) node [anchor=north west][inner sep=0.75pt]  [color={rgb, 255:red, 144; green, 19; blue, 254 }  ,opacity=0.55 ]  {$\mathbb{X}_{1}^{2}$};
\draw (481.79,327.62) node [anchor=north west][inner sep=0.75pt]  [color={rgb, 255:red, 144; green, 19; blue, 254 }  ,opacity=0.55 ]  {$\mathbb{X}_{2}^{3}$};
\draw (558.69,374.48) node [anchor=north west][inner sep=0.75pt]  [color={rgb, 255:red, 144; green, 19; blue, 254 }  ,opacity=0.55 ]  {$\mathbb{X}_{3}^{3}$};
\draw (40.97,277.22) node [anchor=north west][inner sep=0.75pt]    {$\left(\mathbb{X}_{0}^{1} ,\mathbb{X}_{0}^{0}\right)$};
\draw (236.3,277.3) node [anchor=north west][inner sep=0.75pt]  [color={rgb, 255:red, 74; green, 144; blue, 226 }  ,opacity=1 ]  {$\mathbb{X}_{0}^{2}$};
\draw (395.44,278.3) node [anchor=north west][inner sep=0.75pt]    {$\mathbb{X}_{1}^{3}$};
\draw (318.4,227.99) node [anchor=north west][inner sep=0.75pt]  [color={rgb, 255:red, 74; green, 144; blue, 226 }  ,opacity=1 ]  {$\mathbb{X}_{0}^{3}$};
\draw (40.97,177.58) node [anchor=north west][inner sep=0.75pt]    {$\left(\mathbb{X}_{0}^{2} ,\mathbb{X}_{0}^{1}\right)$};
\draw (40.97,78.94) node [anchor=north west][inner sep=0.75pt]    {$\left(\mathbb{X}_{0}^{3} ,\mathbb{X}_{0}^{2}\right)$};
\draw (130.47,129.26) node [anchor=north west][inner sep=0.75pt]    {$\left(\mathbb{X}_{0}^{3} ,\mathbb{X}_{0}^{1}\right)$};
\draw (211.16,78.11) node [anchor=north west][inner sep=0.75pt]    {$\left(\mathbb{X}_{0}^{3} ,\mathbb{X}_{1}^{3}\right)$};
\draw (129.47,226.07) node [anchor=north west][inner sep=0.75pt]    {$\left(\mathbb{X}_{0}^{2} ,\mathbb{X}_{0}^{0}\right)$};
\draw (211.16,177.58) node [anchor=north west][inner sep=0.75pt]    {$\left(\mathbb{X}_{0}^{3} ,\mathbb{X}_{0}^{0}\right)$};
\draw (374.55,177.58) node [anchor=north west][inner sep=0.75pt]  [color={rgb, 255:red, 74; green, 144; blue, 226 }  ,opacity=1 ]  {$\left(\mathbb{X}_{0}^{3} ,\mathbb{X}_{3}^{3}\right)$};
\draw (537.94,176.58) node [anchor=north west][inner sep=0.75pt]    {$\left(\mathbb{X}_{1}^{3} ,\mathbb{X}_{2}^{3}\right)$};
\draw (544.94,78.11) node [anchor=north west][inner sep=0.75pt]  [color={rgb, 255:red, 74; green, 144; blue, 226 }  ,opacity=1 ]  {$\left(\mathbb{X}_{0}^{3} ,\mathbb{X}_{1}^{3}\right)$};
\draw (129.47,29.62) node [anchor=north west][inner sep=0.75pt]    {$\left(\mathbb{X}_{0}^{3} ,_{2}^{3}\mathbb{X}\right)$};
\draw (298.47,28.79) node [anchor=north west][inner sep=0.75pt]    {$\left(\mathbb{X}_{0}^{3} ,_{1}^{2}\mathbb{X}\right)$};
\draw (458.91,28.79) node [anchor=north west][inner sep=0.75pt]    {$\left(\mathbb{X}_{0}^{3} ,_{0}^{1}\mathbb{X}\right)$};
\draw (376.88,78.94) node [anchor=north west][inner sep=0.75pt]    {$\left(\mathbb{X}_{0}^{3} ,_{0}^{2}\mathbb{X}\right)$};
\draw (460.91,128.26) node [anchor=north west][inner sep=0.75pt]  [color={rgb, 255:red, 74; green, 144; blue, 226 }  ,opacity=1 ]  {$\left(\mathbb{X}_{0}^{3} ,\mathbb{X}_{2}^{3}\right)$};
\draw (296.47,127.43) node [anchor=north west][inner sep=0.75pt]    {$\left(\mathbb{X}_{0}^{3} ,_{0}^{3}\mathbb{X}\right)$};
\draw (461.13,226.07) node [anchor=north west][inner sep=0.75pt]    {$\left(\mathbb{X}_{1}^{3} ,\mathbb{X}_{3}^{3}\right)$};
\draw (10.43,133.35) node [anchor=north west][inner sep=0.75pt]    {$\emptyset $};
\draw (10.43,34.71) node [anchor=north west][inner sep=0.75pt]    {$\emptyset $};
\draw (10.43,231.99) node [anchor=north west][inner sep=0.75pt]    {$\emptyset $};
\draw (11.43,330.62) node [anchor=north west][inner sep=0.75pt]  [color={rgb, 255:red, 144; green, 19; blue, 254 }  ,opacity=0.55 ]  {$\emptyset $};
\draw (633.01,330.62) node [anchor=north west][inner sep=0.75pt]  [color={rgb, 255:red, 144; green, 19; blue, 254 }  ,opacity=0.55 ]  {$\emptyset $};
\draw (632.01,231.99) node [anchor=north west][inner sep=0.75pt]    {$\emptyset $};
\draw (632.01,132.52) node [anchor=north west][inner sep=0.75pt]    {$\emptyset $};
\draw (631.01,34.24) node [anchor=north west][inner sep=0.75pt]  [color={rgb, 255:red, 74; green, 144; blue, 226 }  ,opacity=1 ]  {$\emptyset $};
\draw (531.97,274.22) node [anchor=north west][inner sep=0.75pt]    {$\left(\mathbb{X}_{2}^{3} ,\mathbb{X}_{3}^{3}\right)$};

\end{tikzpicture}

    \caption{Pyramid for the case $n=3$, with  ${}^{j}_{i}\X:=\X_0^i\cup\X_j^n$.}
    \label{fig:pyr_drawing}
\end{figure}
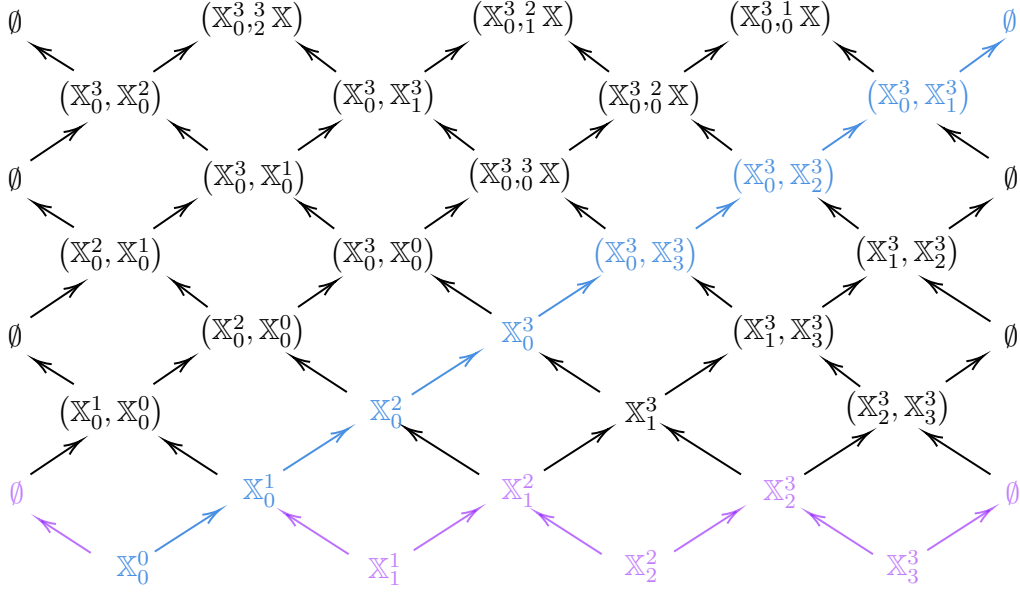

\begin{example}[Pyramid] Figure \ref{fig:pyr_drawing} reports the pyramid built for the case $n=3$. Note that the southern edge sequence corresponds to levelsets zigzag persistence (in cyan), while the left-to-right upward diagonal sequence corresponds to extended persistence (in blue).

\end{example}

\begin{remark}
This construction is called a ``pyramid'' as it can be thought of as a squared-basis pyramid by looking at it from above, where the space $\X_0^n$ is its summit.
\end{remark}

\noindent The Mayer-Vietoris diamonds enable us to effectuate so-called \emph{diamond moves}, that express a bijection between the persistence intervals of two general zigzag diagrams that differ by exactly one Mayer-Vietoris diamond. Thanks to the strong diamond principle, one can go, for example, from extended persistence to levelsets zigzag persistence and vice versa, via a sequence of bijections (diamond moves) between intermediary zigzag diagrams.

\subsection{The Barcode Bijection}

\begin{theorem}(\autocite[Pyramid Theorem]{10.1145/1542362.1542408})
\label{pyr_theo}
There is an explicit bijection between the extended persistence barcode and the levelsets zigzag persistence barcode of $(\X,f)$, that respects homological dimension except for possible shifts of degree $d\in\{-1,1\}$.
\end{theorem}

\begin{proof}
A zigzag diagram in the pyramid is said to be \emph{monotone} if it stretches from the western edge to the eastern edge without backtracking (\textit{i.e.} without making a right-to-left move). Now, for any two monotone zigzags $\mathcal{X}$ and $\mathcal{Y}$, there is a finite sequence of monotone zigzags $\{\mathcal{X}_i\}_{i=1}^{N}$ such that $\mathcal{X}_1=\mathcal{X}$, $\mathcal{X}_N=\mathcal{Y}$ and the zigzag modules $H_p(\mathcal{X}_i)$ and $H_p(\mathcal{X}_{i+1})$ differ by an exact square induced by a Mayer-Vietoris diamond, for any $i\in\{1,...,N-1\}$. Hence, the \emph{Strong Diamond Principle} applies and there is a sequence of bijections
$$\B(H_*(\mathcal{X}))\cong\cdots\cong\B(H_*(\mathcal{X}_i))\cong \B(H_*(\mathcal{X}_{i+1}))\cong\cdots\cong\B(H_*(\mathcal{Y})).$$ As the levelsets zigzag persistence barcode and the extended persistence barcode are both induced by monotone zigzag diagrams in the pyramid, we conclude that there is a bijection between the two of them. Now, the explicit form of the bijection results from tracking down birth and death points along the diamond moves. Finally, a shift of dimension occurs only if the birth and death coordinates find themselves both coinciding with the bottom of a Mayer-Vietoris diamond involved in a diamond move. However, this happens at most once, and thus the assertion about dimension shifts follows by the Strong Diamond Principle.
\end{proof}

\begin{example}[Pyramidal Transformation]
\label{traveling_example}

Let $(\X,f)$ be a pair of Morse type with $n=3$ critical values such that the interval $[\X_1^1,\X_2^2]$ appears in its levelsets zigzag barcode. Denote the levelsets zigzag sequence of $(\X,f)$ by $\mathrm{LZZ}(f)$, and define its up-down sequence to be

$$\mathrm{UD}(f):\X_0^0 \longrightarrow \X_0^1 \longrightarrow \X_0^2 \longrightarrow\X_0^3\longleftarrow\X_1^3\longleftarrow\X_2^3\longleftarrow\X_3^3.$$ 

\noindent Going from one sequence to the other via pyramidal transformation consists in making three consecutive diamond moves, as described below.

\begin{center}
\begin{tikzcd}
\X_0^0 \arrow[r] & \X_0^1           & \color{violet}\X_1^1 \arrow[l] \arrow[r] & \X_1^2 \arrow[d, "\text{Step (1)}", Rightarrow]                     & \X_2^2 \arrow[l] \arrow[r] & \X_2^3           & \X_3^3 \arrow[l]  \\
\X_0^0 \arrow[r] & \X_0^1 \arrow[r] & \color{violet}\X_0^2                     & \X_1^2 \arrow[l] \arrow[d, "\text{Step (2)}", Rightarrow]           & \color{cyan}\X_2^2 \arrow[l] \arrow[r] & \X_2^3           & \X_3^3 \arrow[l] &                 \\
\X_0^0 \arrow[r] & \X_0^1 \arrow[r] & \X_0^2                     & \color{magenta}\X_1^2 \arrow[l] \arrow[r] \arrow[d, "\text{Step (3)}", Rightarrow] & \color{cyan}\X_1^3                     & \X_2^3 \arrow[l] & \X_3^3 \arrow[l] &                 \\
\X_0^0 \arrow[r] & \X_0^1 \arrow[r] & \X_0^2 \arrow[r]           & \color{magenta}\X_0^3                                                              & \X_1^3 \arrow[l]           & \X_2^3 \arrow[l] & \X_3^3 \arrow[l] 
\end{tikzcd}
\end{center}

\noindent The steps indicated in the diagram above involve the following relative Mayer-Vietoris diamonds.
\begin{center}
\begin{tikzcd}
                  & \color{violet}\X_0^2                       &                   &                   & \color{cyan}\X_1^3                       &                   &                   & \color{magenta}\X_0^3                       &                   \\
\X_0^1 \arrow[ru] & (1)              & \X_1^2 \arrow[lu] & \X_1^2 \arrow[ru] & (2)              & \X_2^3 \arrow[lu] & \X_0^2 \arrow[ru] & (3)              & \X_1^3 \arrow[lu] \\
                  & \color{violet}\X_1^1 \arrow[lu] \arrow[ru] &                   &                   & \color{cyan}\X_2^2 \arrow[lu] \arrow[ru] &                   &                   & \color{magenta}\X_1^2 \arrow[lu] \arrow[ru] &                  
\end{tikzcd}
\end{center}

\noindent The interval $[\X_1^1,\X_2^2]$ is transformed according to the bijection expressed in the Pyramid theorem. One obtains the transformation below, whose step-wise birth and death travels are given in Figure \ref{fig:travels}. For example, during the first step, the birth coordinate travels from space $\X_1^1$ to space $\X_0^2$, while the death coordinate stays at space $\X_2^2$. 
\begin{center}
\begin{tikzcd}
{[\X_1^1,\X_2^2]} \arrow[r, "(1)", maps to] & {[\X_0^2,\X_2^2]} \arrow[r, "(2)", maps to] & {[\X_0^2,\X_1^3]} \arrow[r, "(3)", maps to] & {[\X_0^2,\X_1^3]}
\end{tikzcd}
\end{center}

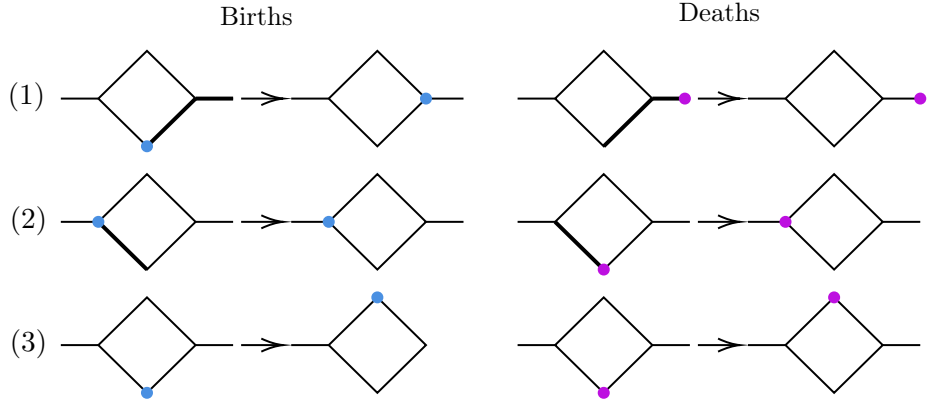
\begin{figure}
    \centering
    
\tikzset{every picture/.style={line width=0.75pt}} 

\begin{tikzpicture}[x=0.75pt,y=0.75pt,yscale=-1,xscale=1]

\draw   (172,131) -- (196.27,155) -- (172,179) -- (147.73,155) -- cycle ;
\draw    (147.73,155) -- (129,155) ;
\draw    (215,155) -- (196.27,155) ;
\draw   (287,131) -- (311.27,155) -- (287,179) -- (262.73,155) -- cycle ;
\draw    (262.73,155) -- (244,155) ;
\draw [line width=0.75]    (330,155) -- (311.27,155) ;
\draw    (219,155) -- (238,155) ;
\draw [shift={(240,155)}, rotate = 180] [color={rgb, 255:red, 0; green, 0; blue, 0 }  ][line width=0.75]    (10.93,-3.29) .. controls (6.95,-1.4) and (3.31,-0.3) .. (0,0) .. controls (3.31,0.3) and (6.95,1.4) .. (10.93,3.29)   ;
\draw   (172,193) -- (196.27,217) -- (172,241) -- (147.73,217) -- cycle ;
\draw    (147.73,217) -- (129,217) ;
\draw    (215,217) -- (196.27,217) ;
\draw   (287,193) -- (311.27,217) -- (287,241) -- (262.73,217) -- cycle ;
\draw    (262.73,217) -- (244,217) ;
\draw    (330,217) -- (311.27,217) ;
\draw    (219,217) -- (238,217) ;
\draw [shift={(240,217)}, rotate = 180] [color={rgb, 255:red, 0; green, 0; blue, 0 }  ][line width=0.75]    (10.93,-3.29) .. controls (6.95,-1.4) and (3.31,-0.3) .. (0,0) .. controls (3.31,0.3) and (6.95,1.4) .. (10.93,3.29)   ;
\draw   (172,255) -- (196.27,279) -- (172,303) -- (147.73,279) -- cycle ;
\draw    (147.73,279) -- (129,279) ;
\draw [line width=0.75]    (215,279) -- (196.27,279) ;
\draw   (287,255) -- (311.27,279) -- (287,303) -- (262.73,279) -- cycle ;
\draw    (262.73,279) -- (244,279) ;
\draw    (219,279) -- (238,279) ;
\draw [shift={(240,279)}, rotate = 180] [color={rgb, 255:red, 0; green, 0; blue, 0 }  ][line width=0.75]    (10.93,-3.29) .. controls (6.95,-1.4) and (3.31,-0.3) .. (0,0) .. controls (3.31,0.3) and (6.95,1.4) .. (10.93,3.29)   ;
\draw   (400,131) -- (424.27,155) -- (400,179) -- (375.73,155) -- cycle ;
\draw    (375.73,155) -- (357,155) ;
\draw [line width=1.5]    (443,155) -- (424.27,155) ;
\draw   (515,131) -- (539.27,155) -- (515,179) -- (490.73,155) -- cycle ;
\draw    (490.73,155) -- (472,155) ;
\draw [line width=0.75]    (558,155) -- (539.27,155) ;
\draw    (447,155) -- (466,155) ;
\draw [shift={(468,155)}, rotate = 180] [color={rgb, 255:red, 0; green, 0; blue, 0 }  ][line width=0.75]    (10.93,-3.29) .. controls (6.95,-1.4) and (3.31,-0.3) .. (0,0) .. controls (3.31,0.3) and (6.95,1.4) .. (10.93,3.29)   ;
\draw   (400,193) -- (424.27,217) -- (400,241) -- (375.73,217) -- cycle ;
\draw    (375.73,217) -- (357,217) ;
\draw    (443,217) -- (424.27,217) ;
\draw   (515,193) -- (539.27,217) -- (515,241) -- (490.73,217) -- cycle ;
\draw    (490.73,217) -- (472,217) ;
\draw    (558,217) -- (539.27,217) ;
\draw    (447,217) -- (466,217) ;
\draw [shift={(468,217)}, rotate = 180] [color={rgb, 255:red, 0; green, 0; blue, 0 }  ][line width=0.75]    (10.93,-3.29) .. controls (6.95,-1.4) and (3.31,-0.3) .. (0,0) .. controls (3.31,0.3) and (6.95,1.4) .. (10.93,3.29)   ;
\draw   (400,255) -- (424.27,279) -- (400,303) -- (375.73,279) -- cycle ;
\draw    (375.73,279) -- (357,279) ;
\draw    (443,279) -- (424.27,279) ;
\draw   (515,255) -- (539.27,279) -- (515,303) -- (490.73,279) -- cycle ;
\draw    (490.73,279) -- (472,279) ;
\draw    (558,279) -- (539.27,279) ;
\draw    (447,279) -- (466,279) ;
\draw [shift={(468,279)}, rotate = 180] [color={rgb, 255:red, 0; green, 0; blue, 0 }  ][line width=0.75]    (10.93,-3.29) .. controls (6.95,-1.4) and (3.31,-0.3) .. (0,0) .. controls (3.31,0.3) and (6.95,1.4) .. (10.93,3.29)   ;
\draw  [color={rgb, 255:red, 74; green, 144; blue, 226 }  ,draw opacity=1 ][fill={rgb, 255:red, 74; green, 144; blue, 226 }  ,fill opacity=1 ] (308.77,155) .. controls (308.77,153.62) and (309.89,152.5) .. (311.27,152.5) .. controls (312.65,152.5) and (313.77,153.62) .. (313.77,155) .. controls (313.77,156.38) and (312.65,157.5) .. (311.27,157.5) .. controls (309.89,157.5) and (308.77,156.38) .. (308.77,155) -- cycle ;
\draw  [color={rgb, 255:red, 189; green, 16; blue, 224 }  ,draw opacity=1 ][fill={rgb, 255:red, 189; green, 16; blue, 224 }  ,fill opacity=1 ] (438,155) .. controls (438,153.62) and (439.12,152.5) .. (440.5,152.5) .. controls (441.88,152.5) and (443,153.62) .. (443,155) .. controls (443,156.38) and (441.88,157.5) .. (440.5,157.5) .. controls (439.12,157.5) and (438,156.38) .. (438,155) -- cycle ;
\draw  [color={rgb, 255:red, 189; green, 16; blue, 224 }  ,draw opacity=1 ][fill={rgb, 255:red, 189; green, 16; blue, 224 }  ,fill opacity=1 ] (555.5,155) .. controls (555.5,153.62) and (556.62,152.5) .. (558,152.5) .. controls (559.38,152.5) and (560.5,153.62) .. (560.5,155) .. controls (560.5,156.38) and (559.38,157.5) .. (558,157.5) .. controls (556.62,157.5) and (555.5,156.38) .. (555.5,155) -- cycle ;
\draw  [color={rgb, 255:red, 189; green, 16; blue, 224 }  ,draw opacity=1 ][fill={rgb, 255:red, 189; green, 16; blue, 224 }  ,fill opacity=1 ] (512.5,255) .. controls (512.5,253.62) and (513.62,252.5) .. (515,252.5) .. controls (516.38,252.5) and (517.5,253.62) .. (517.5,255) .. controls (517.5,256.38) and (516.38,257.5) .. (515,257.5) .. controls (513.62,257.5) and (512.5,256.38) .. (512.5,255) -- cycle ;
\draw  [color={rgb, 255:red, 189; green, 16; blue, 224 }  ,draw opacity=1 ][fill={rgb, 255:red, 189; green, 16; blue, 224 }  ,fill opacity=1 ] (397.5,303) .. controls (397.5,301.62) and (398.62,300.5) .. (400,300.5) .. controls (401.38,300.5) and (402.5,301.62) .. (402.5,303) .. controls (402.5,304.38) and (401.38,305.5) .. (400,305.5) .. controls (398.62,305.5) and (397.5,304.38) .. (397.5,303) -- cycle ;
\draw  [color={rgb, 255:red, 189; green, 16; blue, 224 }  ,draw opacity=1 ][fill={rgb, 255:red, 189; green, 16; blue, 224 }  ,fill opacity=1 ] (488.23,217) .. controls (488.23,215.62) and (489.35,214.5) .. (490.73,214.5) .. controls (492.11,214.5) and (493.23,215.62) .. (493.23,217) .. controls (493.23,218.38) and (492.11,219.5) .. (490.73,219.5) .. controls (489.35,219.5) and (488.23,218.38) .. (488.23,217) -- cycle ;
\draw [line width=1.5]    (172,179) -- (196.27,155) ;
\draw  [color={rgb, 255:red, 74; green, 144; blue, 226 }  ,draw opacity=1 ][fill={rgb, 255:red, 74; green, 144; blue, 226 }  ,fill opacity=1 ] (169.5,179) .. controls (169.5,177.62) and (170.62,176.5) .. (172,176.5) .. controls (173.38,176.5) and (174.5,177.62) .. (174.5,179) .. controls (174.5,180.38) and (173.38,181.5) .. (172,181.5) .. controls (170.62,181.5) and (169.5,180.38) .. (169.5,179) -- cycle ;
\draw [line width=1.5]    (196.27,155) -- (215,155) ;
\draw [line width=1.5]    (147.73,217) -- (172,241) ;
\draw  [color={rgb, 255:red, 74; green, 144; blue, 226 }  ,draw opacity=1 ][fill={rgb, 255:red, 74; green, 144; blue, 226 }  ,fill opacity=1 ] (145.23,217) .. controls (145.23,215.62) and (146.35,214.5) .. (147.73,214.5) .. controls (149.11,214.5) and (150.23,215.62) .. (150.23,217) .. controls (150.23,218.38) and (149.11,219.5) .. (147.73,219.5) .. controls (146.35,219.5) and (145.23,218.38) .. (145.23,217) -- cycle ;
\draw  [color={rgb, 255:red, 74; green, 144; blue, 226 }  ,draw opacity=1 ][fill={rgb, 255:red, 74; green, 144; blue, 226 }  ,fill opacity=1 ] (260.23,217) .. controls (260.23,215.62) and (261.35,214.5) .. (262.73,214.5) .. controls (264.11,214.5) and (265.23,215.62) .. (265.23,217) .. controls (265.23,218.38) and (264.11,219.5) .. (262.73,219.5) .. controls (261.35,219.5) and (260.23,218.38) .. (260.23,217) -- cycle ;
\draw  [color={rgb, 255:red, 74; green, 144; blue, 226 }  ,draw opacity=1 ][fill={rgb, 255:red, 74; green, 144; blue, 226 }  ,fill opacity=1 ] (169.5,303) .. controls (169.5,301.62) and (170.62,300.5) .. (172,300.5) .. controls (173.38,300.5) and (174.5,301.62) .. (174.5,303) .. controls (174.5,304.38) and (173.38,305.5) .. (172,305.5) .. controls (170.62,305.5) and (169.5,304.38) .. (169.5,303) -- cycle ;
\draw  [color={rgb, 255:red, 74; green, 144; blue, 226 }  ,draw opacity=1 ][fill={rgb, 255:red, 74; green, 144; blue, 226 }  ,fill opacity=1 ] (284.5,255) .. controls (284.5,253.62) and (285.62,252.5) .. (287,252.5) .. controls (288.38,252.5) and (289.5,253.62) .. (289.5,255) .. controls (289.5,256.38) and (288.38,257.5) .. (287,257.5) .. controls (285.62,257.5) and (284.5,256.38) .. (284.5,255) -- cycle ;
\draw [line width=1.5]    (375.73,217) -- (400,241) ;
\draw  [color={rgb, 255:red, 189; green, 16; blue, 224 }  ,draw opacity=1 ][fill={rgb, 255:red, 189; green, 16; blue, 224 }  ,fill opacity=1 ] (397.5,241) .. controls (397.5,239.62) and (398.62,238.5) .. (400,238.5) .. controls (401.38,238.5) and (402.5,239.62) .. (402.5,241) .. controls (402.5,242.38) and (401.38,243.5) .. (400,243.5) .. controls (398.62,243.5) and (397.5,242.38) .. (397.5,241) -- cycle ;
\draw [line width=1.5]    (400,179) -- (424.27,155) ;

\draw (101,145.4) node [anchor=north west][inner sep=0.75pt]    {$( 1)$};
\draw (102,207.4) node [anchor=north west][inner sep=0.75pt]    {$( 2)$};
\draw (102,269.4) node [anchor=north west][inner sep=0.75pt]    {$( 3)$};
\draw (207,107.4) node [anchor=north west][inner sep=0.75pt]  [font=\small]  {$\text{Births}$};
\draw (436,105.4) node [anchor=north west][inner sep=0.75pt]  [font=\small]  {$\text{Deaths}$};

\end{tikzpicture}
\caption{Birth and death travels along three consecutive diamond moves.}
    \label{fig:travels}
\end{figure}
\end{example}

\noindent From the pyramid principle, one can deduce the following result, which describes an explicit bijection between levelsets zigzag persistence and extended persistence. The proof relies on case-by-case investigation, as in Example \ref{traveling_example}. Here, intervals appearing in the extended persistence barcode are classified into four types, as described in Section \ref{extended}. 

\begin{theorem}[Barcode Bijection]
\label{bijection_theorem}
One has the following correspondence between the intervals of the extended persistence barcode (left) and intervals of the levelsets zigzag persistence barcode (right).

\begin{center}
\begin{tabular}{|c c c|} 
 \hline
 Type & Extended & Levelsets zigzag\\ [0.5ex] 
 \hline
 I $(i<j)$ & $[\X_0^i,\X_0^{j-1}]$ & $[\X_{i-1}^i,\X_{j-1}^{j-1}]$ \\ 
\hline
 II $(i<j)$ & $[(\X_0^n,\X_{j-1}^n),(\X_0^n,\X_i^n)]^+$ & $[\X_i^i,\X_{j-1}^j]$ \\
 \hline
 III $(i\leq j)$ & $[\X_0^i,(\X_0^n,\X_j^n)]$ & $[\X_{i-1}^i,\X_{j-1}^j]$ \\
\hline
 IV $(i<j)$ & $[\X_0^j,(\X_0^n,\X_i^n)]^+$ & $[\X_{i}^i,\X_{j-1}^{j-1}]$ \\
 \hline
\end{tabular}
\end{center}

\end{theorem}
\begin{proof}
This immediately follows from the proof of Theorem \ref{pyr_theo}.
\end{proof}

\section{Continuous Levelsets Persistence}

In this section, we introduce the continuous levelsets persistence associated to any real valued function, without the need for a Morse type hypothesis. When the context is clear, we shall only call this construction levelsets persistence, as opposed to levelsets zigzag persistence. This construction was initially studied in \cite{2018}.

\subsection{Levelsets persistence modules and interleavings}

For a topological space $\X$ and a function $f:\X\rightarrow\R$, ordinary persistence usually studies the sub-levelsets of $f$, that is, the evolution of the topology of $\S(f)(t) := f^{-1}((-\infty, t])$, as the real parameter $t$ varies. Nevertheless, the persistent homology of the sub-levelsets of $f$ may fail to detect significant topological features that go “upwards” (in the sense of $f$), see figure \ref{fig:cesub}. 

\begin{figure}
    \centering
    \includegraphics[width = 0.5 \textwidth]{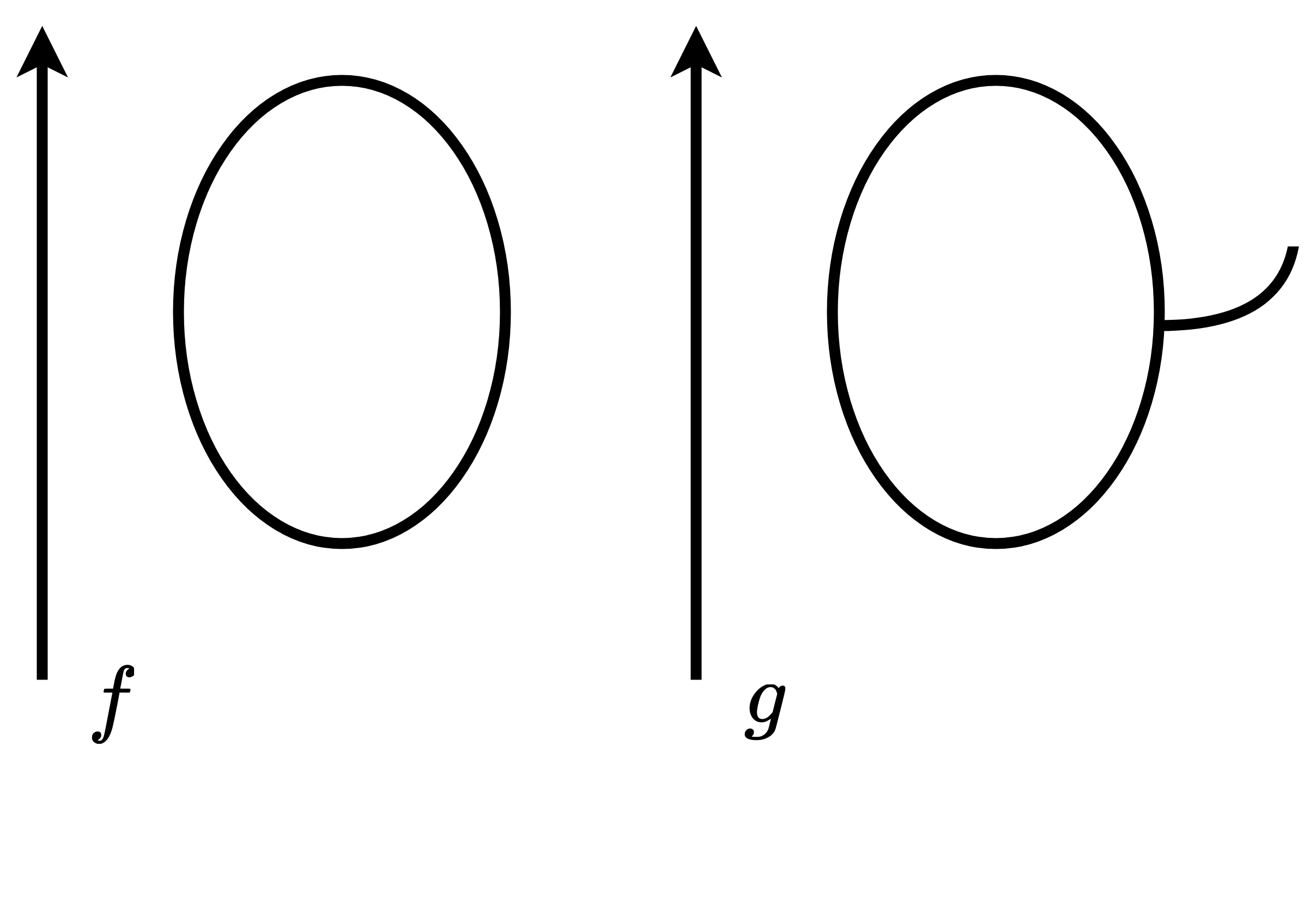}
    \caption{Two functions $f$ and $g$ having isomorphic $0$-th sub-levelsets persistence modules, but non-isomorphic $0$-th levelsets persistence modules.}
    \label{fig:cesub}
\end{figure}

To enrich the topological information extracted from $f$, one can instead study its  \textit{levelsets}  $f^{-1}((a,b))$, as $a < b$ vary. Considering the subposet $\U\subset\R^{\mathrm{op}}\times \R$ consisting of pairs $(a,b)$ with $a < b$, the \textit{level-set filtration of $(\X,f)$} is given by the functor 
\begin{equation*}
\mathcal{L}(f):\begin{aligned}
 \U & \rightarrow \mathbf{Top}\\
(a,b)&\mapsto f^{-1}((a,b))   
\end{aligned},
\end{equation*}

\noindent with $\U$ being a category as a poset, where the relation $(u_1,u_2)\preceq (v_1,v_2)$ is determined by $u_1\geq v_1$ and $u_2\leq v_2$. Indeed, if $v_1\leq u_1\leq u_2\leq v_2$ holds, then there is an inclusion map $\mathcal{L}(f)(u_1,u_2)\xhookrightarrow{} \mathcal{L}(f)(v_1,v_2)$.\\

\begin{definition}
For $p \in \Z_{\geq 0}$, the $p$-th levelsets persistence module of $f : X \to \R$ is the functor $\mathcal{L}_p(f) := H_p \circ \mathcal{L}(f) : \mathbb{U} \longrightarrow \text{Vect}_K$.
\end{definition}

As in the situation of classical persistence, it is possible to introduce a notion of interleaving distance between functors $\mathbb{U} \longrightarrow \text{Vect}_K$. Given $M$ such a functor, and $\varepsilon \geq 0$, we define the $\varepsilon$ shift of $M$ as the functor $M[\varepsilon] : \mathbb{U} \longrightarrow \text{Vect}_K $ defined, for $(x_1,x_2)\in \mathbb{U}$, by:

\[M[\varepsilon](x_1,x_2) := M(x_1 - \varepsilon, x_2 + \varepsilon). \]

There is a natural transformation $\tau_\varepsilon^M : M \longrightarrow M[\varepsilon]$ called the smoothing morphism of $M$.

\begin{definition}

Let $M,N: \mathbb{U} \longrightarrow \text{Vect}_K$ and $\varepsilon \geq 0$. An $\varepsilon$-interleaving between $M$ and $N$ is the data of two morphisms $f:M \to M[\varepsilon]$ and $g:N \to N[\varepsilon]$ fitting in a commutative diagram:$$ \xymatrix{
M  \ar[rd]\ar@/^0.7cm/[rr]^{\tau_{2\varepsilon}^M} \ar[r]^{f} & N[\varepsilon]  \ar[rd] \ar[r]^{g[\varepsilon]} & M[2\varepsilon] \\
N  \ar[ur]\ar@/_0.7cm/[rr]_{\tau_{2\varepsilon}^N} \ar[r]^{g} & M[\varepsilon]  \ar[ur] \ar[r]^{f[\varepsilon]} & N[2\varepsilon]
   } $$
   
In this situation, we will say that $M$ and $N$ are $\varepsilon$-interleaved and write {$M\sim_\varepsilon N$}.

\end{definition}

\begin{definition}
The interleaving distance between $M$ and $N: \mathbb{U} \longrightarrow \text{Vect}_K$ is the possibly infinite number:

\[d_I(M,N) := \inf\{\varepsilon\geq 0 \mid M \sim_\varepsilon N\}.\]
\end{definition}

\begin{proposition}
The interleaving distance satisfies the triangle inequality.
\end{proposition}

The interleaving distance allows expressing the stability property of the levelsets persistence construction.

\begin{theorem}
Let $f,g : \mathbb{X} \to \R$ be functions. Then for all $p\in \Z_{\geq 0}$:

\[d_I(\mathcal{L}_p(f), \mathcal{L}_p(g)) \leq \sup_{x \in \mathbb{X}} \|f(x) - g(x)\|.\]
\end{theorem}

\subsection{Block decomposition}

Since levelsets persistence modules have two parameters, one cannot apply straightforwardly the ordinary persistence theory  to ensure that they have a barcode decomposition. Instead, we need to observe that even though they have two parameters, levelsets persistence module are algebraic constructions originating from one-parameter filtrations. This will be expressed by the property of middle-exactness.

\begin{definition}
A functor $M : \mathbb{U} \longrightarrow \text{Vect}_K$ is said to be middle-exact, if for all $y_1 < x_1 < x_2 < y_2$, the following diagram of vector spaces:

\[\xymatrix{ M(y_1,x_2) \ar[r] & M(y_1,y_2) \\ 
M(x_1,x_2) \ar[r] \ar[u] & M(x_1, y_2) \ar[u] }\]

is an exact square. 
\end{definition}

The classical Mayer-Vietoris exact sequence yields:

\begin{proposition}
Let $f : \mathbb{X} \longrightarrow \R$ be a continuous map. Then for all $p \in \Z_{\geq 0}$, the functor $\mathcal{L}_p(f)$ is middle-exact.
\end{proposition}

In the sequel, we will use the symbol $\overline{<}$ (resp. $\overline{>}$) to denote either $<$ or $\leq$ (resp. $>$ or $\geq$). There are simple examples of middle-exact functors given by so-called block modules. A block is a subset of $ \mathbb{U}$ of one of the following form.

\begin{enumerate}
    \item For $a<b\in [-\infty,+\infty]$, set $(a,b)_{\mathrm{BL}}=\{(x,y)\in\U\mid a\overline{<}x,y \overline{<}    b\}$
    \item For $a<b\in \R\cup\{+\infty\}$, set $[a,b)_{\mathrm{BL}}=\{(x,y)\in\U\mid a\overline{<} y \overline{<} b\}$
    \item For $a<b\in \R\cup\{-\infty\}$, set $(a,b]_{\mathrm{BL}}=\{(x,y)\in\U\mid a \overline{<} x \overline{<} b\}$
    \item For $a\leq b\in \R$, set $[a,b]_{\mathrm{BL}}=\{(x,y)\in\U\mid x \overline{<} b,y\overline{>} a\}$
    \item For $a<b\in\R$, set $[b,a]_{\mathrm{BL}}=\{(x,y)\in\U\mid x\overline{<} a \overline{<} b \overline{<} y\}$
\end{enumerate}

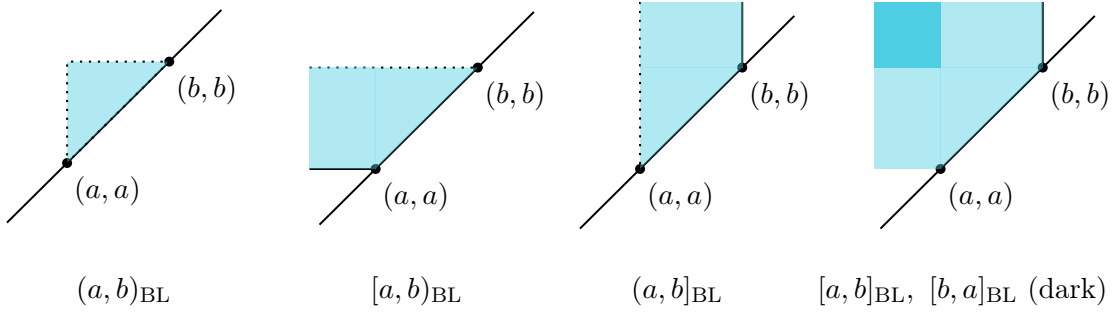
\begin{figure}
    \centering

\tikzset{every picture/.style={line width=0.75pt}} 

\begin{tikzpicture}[x=0.75pt,y=0.75pt,yscale=-1,xscale=1]

\draw    (71,152) -- (178,45) ;
\draw  [fill={rgb, 255:red, 0; green, 0; blue, 0 }  ,fill opacity=1 ] (99,122) .. controls (99,120.9) and (99.9,120) .. (101,120) .. controls (102.1,120) and (103,120.9) .. (103,122) .. controls (103,123.1) and (102.1,124) .. (101,124) .. controls (99.9,124) and (99,123.1) .. (99,122) -- cycle ;
\draw  [fill={rgb, 255:red, 0; green, 0; blue, 0 }  ,fill opacity=1 ] (150,71) .. controls (150,69.9) and (150.9,69) .. (152,69) .. controls (153.1,69) and (154,69.9) .. (154,71) .. controls (154,72.1) and (153.1,73) .. (152,73) .. controls (150.9,73) and (150,72.1) .. (150,71) -- cycle ;
\draw    (227,153) -- (334,46) ;
\draw  [fill={rgb, 255:red, 0; green, 0; blue, 0 }  ,fill opacity=1 ] (253,125) .. controls (253,123.9) and (253.9,123) .. (255,123) .. controls (256.1,123) and (257,123.9) .. (257,125) .. controls (257,126.1) and (256.1,127) .. (255,127) .. controls (253.9,127) and (253,126.1) .. (253,125) -- cycle ;
\draw  [fill={rgb, 255:red, 0; green, 0; blue, 0 }  ,fill opacity=1 ] (304,74) .. controls (304,72.9) and (304.9,72) .. (306,72) .. controls (307.1,72) and (308,72.9) .. (308,74) .. controls (308,75.1) and (307.1,76) .. (306,76) .. controls (304.9,76) and (304,75.1) .. (304,74) -- cycle ;
\draw  [fill={rgb, 255:red, 74; green, 205; blue, 226 }  ,fill opacity=0.43 ][dash pattern={on 0.84pt off 2.51pt}] (152,71) -- (101,122) -- (101,71) -- cycle ;
\draw  [draw opacity=0][fill={rgb, 255:red, 74; green, 205; blue, 226 }  ,fill opacity=0.43 ][dash pattern={on 0.84pt off 2.51pt}] (306,74) -- (255,125) -- (255,74) -- cycle ;
\draw    (222,125) -- (253,125) ;
\draw  [dash pattern={on 0.84pt off 2.51pt}]  (222,74) -- (304,74) ;
\draw    (357,155) -- (464,48) ;
\draw  [fill={rgb, 255:red, 0; green, 0; blue, 0 }  ,fill opacity=1 ] (385,125) .. controls (385,123.9) and (385.9,123) .. (387,123) .. controls (388.1,123) and (389,123.9) .. (389,125) .. controls (389,126.1) and (388.1,127) .. (387,127) .. controls (385.9,127) and (385,126.1) .. (385,125) -- cycle ;
\draw  [fill={rgb, 255:red, 0; green, 0; blue, 0 }  ,fill opacity=1 ] (436,74) .. controls (436,72.9) and (436.9,72) .. (438,72) .. controls (439.1,72) and (440,72.9) .. (440,74) .. controls (440,75.1) and (439.1,76) .. (438,76) .. controls (436.9,76) and (436,75.1) .. (436,74) -- cycle ;
\draw  [draw opacity=0][fill={rgb, 255:red, 74; green, 205; blue, 226 }  ,fill opacity=0.43 ][dash pattern={on 0.84pt off 2.51pt}] (438,74) -- (387,125) -- (387,74) -- cycle ;
\draw    (438,41) -- (438,72) ;
\draw  [dash pattern={on 0.84pt off 2.51pt}]  (387,41) -- (387,123) ;
\draw  [draw opacity=0][fill={rgb, 255:red, 74; green, 205; blue, 226 }  ,fill opacity=0.43 ] (222,74) -- (255,74) -- (255,125) -- (222,125) -- cycle ;
\draw  [draw opacity=0][fill={rgb, 255:red, 74; green, 205; blue, 226 }  ,fill opacity=0.43 ] (438,41) -- (438,74) -- (387,74) -- (387,41) -- cycle ;
\draw    (507,155) -- (614,48) ;
\draw  [fill={rgb, 255:red, 0; green, 0; blue, 0 }  ,fill opacity=1 ] (535,125) .. controls (535,123.9) and (535.9,123) .. (537,123) .. controls (538.1,123) and (539,123.9) .. (539,125) .. controls (539,126.1) and (538.1,127) .. (537,127) .. controls (535.9,127) and (535,126.1) .. (535,125) -- cycle ;
\draw  [fill={rgb, 255:red, 0; green, 0; blue, 0 }  ,fill opacity=1 ] (586,74) .. controls (586,72.9) and (586.9,72) .. (588,72) .. controls (589.1,72) and (590,72.9) .. (590,74) .. controls (590,75.1) and (589.1,76) .. (588,76) .. controls (586.9,76) and (586,75.1) .. (586,74) -- cycle ;
\draw  [draw opacity=0][fill={rgb, 255:red, 74; green, 205; blue, 226 }  ,fill opacity=0.43 ][dash pattern={on 0.84pt off 2.51pt}] (588,74) -- (537,125) -- (537,74) -- cycle ;
\draw    (588,41) -- (588,72) ;
\draw  [draw opacity=0][fill={rgb, 255:red, 74; green, 205; blue, 226 }  ,fill opacity=0.43 ] (588,41) -- (588,74) -- (537,74) -- (537,41) -- cycle ;
\draw  [draw opacity=0][fill={rgb, 255:red, 74; green, 205; blue, 226 }  ,fill opacity=0.43 ] (504,74) -- (537,74) -- (537,125) -- (504,125) -- cycle ;
\draw  [draw opacity=0][fill={rgb, 255:red, 74; green, 205; blue, 226 }  ,fill opacity=0.98 ][line width=0.75]  (504,41) -- (537,41) -- (537,74) -- (504,74) -- cycle ;

\draw (103,127.4) node [anchor=north west][inner sep=0.75pt]    {$( a,a)$};
\draw (154,76.4) node [anchor=north west][inner sep=0.75pt]    {$( b,b)$};
\draw (257,130.4) node [anchor=north west][inner sep=0.75pt]    {$( a,a)$};
\draw (308,79.4) node [anchor=north west][inner sep=0.75pt]    {$( b,b)$};
\draw (389,130.4) node [anchor=north west][inner sep=0.75pt]    {$( a,a)$};
\draw (440,79.4) node [anchor=north west][inner sep=0.75pt]    {$( b,b)$};
\draw (539,130.4) node [anchor=north west][inner sep=0.75pt]    {$( a,a)$};
\draw (590,79.4) node [anchor=north west][inner sep=0.75pt]    {$( b,b)$};
\draw (104,177.4) node [anchor=north west][inner sep=0.75pt]    {$( a,b)_{\mathrm{BL}}$};
\draw (251,177.4) node [anchor=north west][inner sep=0.75pt]    {$[ a,b)_{\mathrm{BL}}$};
\draw (381,177.4) node [anchor=north west][inner sep=0.75pt]    {$( a,b]_{\mathrm{BL}}$};
\draw (474,177.4) node [anchor=north west][inner sep=0.75pt]    {$[ a,b]_{\mathrm{BL}} ,\ [ b,a]_{\mathrm{BL}} \ \text{(dark)}$};

\end{tikzpicture}

    \caption{Five different types of blocks for the levelsets persistence.}
    \label{fig:level_blocks}
\end{figure}

Given a block $B$, we define the functor $K_B :  \mathbb{U} \longrightarrow \textnormal{Vect}_K$ by:

\[K_B(u) := \begin{cases}
K~\text{if}~u \in B \\ 0 ~\text{else}
\end{cases}~~~\textnormal{and}~~~K_B(u\leq v) := \begin{cases}
\text{id}_K ~\text{if}~u,v \in B \\ 0~\text{else}
\end{cases}.\]

\begin{definition}

A functor $M : \mathbb{U} \longrightarrow \text{Vect}_K$ is said to be a block-decomposable module if there exists a multi-set of blocks $\mathbb{B}(M)$ such that: \begin{itemize}
    \item $M \simeq \bigoplus_{B \in \mathbb{B}(M)} K_B$,
    \item for all compact subsets $S \subset \mathbb{U}$, the multi-set $\{B \in \mathbb{B}(M) \mid B \cap S \not = \emptyset\}$ is finite.
\end{itemize}

\end{definition}

By \cite[Theorem 1.1]{BotCraw18}, if $M$ is block decomposable, $\mathbb{B}(M)$ is unique up to reordering of the blocks.

\begin{proposition}
A block-decomposable module is middle-exact and pointwise finite dimensional.
\end{proposition}

The converse holds, and was proved by Cochoy and Oudot in \cite{Cochoy16}.

\begin{theorem}[\cite{Cochoy16}] 
Let $M : \mathbb{U} \longrightarrow \text{Vect}_K$ be a pointwise finite dimensional and middle-exact module. Then $M$ is block-decomposable.
\end{theorem}

\subsection{Bottleneck distance and isometry theorem}

In order to compute the interleaving distance between block decomposable modules, we will introduce, as per classical persistence, a matching distance between multi-sets of blocks. To do so, we follow \cite{2018} and introduce the following partition of block barcodes. 

\begin{definition}
A block $B \subset \mathbb{U}$ is said to be:

\begin{itemize}
    \item of type \textbf{o}, if there exists $a<b$ in $\R$ such that $B=(a,b)_{\mathrm{BL}}$ ;
    
    \item of type \textbf{co}, if there exists $a<b$ in $\R$ such that $B=[a,b)_{\mathrm{BL}}$ or $(-\infty,b)_{\mathrm{BL}}$ ;
    
    \item of type \textbf{oc}, if there exists $a<b$ in $\R$ such that $B=(a,b]_{\mathrm{BL}}$ or $(a,\infty)_{\mathrm{BL}}$ ;
    
    \item of type \textbf{c}, if there exists $a,b\in \R$ in $\R$ such that $B=[a,b]_{\mathrm{BL}}$ or $[a,\infty)_{\mathrm{BL}}$ or $(-\infty,b]_{\mathrm{BL}}$. In the case where $a < b$, we say that $B$ is of subtype $\textbf{c}_1$. Otherwise, we say that $B$ is of subtype $\textbf{c}_2$.
\end{itemize}
\end{definition}

In the following, we will use the notation $\langle a,b \rangle_{\mathrm{BL}}$ when we do not want to specify the orientation of the brackets of the interval.
\begin{lemma}[\cite{2018}]
Let $a\leq b$ and $a'\leq b'$ in  $ \R\cup \{\pm \infty\}$ and $\varepsilon \geq 0$. Then:

\begin{enumerate}
    \item one has $K_{\langle a,b \rangle_{\mathrm{BL}}} \sim_{2 \varepsilon} 0$ if and only if one of the following is true:
    \begin{itemize}
        \item $\langle a,b \rangle_{\mathrm{BL}}$ is of type \textbf{co} or \textbf{oc} and $b-a \leq 2 \varepsilon$;
        
        \item $\langle a,b \rangle_{\mathrm{BL}}$ is of type \textbf{o} and $b-a \leq 4 \varepsilon$ ;
    \end{itemize}
    
    \item one has $K_{\langle a,b \rangle_{\mathrm{BL}}} \sim_{\varepsilon} K_{\langle a',b' \rangle_{\mathrm{BL}}}$ if and only if one of the following is true:
    
    \begin{itemize}
        \item the blocks $\langle a,b \rangle_{\mathrm{BL}}$ and $\langle a',b' \rangle_{\mathrm{BL}} $ are of the same type and $\max(|a-a'|,|b-b'|) \leq \varepsilon$ ; 
        \item $K_{\langle a,b \rangle_{\mathrm{BL}}} \sim_{2 \varepsilon} 0$ and $K_{\langle a',b' \rangle_{\mathrm{BL}}} \sim_{2 \varepsilon} 0$.
    \end{itemize}
\end{enumerate}

\end{lemma} 

\begin{definition} Let $\mathbb{B}_1,\mathbb{B}_2$ be two multisets of blocks, and $\varepsilon\geq 0$. An $\varepsilon$-matching between $\mathbb{B}_1$ and $\mathbb{B}_2$ is the data of two sub-multi-sets $\mathcal{X}_1 \subset \mathbb{B}_1$ and  $\mathcal{X}_2 \subset \mathbb{B}_2$ and a bijection $\sigma : \mathcal{X}_1 \longrightarrow \mathcal{X}_2$ satisfying:

\begin{itemize}
    \item for all $B \in \mathcal{X}_1$, $K_B \sim_\varepsilon K_{\sigma(B)}$;
    
    \item for all $B \in \mathbb{B}_1 \backslash \mathcal{X}_1 \cup  \mathbb{B}_2 \backslash \mathcal{X}_2$, $K_B \sim_{\varepsilon} 0$.
\end{itemize}

\end{definition}

\begin{definition}Let $\mathbb{B}_1,\mathbb{B}_2$ be two multisets of blocks. Their bottleneck distance is the possibly infinite number:
\[d_B(\mathbb{B}_1,\mathbb{B}_2) := \inf \{\varepsilon \geq 0 \mid ~\textnormal{$\mathbb{B}_1$ and $\mathbb{B}_2$ are $\varepsilon$-matched}\}.\]

\end{definition}

\begin{theorem}[\cite{Bjer16}]
Let $M$ and $N$ be two pointwise finite dimensional block-decomposable modules. Then:

\[d_I(M,N) = d_B(\mathbb{B}(M), \mathbb{B}(M)).\]
\end{theorem}

Therefore, when $f$ is continuous and is such that $\mathcal{L}_p(u)$ is finite dimensional for all $u\in \mathbb{U}$, the functor $\mathcal{L}_p$ is block decomposable. In this situation, we call $\mathbb{B}(\mathcal{L}_p(f))$ the $p$-th levelsets barcode of $f$.

Wrapping together the results of this section, we obtain the following.

\begin{corollary}

Let $f,g : \mathbb{X} \longrightarrow \R$ be  continuous functions such that for all $p\in \Z_{\geq 0}$, $\mathcal{L}_p(f)$ and $\mathcal{L}_p(g)$ are pointwise finite dimensional, then:

\[d_I(\mathcal{L}_p(f), \mathcal{L}_p(g)) = d_B(\mathbb{B}(\mathcal{L}_p(f)), \mathbb{B}(\mathcal{L}_p(g)) \leq \sup_{x\in \mathbb{X}}\|f(x) - g(x)\|.\]

\end{corollary}

\section{Relative Interlevelsets Cohomology}
\label{pyr2}

As we already saw, an algebraic stability theory for levelsets persistence was successfully introduced in \cite{2018}, by following almost the same strategy than for one-parameter ordinary persistence. Nevertheless, the counterpart construction regarding extended persistence is more complicated, as extended persistence is in some sense more intricate than levelsets persistence. In \cite{bauer2021structure}, the authors define the Relative Interlevelsets persistence, as a continuous and functorial analogue to the Mayer-Vietoris pyramid (definition \ref{d:MVpyramid}) associated to a real-valued function. It provides a stable, continuous and functorial way of deducing the levelsets persistence barcode of a pair $(\X,f)$ from its extended persistence barcode, and vice-versa.

\bigskip

\noindent \textbf{Notation.} Consider the inverted plane $P:=\R\times\R^{\mathrm{op}}$, where $\R=(\R,\leq)$ and $\R^{\mathrm{op}}=(\R,\geq)$ are posets. One obtains a poset relation on $P$ by $(a,b)\preceq (c,d)$ if and only if $a\leq c$ and $b\geq d$. Moreover, given a point $m\in\M$ where $\M$ is the subset of $P$ defined below, one defines the sets $$\uparrow(m):=\{u\in\M\mid m\preceq u\}\text{ and }\downarrow(m):=\{u\in\M\mid u\preceq m\}.$$ For a subset $S\subset\M$, one defines similarly the sets $$\uparrow(S):=\{u\in\M\mid m\preceq u, \forall m\in S\}\text{ and }\downarrow(S):=\{u\in\M\mid u\preceq m, \forall m\in S\}.$$ 

\subsection{The RISC Functor}

\begin{definition}[Big Strip]
The subset $\M\subset P$ is defined as the convex hull formed by the lines $l_1:=\{(x,y)\in P\mid y=1-x\}$ and $l_2:=\{(x,y)\in P\mid y=-1-x\}$. We call $\M$ the \emph{big strip}, and consider its poset structure inherited from $P$.
\end{definition}

\noindent Write the extended real line as $\Bar{R}:=\R\cup\{\pm\infty\}$. Let $\blacktriangle:\Bar{R}\rightarrow \M$ be an embedding such that the injected copy of $\Bar{R}$ is orthogonal to $l_1$ and goes through the origin of $P$. Furthermore, we write $\bigstar=\mathrm{Im}_\blacktriangle(\Bar{R})$ for the injected copy of $\Bar{R}$ (\textit{cf.} Figure \ref{fig:big_strip}). \\

\begin{figure}
    \centering

\tikzset{every picture/.style={line width=0.75pt}} 

\begin{tikzpicture}[x=0.75pt,y=0.75pt,yscale=-1,xscale=1]

\draw  [color={rgb, 255:red, 155; green, 155; blue, 155 }  ,draw opacity=1 ][fill={rgb, 255:red, 155; green, 155; blue, 155 }  ,fill opacity=0.13 ] (194.5,73.5) -- (451,330) -- (310.5,330.5) -- (54,74) -- cycle ;
\draw    (85,182.01) -- (408,183) ;
\draw [shift={(83,182)}, rotate = 0.18] [color={rgb, 255:red, 0; green, 0; blue, 0 }  ][line width=0.75]    (10.93,-3.29) .. controls (6.95,-1.4) and (3.31,-0.3) .. (0,0) .. controls (3.31,0.3) and (6.95,1.4) .. (10.93,3.29)   ;
\draw    (230,75.5) -- (230,303.5) ;
\draw [shift={(230,73.5)}, rotate = 90] [color={rgb, 255:red, 0; green, 0; blue, 0 }  ][line width=0.75]    (10.93,-3.29) .. controls (6.95,-1.4) and (3.31,-0.3) .. (0,0) .. controls (3.31,0.3) and (6.95,1.4) .. (10.93,3.29)   ;
\draw [color={rgb, 255:red, 57; green, 107; blue, 170 }  ,draw opacity=1 ]   (195.5,217) -- (267.5,147) ;
\draw [shift={(267.5,147)}, rotate = 315.81] [color={rgb, 255:red, 57; green, 107; blue, 170 }  ,draw opacity=1 ][fill={rgb, 255:red, 57; green, 107; blue, 170 }  ,fill opacity=1 ][line width=0.75]      (0, 0) circle [x radius= 3.35, y radius= 3.35]   ;
\draw [shift={(195.5,217)}, rotate = 315.81] [color={rgb, 255:red, 57; green, 107; blue, 170 }  ,draw opacity=1 ][fill={rgb, 255:red, 57; green, 107; blue, 170 }  ,fill opacity=1 ][line width=0.75]      (0, 0) circle [x radius= 3.35, y radius= 3.35]   ;
\draw  [color={rgb, 255:red, 255; green, 255; blue, 255 }  ,draw opacity=1 ][fill={rgb, 255:red, 255; green, 255; blue, 255 }  ,fill opacity=1 ] (53,58) -- (83,58) -- (83,205) -- (53,205) -- cycle ;
\draw  [color={rgb, 255:red, 255; green, 255; blue, 255 }  ,draw opacity=1 ][fill={rgb, 255:red, 255; green, 255; blue, 255 }  ,fill opacity=1 ] (211.5,68.5) -- (211.5,74.5) -- (64.5,74.5) -- (64.5,68.5) -- cycle ;
\draw  [color={rgb, 255:red, 255; green, 255; blue, 255 }  ,draw opacity=1 ][fill={rgb, 255:red, 255; green, 255; blue, 255 }  ,fill opacity=1 ] (451,330) -- (451,335.5) -- (304,335.5) -- (304,330) -- cycle ;
\draw  [color={rgb, 255:red, 255; green, 255; blue, 255 }  ,draw opacity=1 ][fill={rgb, 255:red, 255; green, 255; blue, 255 }  ,fill opacity=1 ] (427,209) -- (451,209) -- (451,356) -- (427,356) -- cycle ;

\draw (262,295.4) node [anchor=north west][inner sep=0.75pt]    {$l_{1}$};
\draw (388,245.4) node [anchor=north west][inner sep=0.75pt]    {$l_{2}$};
\draw (173,222.4) node [anchor=north west][inner sep=0.75pt]  [font=\footnotesize]  {$\bigstar $};
\draw (237,74.4) node [anchor=north west][inner sep=0.75pt]    {$\mathbb{R}$};
\draw (88,195.4) node [anchor=north west][inner sep=0.75pt]    {$\mathbb{R}^{\text{op}}$};
\draw (136,107.4) node [anchor=north west][inner sep=0.75pt]    {$\mathbb{M}$};

\end{tikzpicture}

    \caption{The big strip $\M$ containing $\bigstar=\mathrm{Im}\blacktriangle$ as a subset.}
    \label{fig:big_strip}
\end{figure}
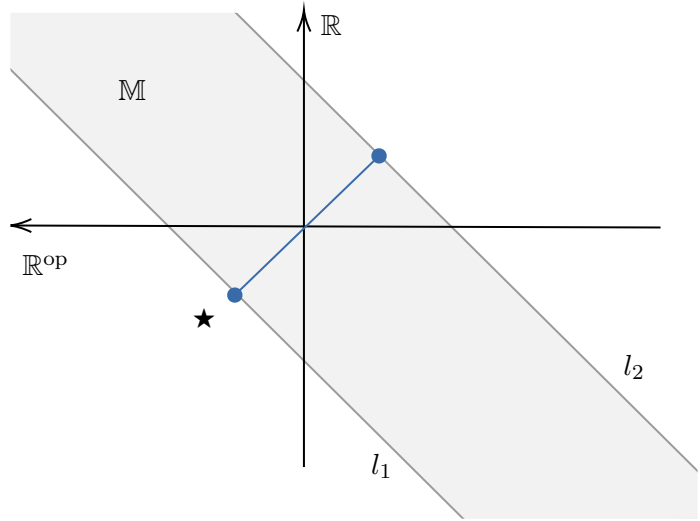

\noindent Let $T\in\mathrm{End}(\M)$ be the invertible endomorphism of posets defined as follows. For $m\in \M$, let $h_1$ (resp. $v_1$) be the horizontal (resp. vertical) line passing through $m$. Let $a$ be the intersection of $l_2$ and $v_1$. Let $h_2$ be the horizontal line passing through $a$, and let $b$ be the intersection of $l_1$ and $h_1$. Finally, let $v_2$ be the vertical line passing through $b$. For $m\in\M$, $T(m)$ is defined as the intersection point of $v_2$ and $h_2$ (\textit{cf.} Figure \ref{fig:endo_T}).

\begin{figure}
    \centering

\tikzset{every picture/.style={line width=0.75pt}} 

\begin{tikzpicture}[x=0.75pt,y=0.75pt,yscale=-1,xscale=1]

\draw  [fill={rgb, 255:red, 74; green, 144; blue, 226 }  ,fill opacity=0 ] (232.75,134.75) -- (307.5,209.5) -- (232.75,284.25) -- (158,209.5) -- cycle ;
\draw    (120.01,172) -- (158,209.5) ;
\draw    (225,127) -- (340,242) ;
\draw    (232.75,284.25) -- (249.75,301.25) ;
\draw  [dash pattern={on 4.5pt off 4.5pt}]  (249.75,301.25) -- (271.25,323.25) ;
\draw  [dash pattern={on 4.5pt off 4.5pt}]  (340,242) -- (378.25,280.25) ;
\draw  [dash pattern={on 4.5pt off 4.5pt}]  (80.01,132) -- (120.01,172) ;
\draw  [dash pattern={on 4.5pt off 4.5pt}]  (187.5,89.5) -- (225,127) ;
\draw    (232.75,134.75) -- (232.75,284.25) ;
\draw    (158,209.5) -- (307.5,209.5) ;
\draw [color={rgb, 255:red, 57; green, 107; blue, 170 }  ,draw opacity=1 ][line width=2.25]    (158,209.5) -- (232.75,134.75) ;
\draw  [draw opacity=0][fill={rgb, 255:red, 226; green, 130; blue, 130 }  ,fill opacity=1 ] (246.5,252.25) .. controls (246.5,250.87) and (247.62,249.75) .. (249,249.75) .. controls (250.38,249.75) and (251.5,250.87) .. (251.5,252.25) .. controls (251.5,253.63) and (250.38,254.75) .. (249,254.75) .. controls (247.62,254.75) and (246.5,253.63) .. (246.5,252.25) -- cycle ;
\draw  [draw opacity=0][fill={rgb, 255:red, 226; green, 130; blue, 130 }  ,fill opacity=1 ] (198,152.25) .. controls (198,150.87) and (199.12,149.75) .. (200.5,149.75) .. controls (201.88,149.75) and (203,150.87) .. (203,152.25) .. controls (203,153.63) and (201.88,154.75) .. (200.5,154.75) .. controls (199.12,154.75) and (198,153.63) .. (198,152.25) -- cycle ;
\draw  [color={rgb, 255:red, 0; green, 0; blue, 0 }  ,draw opacity=0.55 ][dash pattern={on 0.84pt off 2.51pt}] (200.5,152.25) -- (249.5,152.25) -- (249.5,252.25) -- (200.5,252.25) -- cycle ;

\draw (139,209.4) node [anchor=north west][inner sep=0.75pt]  [font=\footnotesize]  {$\bigstar $};
\draw (251,234.57) node [anchor=north west][inner sep=0.75pt]    {$m$};
\draw (162.17,134.4) node [anchor=north west][inner sep=0.75pt]    {$T( m)$};

\end{tikzpicture}

    \caption{Schematic picture of the map $T$.}
    \label{fig:endo_T}
\end{figure}
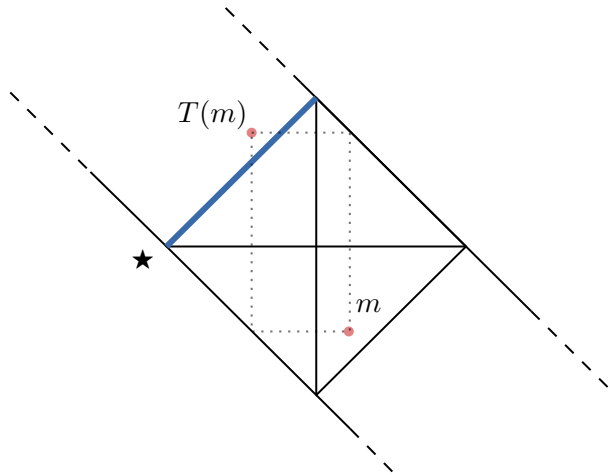

\begin{remark}
The map $T$ acts on $\M$ by composition when considering the action $\Z\cong \langle T\rangle\circlearrowright\M$ given by $k\cdot m = T^k(m)$ for $m\in\M$. This action induces a tessellation of the big strip $\M$ by seeing it as the orbit of $D$, introduced in Definition \ref{D} below. 
\end{remark}

\begin{definition}[Fundamental Domain]
\label{D}
The \emph{fundamental domain} of $\M$ is defined as $$D:=\downarrow(\bigstar)\setminus T^{-1}(\downarrow(\bigstar)).$$
\end{definition}

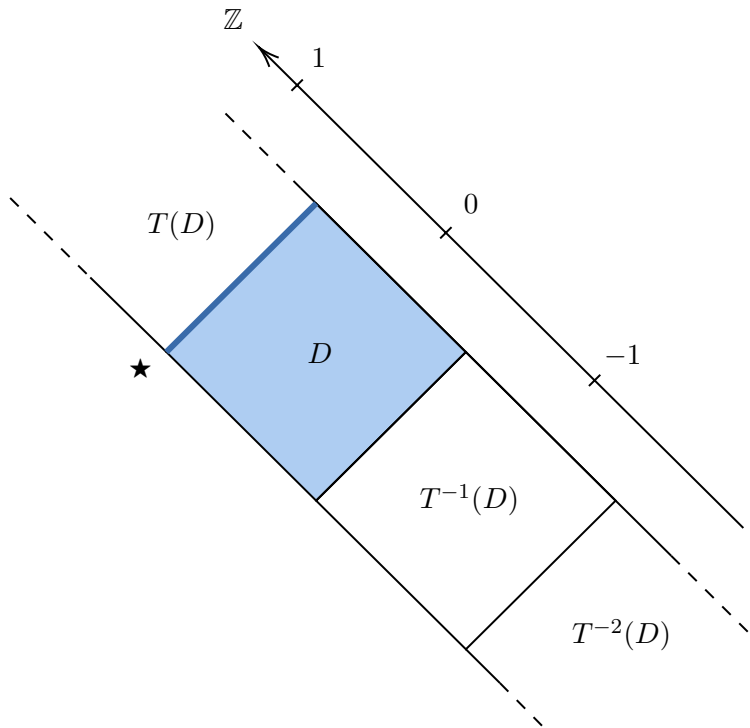
\begin{figure}
    \centering
\tikzset{every picture/.style={line width=0.75pt}} 

\tikzset{every picture/.style={line width=0.75pt}} 

\begin{tikzpicture}[x=0.75pt,y=0.75pt,yscale=-1,xscale=1]

\draw  [fill={rgb, 255:red, 74; green, 144; blue, 226 }  ,fill opacity=0.45 ] (210.75,114.75) -- (285.5,189.5) -- (210.75,264.25) -- (136,189.5) -- cycle ;
\draw   (285.5,189.5) -- (360.25,264.25) -- (285.5,339) -- (210.75,264.25) -- cycle ;
\draw    (98.01,152) -- (136,189.5) ;
\draw    (203,107) -- (388,292) ;
\draw    (285.5,339) -- (302.5,356) ;
\draw  [dash pattern={on 4.5pt off 4.5pt}]  (302.5,356) -- (324,377.99) ;
\draw  [dash pattern={on 4.5pt off 4.5pt}]  (388,292) -- (426.25,330.25) ;
\draw  [dash pattern={on 4.5pt off 4.5pt}]  (58.01,112) -- (98.01,152) ;
\draw  [dash pattern={on 4.5pt off 4.5pt}]  (165.5,69.5) -- (203,107) ;
\draw    (424,278) -- (183.41,37.41) (346.93,206.58) -- (352.58,200.93)(272.68,132.34) -- (278.34,126.68)(198.43,58.09) -- (204.09,52.43) ;
\draw [shift={(182,36)}, rotate = 45] [color={rgb, 255:red, 0; green, 0; blue, 0 }  ][line width=0.75]    (10.93,-3.29) .. controls (6.95,-1.4) and (3.31,-0.3) .. (0,0) .. controls (3.31,0.3) and (6.95,1.4) .. (10.93,3.29)   ;
\draw [color={rgb, 255:red, 57; green, 107; blue, 170 }  ,draw opacity=1 ][line width=2.25]    (136,189.5) -- (210.75,114.75) ;

\draw (125,117.4) node [anchor=north west][inner sep=0.75pt]    {$T( D)$};
\draw (261,254.4) node [anchor=north west][inner sep=0.75pt]    {$T^{-1}( D)$};
\draw (337,322.4) node [anchor=north west][inner sep=0.75pt]    {$T^{-2}( D)$};
\draw (205,183.4) node [anchor=north west][inner sep=0.75pt]    {$D$};
\draw (163,15.4) node [anchor=north west][inner sep=0.75pt]    {$\mathbb{Z}$};
\draw (353,184.4) node [anchor=north west][inner sep=0.75pt]    {$-1$};
\draw (283,108.4) node [anchor=north west][inner sep=0.75pt]    {$0$};
\draw (207,35.4) node [anchor=north west][inner sep=0.75pt]    {$1$};
\draw (116,191.4) node [anchor=north west][inner sep=0.75pt]  [font=\footnotesize]  {$\bigstar $};

\end{tikzpicture}

\caption{Tessellation on $\M$ induced by the action $\Z\cong \langle T\rangle\circlearrowright\M$.}
    \label{fig:tessellation}
\end{figure}

\noindent The goal is to define a functor that \textit{reads} $\M$ as a gluing of all the \textit{homology pyramids} (those are diagrams of vector spaces obtained by applying a homology functor to the pyramid introduced in Section \ref{pyr1}). To this end, one begins by defining a morphism of posets $\rho$ that associates to each point $m\in\M$ a pair of open intervals $(\rho_1(m),\rho_2(m))$.

\begin{definition}
\label{rho_def}
One defines the map 
\begin{equation*}\rho:
\begin{cases}
             \M\rightarrow \mathcal{P}:=\mathrm{Op}(\R)\times\mathrm{Op}(\R)\\
             m\mapsto \left(\blacktriangle^{-1}(\mathrm{int}(\downarrow T(m))),\blacktriangle^{-1}(\M\setminus\uparrow m)\right),
       \end{cases}
\end{equation*}
where $\mathrm{Op}(\R)$ denotes the set of open sets of $\R$ (for the standard topology).
\end{definition}

\noindent One can endow $\mathcal{P}$ with a poset relation $\preceq_{\mathcal{P}}$ defined as follows. For two pairs $(I_1,J_1), (I_2,J_2)\in\mathcal{P}$, one has $(I_1,J_1)\preceq_{\mathcal{P}} (I_2,J_2)$ if and only if $I_1\subseteq I_2$ and $J_1\subseteq J_2$. This way, the map $\rho:\M\rightarrow\mathcal{P}$ becomes a morphism of posets (\textit{i.e.} is monotone).

\begin{figure}
    \centering

\tikzset{every picture/.style={line width=0.75pt}} 

\begin{tikzpicture}[x=0.75pt,y=0.75pt,yscale=-1,xscale=1]

\draw  [fill={rgb, 255:red, 74; green, 144; blue, 226 }  ,fill opacity=0.45 ] (210.75,114.75) -- (285.5,189.5) -- (210.75,264.25) -- (136,189.5) -- cycle ;
\draw    (98.01,152) -- (136,189.5) ;
\draw    (203,107) -- (318,222) ;
\draw    (210.75,264.25) -- (227.75,281.25) ;
\draw  [dash pattern={on 4.5pt off 4.5pt}]  (227.75,281.25) -- (249.25,303.25) ;
\draw  [dash pattern={on 4.5pt off 4.5pt}]  (318,222) -- (356.25,260.25) ;
\draw  [dash pattern={on 4.5pt off 4.5pt}]  (58.01,112) -- (98.01,152) ;
\draw  [dash pattern={on 4.5pt off 4.5pt}]  (165.5,69.5) -- (203,107) ;
\draw    (210.75,114.75) -- (210.75,264.25) ;
\draw    (136,189.5) -- (285.5,189.5) ;
\draw [color={rgb, 255:red, 57; green, 107; blue, 170 }  ,draw opacity=1 ][line width=2.25]    (136,189.5) -- (210.75,114.75) ;
\draw  [dash pattern={on 0.84pt off 2.51pt}] (234.88,154) -- (234.88,225) -- (186.63,225) -- (186.63,154) -- cycle ;
\draw  [draw opacity=0][fill={rgb, 255:red, 226; green, 130; blue, 130 }  ,fill opacity=1 ] (232.38,225) .. controls (232.38,223.62) and (233.49,222.5) .. (234.88,222.5) .. controls (236.26,222.5) and (237.38,223.62) .. (237.38,225) .. controls (237.38,226.38) and (236.26,227.5) .. (234.88,227.5) .. controls (233.49,227.5) and (232.38,226.38) .. (232.38,225) -- cycle ;
\draw  [draw opacity=0][fill={rgb, 255:red, 226; green, 130; blue, 130 }  ,fill opacity=1 ] (184.13,225) .. controls (184.13,223.62) and (185.25,222.5) .. (186.63,222.5) .. controls (188.01,222.5) and (189.13,223.62) .. (189.13,225) .. controls (189.13,226.38) and (188.01,227.5) .. (186.63,227.5) .. controls (185.25,227.5) and (184.13,226.38) .. (184.13,225) -- cycle ;
\draw  [draw opacity=0][fill={rgb, 255:red, 226; green, 130; blue, 130 }  ,fill opacity=1 ] (184.13,154) .. controls (184.13,152.62) and (185.25,151.5) .. (186.63,151.5) .. controls (188.01,151.5) and (189.13,152.62) .. (189.13,154) .. controls (189.13,155.38) and (188.01,156.5) .. (186.63,156.5) .. controls (185.25,156.5) and (184.13,155.38) .. (184.13,154) -- cycle ;
\draw  [draw opacity=0][fill={rgb, 255:red, 226; green, 130; blue, 130 }  ,fill opacity=1 ] (232.38,154) .. controls (232.38,152.62) and (233.49,151.5) .. (234.88,151.5) .. controls (236.26,151.5) and (237.38,152.62) .. (237.38,154) .. controls (237.38,155.38) and (236.26,156.5) .. (234.88,156.5) .. controls (233.49,156.5) and (232.38,155.38) .. (232.38,154) -- cycle ;
\draw    (239.38,149) -- (272.57,116.4) ;
\draw [shift={(274,115)}, rotate = 135.52] [color={rgb, 255:red, 0; green, 0; blue, 0 }  ][line width=0.75]    (10.93,-3.29) .. controls (6.95,-1.4) and (3.31,-0.3) .. (0,0) .. controls (3.31,0.3) and (6.95,1.4) .. (10.93,3.29)   ;
\draw [shift={(239.38,149)}, rotate = 135.52] [color={rgb, 255:red, 0; green, 0; blue, 0 }  ][line width=0.75]    (0,5.59) -- (0,-5.59)   ;
\draw    (181.25,148.5) -- (150.41,117.42) ;
\draw [shift={(149,116)}, rotate = 45.22] [color={rgb, 255:red, 0; green, 0; blue, 0 }  ][line width=0.75]    (10.93,-3.29) .. controls (6.95,-1.4) and (3.31,-0.3) .. (0,0) .. controls (3.31,0.3) and (6.95,1.4) .. (10.93,3.29)   ;
\draw [shift={(181.25,148.5)}, rotate = 45.22] [color={rgb, 255:red, 0; green, 0; blue, 0 }  ][line width=0.75]    (0,5.59) -- (0,-5.59)   ;
\draw    (182.13,230) -- (148.42,263.59) ;
\draw [shift={(147,265)}, rotate = 315.1] [color={rgb, 255:red, 0; green, 0; blue, 0 }  ][line width=0.75]    (10.93,-3.29) .. controls (6.95,-1.4) and (3.31,-0.3) .. (0,0) .. controls (3.31,0.3) and (6.95,1.4) .. (10.93,3.29)   ;
\draw [shift={(182.13,230)}, rotate = 315.1] [color={rgb, 255:red, 0; green, 0; blue, 0 }  ][line width=0.75]    (0,5.59) -- (0,-5.59)   ;
\draw    (240.11,230.05) -- (281.55,269.62) ;
\draw [shift={(283,271)}, rotate = 223.68] [color={rgb, 255:red, 0; green, 0; blue, 0 }  ][line width=0.75]    (10.93,-3.29) .. controls (6.95,-1.4) and (3.31,-0.3) .. (0,0) .. controls (3.31,0.3) and (6.95,1.4) .. (10.93,3.29)   ;
\draw [shift={(240.11,230.05)}, rotate = 223.68] [color={rgb, 255:red, 0; green, 0; blue, 0 }  ][line width=0.75]    (0,5.59) -- (0,-5.59)   ;
\draw  [dash pattern={on 0.84pt off 2.51pt}]  (234.88,227.5) -- (235.11,232.05) -- (235.5,239.5) ;
\draw  [dash pattern={on 0.84pt off 2.51pt}]  (237.38,225) -- (248.99,225.05) -- (242.5,225) ;
\draw  [fill={rgb, 255:red, 0; green, 0; blue, 0 }  ,fill opacity=1 ] (234.5,238.5) .. controls (234.5,237.95) and (234.95,237.5) .. (235.5,237.5) .. controls (236.05,237.5) and (236.5,237.95) .. (236.5,238.5) .. controls (236.5,239.05) and (236.05,239.5) .. (235.5,239.5) .. controls (234.95,239.5) and (234.5,239.05) .. (234.5,238.5) -- cycle ;
\draw  [fill={rgb, 255:red, 0; green, 0; blue, 0 }  ,fill opacity=1 ] (248.99,225.05) .. controls (248.99,224.49) and (249.44,224.05) .. (249.99,224.05) .. controls (250.54,224.05) and (250.99,224.49) .. (250.99,225.05) .. controls (250.99,225.6) and (250.54,226.05) .. (249.99,226.05) .. controls (249.44,226.05) and (248.99,225.6) .. (248.99,225.05) -- cycle ;
\draw  [fill={rgb, 255:red, 0; green, 0; blue, 0 }  ,fill opacity=1 ] (233,246) -- (237,254) -- (229,254) -- cycle ;
\draw  [fill={rgb, 255:red, 0; green, 0; blue, 0 }  ,fill opacity=1 ] (257,228) -- (261,236) -- (253,236) -- cycle ;

\draw (214,170.4) node [anchor=north west][inner sep=0.75pt]    {$D$};
\draw (117,189.4) node [anchor=north west][inner sep=0.75pt]  [font=\footnotesize]  {$\bigstar $};
\draw (261.99,224.45) node [anchor=north west][inner sep=0.75pt]  [font=\small]  {$b$};
\draw (238.5,241.9) node [anchor=north west][inner sep=0.75pt]  [font=\small]  {$a$};
\draw (277,95.4) node [anchor=north west][inner sep=0.75pt]    {$([ a,\infty ) ,[ b,\infty ))$};
\draw (285,274.4) node [anchor=north west][inner sep=0.75pt]    {$([ a,b] ,\emptyset )$};
\draw (20,269.4) node [anchor=north west][inner sep=0.75pt]    {$(( -\infty ,b] ,( -\infty ,a])$};
\draw (64,91.4) node [anchor=north west][inner sep=0.75pt]    {$(\mathbb{R} ,\mathbb{R} \setminus ( a,b))$};

\end{tikzpicture}

    \caption{Schematic functioning of the map $\rho$.}
    \label{fig:rho_map}
\end{figure}
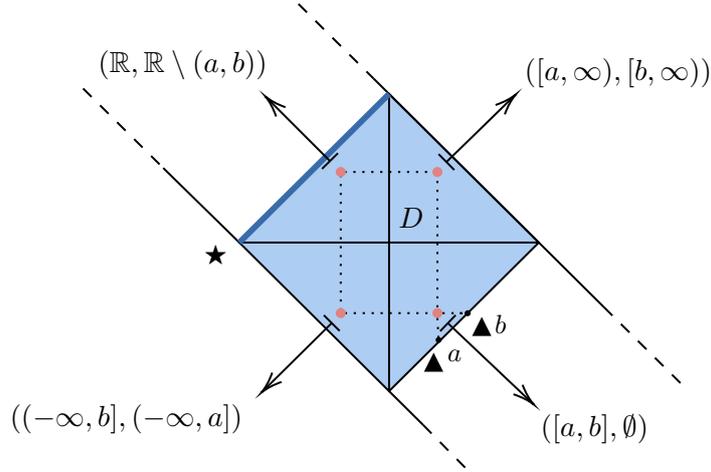

\begin{definition}[Operator $(\cdot)^\Z$]
For a category $\mathcal{C}$, one can define $\mathcal{C}^\Z$ as the category having for objects maps $M^\bullet:\Z\rightarrow\mathcal{C}:n\mapsto M^n$. Morphisms of $\mathcal{C}^\Z$ are  defined pointwise.
\end{definition}

\noindent \textbf{Notation.} For a category $\mathcal{C}$, let $\mathcal{C}^{\mathrm{op}}$ denote the opposite category, with the same objects and arrows reversed. Let $\mathrm{vect}_K$ denote the category of finite dimensional $K$-vector spaces with linear maps as morphisms, and let $\mathrm{Vect}_K^\Z$ denote the category of $\Z$-graded vector spaces over $K$ with pointwise linear maps as morphisms.

\begin{definition}[Evaluation]
Let $\mathcal{C}$ be a category. Define the \emph{evaluation functor} \begin{equation*}
    \mathrm{ev}^0: 
    \begin{cases}
    \mathcal{C}^\Z\rightarrow \mathcal{C}\\
    M^\bullet\mapsto M^0
    \end{cases}.
\end{equation*}
\end{definition}

\noindent Now, one can define the notion of \textit{extended persistence diagram} associated to a continuous function $f:\X\rightarrow\R$ in the sense of  \cite{bauer2021structure}. First, one needs to define the \textit{RISC functor} associated to $f$, as well as the persistence diagram of a contravariant functor $F:\M^{\mathrm{op}}\rightarrow \mathrm{Vect}_K$. Consider the strip $\M$ as a category whose objects are the points on $\M$ and whose arrows $\rightarrow$ are given by the relations $\preceq$. Let $\mathrm{int}(\M)$ denote the interior of the strip $\M$ and $\partial\M$ its boundary. Finally, consider a general cohomology theory $\mathcal{H}^\bullet$ that takes values in the category $\mathrm{Vect}_K^\Z$, sending weak equivalences to isomorphisms.\\

\noindent Consider $D$ as a category whose objects are the points of $D$ and whose morphisms are the relations $\preceq$. Define the map  $$F'_f:=F':D\rightarrow \left (\mathrm{Vect}_K^\Z\right )^{\mathrm{op}}, m\mapsto \mathcal{H}^\bullet(f^{-1}(\rho_1(m)),f^{-1}(\rho_2(m)))$$ assigning to each point of $\downarrow(\bigstar)$ a homology group, where $p\in\N$ is the only integer such that $T^p(m)\in D$ and $\rho=(\rho_1,\rho_2)$.\\

\noindent \textbf{Fact.} \textit{$F'$ is a contravariant functor.} Indeed, for two elements $u\preceq v\in\M$, we have $\rho_i(u)\subseteq \rho_i(v)$, $i=1,2$ and thus we obtain a linear map $F'_f(v)\longrightarrow F'_f(u)$. This way, $F'_f$ inherits the properties of a contravariant functor from the cohomology theory $\mathcal{H}^\bullet$. \\

\noindent Let $\Sigma:(\mathrm{Vect}_K^\Z)^{\mathrm{op}}\rightarrow (\mathrm{Vect}_K^\Z)^{\mathrm{op}}$ be the degree-shift endofunctor acting as $\Sigma(M^\bullet)=M^{\bullet-1}$. One can extend the functor $F'$ to a functor $F:\M \longrightarrow \left (\mathrm{Vect}_K^\Z\right )^{\mathrm{op}}$ such that the square 

\begin{center}
\begin{tikzcd}
\M \arrow[d, "T"] \arrow[r, "F"] & (\mathrm{Vect}_K^\Z)^{\mathrm{op}} \arrow[d, "\Sigma"] \\
\M \arrow[r, "F"]                & (\mathrm{Vect}_K^\Z)^{\mathrm{op}}                    
\end{tikzcd}
\end{center}
commutes, \textit{i.e.} $\Sigma\circ F = F\circ T$. \\

\noindent Note that the transformation $T$ corresponds to degree-shifts, and thus there is unnecessary information within the extended functor $F$. First, consider the opposite functor $F^{\mathrm{op}}:\M^{\mathrm{op}}\rightarrow\mathrm{Vect}_K^\Z$. Second, compose it with an evaluation map to obtain the desired functor $$h(f):=\mathrm{ev}^0\circ F^{\mathrm{op}}:\M^{\mathrm{op}}\rightarrow\mathrm{Vect}_K,$$ where there is no redundancy in the information it contains.

\begin{definition}[RISC]
Let $f:\X\rightarrow\R$ be a continuous function. The \emph{relative interlevelsets cohomology (RISC) functor} of $(\X,f)$ is defined as the functor $$h(f):=\mathrm{ev}^0\circ F^{\mathrm{op}}:\M^{\mathrm{op}}\rightarrow\mathrm{Vect}_K.$$
\end{definition}

\begin{definition}[Persistence diagram]
Let $G:\M^{\mathrm{op}}\rightarrow \mathrm{Vect}_K$ be a contravariant pointwise finite dimensional functor  that vanishes on $\partial\M$. The $\emph{extended persistence diagram}$ of $G$ is defined as the map 
\begin{equation*}
    \mathrm{Dgm}(G):
    \begin{cases}
    \mathrm{int}(\M)\rightarrow \N_0\\
    m\mapsto \mathrm{dim}_K(G(m))-\mathrm{dim}_K\left(\sum_{u\succ m} \mathrm{Im} G(m\preceq u)\right)
    \end{cases}
\end{equation*}
\end{definition}

\begin{definition}[Extended persistence diagram]
Let $f:\X\rightarrow\R$ be a continuous function with RISC $h(f)$ that is pointwise finite dimensional. The \emph{extended persistence diagram} of $(\X,f)$ is defined as the map $\mathrm{Dgm}(f):=\mathrm{Dgm}(h(f))$.
\end{definition}

\begin{figure}
    \centering

\tikzset{every picture/.style={line width=0.75pt}} 

\begin{tikzpicture}[x=0.75pt,y=0.75pt,yscale=-1,xscale=1]

\draw  [draw opacity=0][fill={rgb, 255:red, 209; green, 180; blue, 233 }  ,fill opacity=1 ] (213,227) -- (320.99,227) -- (320.99,309.62) -- (213,309.62) -- cycle ;
\draw  [draw opacity=0][fill={rgb, 255:red, 74; green, 144; blue, 226 }  ,fill opacity=0.62 ] (239,238) -- (332,238) -- (332,334) -- (239,334) -- cycle ;
\draw [color={rgb, 255:red, 86; green, 189; blue, 216 }  ,draw opacity=1 ][line width=1.5]    (315.01,40.6) -- (341,98) ;
\draw [color={rgb, 255:red, 86; green, 189; blue, 216 }  ,draw opacity=1 ][line width=1.5]    (315.01,40.6) -- (316.25,84.42) ;
\draw [color={rgb, 255:red, 24; green, 97; blue, 150 }  ,draw opacity=1 ][line width=1.5]    (285.27,21) -- (286,162) ;
\draw [color={rgb, 255:red, 124; green, 102; blue, 202 }  ,draw opacity=1 ][line width=1.5]    (259.26,65.97) -- (232,98.26) ;
\draw [color={rgb, 255:red, 124; green, 102; blue, 202 }  ,draw opacity=1 ][line width=1.5]    (232,98.26) -- (261,146) ;
\draw [color={rgb, 255:red, 124; green, 102; blue, 202 }  ,draw opacity=1 ][line width=1.5]    (261,146) -- (272.88,100.57) ;
\draw [color={rgb, 255:red, 124; green, 102; blue, 202 }  ,draw opacity=1 ][line width=1.5]    (272.88,100.57) -- (259.26,65.97) ;
\draw [color={rgb, 255:red, 86; green, 189; blue, 216 }  ,draw opacity=1 ][line width=1.5]    (345.01,40.6) -- (346.25,84.42) ;
\draw  [fill={rgb, 255:red, 74; green, 144; blue, 226 }  ,fill opacity=0.45 ] (247.15,151.6) -- (340.7,247.16) -- (247.15,342.71) -- (153.61,247.16) -- cycle ;
\draw   (340.7,247.16) -- (434.24,342.71) -- (340.7,438.27) -- (247.15,342.71) -- cycle ;
\draw    (106,199.22) -- (153.56,247.16) ;
\draw    (237.43,141.69) -- (469,378.19) ;
\draw    (340.7,438.27) -- (361.98,460) ;
\draw [color={rgb, 255:red, 57; green, 107; blue, 170 }  ,draw opacity=1 ][line width=0.75]    (153.56,247.16) -- (247.13,151.6) ;
\draw    (395,3) -- (395,169) ;
\draw [shift={(395,1)}, rotate = 90] [color={rgb, 255:red, 0; green, 0; blue, 0 }  ][line width=0.75]    (10.93,-3.29) .. controls (6.95,-1.4) and (3.31,-0.3) .. (0,0) .. controls (3.31,0.3) and (6.95,1.4) .. (10.93,3.29)   ;
\draw  [dash pattern={on 0.84pt off 2.51pt}]  (286,162) -- (396,162) ;
\draw  [dash pattern={on 0.84pt off 2.51pt}]  (285.27,21) -- (395.27,21) ;
\draw  [dash pattern={on 0.84pt off 2.51pt}]  (315.01,40.6) -- (394,41) ;
\draw  [dash pattern={on 0.84pt off 2.51pt}]  (341,98) -- (394,98) ;
\draw [color={rgb, 255:red, 86; green, 189; blue, 216 }  ,draw opacity=1 ][line width=1.5]    (345.01,40.6) -- (371,98) ;
\draw  [dash pattern={on 0.84pt off 2.51pt}]  (242.67,146.86) -- (396,147) ;
\draw  [dash pattern={on 0.84pt off 2.51pt}]  (259.26,65.97) -- (395,66) ;
\draw  [fill={rgb, 255:red, 0; green, 0; blue, 0 }  ,fill opacity=1 ] (161.06,238) .. controls (161.06,237.17) and (161.73,236.5) .. (162.56,236.5) .. controls (163.39,236.5) and (164.06,237.17) .. (164.06,238) .. controls (164.06,238.83) and (163.39,239.5) .. (162.56,239.5) .. controls (161.73,239.5) and (161.06,238.83) .. (161.06,238) -- cycle ;
\draw  [dash pattern={on 0.84pt off 2.51pt}]  (101.02,176.43) -- (162.56,236.5) ;
\draw  [dash pattern={on 0.84pt off 2.51pt}]  (176.99,98.66) -- (238.53,158.73) ;
\draw  [dash pattern={on 0.84pt off 2.51pt}]  (184.57,126.38) -- (228.53,169.73) ;
\draw  [dash pattern={on 0.84pt off 2.51pt}]  (165.7,174.44) -- (195.35,203.38) ;
\draw  [dash pattern={on 0.84pt off 2.51pt}]  (85.86,143.7) -- (171.56,227.5) ;
\draw  [dash pattern={on 0.84pt off 2.51pt}]  (136.1,111.71) -- (212.02,185.84) ;
\draw  [fill={rgb, 255:red, 0; green, 0; blue, 0 }  ,fill opacity=1 ] (171.56,227.5) .. controls (171.56,226.67) and (172.23,226) .. (173.06,226) .. controls (173.89,226) and (174.56,226.67) .. (174.56,227.5) .. controls (174.56,228.33) and (173.89,229) .. (173.06,229) .. controls (172.23,229) and (171.56,228.33) .. (171.56,227.5) -- cycle ;
\draw  [fill={rgb, 255:red, 0; green, 0; blue, 0 }  ,fill opacity=1 ] (193.85,204.88) .. controls (193.85,204.05) and (194.52,203.38) .. (195.35,203.38) .. controls (196.17,203.38) and (196.85,204.05) .. (196.85,204.88) .. controls (196.85,205.71) and (196.17,206.38) .. (195.35,206.38) .. controls (194.52,206.38) and (193.85,205.71) .. (193.85,204.88) -- cycle ;
\draw  [fill={rgb, 255:red, 0; green, 0; blue, 0 }  ,fill opacity=1 ] (212.02,185.84) .. controls (212.02,185.01) and (212.69,184.34) .. (213.52,184.34) .. controls (214.35,184.34) and (215.02,185.01) .. (215.02,185.84) .. controls (215.02,186.67) and (214.35,187.34) .. (213.52,187.34) .. controls (212.69,187.34) and (212.02,186.67) .. (212.02,185.84) -- cycle ;
\draw  [fill={rgb, 255:red, 0; green, 0; blue, 0 }  ,fill opacity=1 ] (227.03,171.23) .. controls (227.03,170.4) and (227.7,169.73) .. (228.53,169.73) .. controls (229.36,169.73) and (230.03,170.4) .. (230.03,171.23) .. controls (230.03,172.05) and (229.36,172.73) .. (228.53,172.73) .. controls (227.7,172.73) and (227.03,172.05) .. (227.03,171.23) -- cycle ;
\draw  [fill={rgb, 255:red, 0; green, 0; blue, 0 }  ,fill opacity=1 ] (237.03,160.23) .. controls (237.03,159.4) and (237.7,158.73) .. (238.53,158.73) .. controls (239.36,158.73) and (240.03,159.4) .. (240.03,160.23) .. controls (240.03,161.05) and (239.36,161.73) .. (238.53,161.73) .. controls (237.7,161.73) and (237.03,161.05) .. (237.03,160.23) -- cycle ;
\draw  [dash pattern={on 4.5pt off 4.5pt}]  (164.06,238) -- (332,238) ;
\draw  [dash pattern={on 4.5pt off 4.5pt}]  (238.53,161.73) -- (239,334) ;
\draw  [fill={rgb, 255:red, 0; green, 0; blue, 0 }  ,fill opacity=1 ] (76.5,162) -- (82,172) -- (71,172) -- cycle ;
\draw  [fill={rgb, 255:red, 0; green, 0; blue, 0 }  ,fill opacity=1 ] (166.5,83) -- (172,93) -- (161,93) -- cycle ;
\draw  [dash pattern={on 4.5pt off 4.5pt}]  (213.52,187.34) -- (213.99,309.62) ;
\draw  [dash pattern={on 4.5pt off 4.5pt}]  (174.56,227.5) -- (320,227) ;
\draw [color={rgb, 255:red, 255; green, 255; blue, 255 }  ,draw opacity=1 ][line width=2.25]    (187.13,284.71) -- (251,350) ;
\draw [color={rgb, 255:red, 24; green, 97; blue, 150 }  ,draw opacity=1 ][line width=1.5]    (239,238) -- (332,238) ;
\draw [color={rgb, 255:red, 24; green, 97; blue, 150 }  ,draw opacity=1 ][line width=1.5]    (239,238) -- (239,334) ;
\draw [color={rgb, 255:red, 24; green, 97; blue, 150 }  ,draw opacity=1 ][line width=1.5]  [dash pattern={on 5.63pt off 4.5pt}]  (239,334) -- (332,334) ;
\draw [color={rgb, 255:red, 24; green, 97; blue, 150 }  ,draw opacity=1 ][line width=1.5]  [dash pattern={on 5.63pt off 4.5pt}]  (332,238) -- (332,334) ;
\draw [color={rgb, 255:red, 124; green, 102; blue, 202 }  ,draw opacity=1 ][line width=1.5]    (213,227) -- (320,227) ;
\draw [color={rgb, 255:red, 124; green, 102; blue, 202 }  ,draw opacity=1 ][line width=1.5]    (213,227) -- (213.99,309.62) ;
\draw [color={rgb, 255:red, 124; green, 102; blue, 202 }  ,draw opacity=1 ][line width=1.5]  [dash pattern={on 5.63pt off 4.5pt}]  (320,227) -- (320.99,309.62) ;
\draw [color={rgb, 255:red, 124; green, 102; blue, 202 }  ,draw opacity=1 ][line width=1.5]  [dash pattern={on 5.63pt off 4.5pt}]  (213.99,309.62) -- (320.99,309.62) ;

\draw (131.17,251.15) node [anchor=north west][inner sep=0.75pt]  [font=\footnotesize]  {$\bigstar $};
\draw (378,1.4) node [anchor=north west][inner sep=0.75pt]    {$f$};
\draw (237,30.4) node [anchor=north west][inner sep=0.75pt]    {$\mathbb{X}$};
\draw (396,150.4) node [anchor=north west][inner sep=0.75pt]    {$a_{1}$};
\draw (396,135.4) node [anchor=north west][inner sep=0.75pt]    {$a_{2}$};
\draw (397,88.4) node [anchor=north west][inner sep=0.75pt]    {$a_{3}$};
\draw (396,55.4) node [anchor=north west][inner sep=0.75pt]    {$a_{4}$};
\draw (396,31.4) node [anchor=north west][inner sep=0.75pt]    {$a_{5}$};
\draw (396,11.4) node [anchor=north west][inner sep=0.75pt]    {$a_{6}$};
\draw (84,156.4) node [anchor=north west][inner sep=0.75pt]    {$a_{1}$};
\draw (174,78.4) node [anchor=north west][inner sep=0.75pt]    {$a_{6}$};
\draw (401,255.4) node [anchor=north west][inner sep=0.75pt]    {$\mathbb{M}$};

\end{tikzpicture}

    \caption{The contravariant block decomposition (bottom) of the RISC functor $h(f)$ associated to the simplicial complex $\X$ with height function $f$ (top of the figure).}
    \label{fig:blocks_example}
\end{figure}
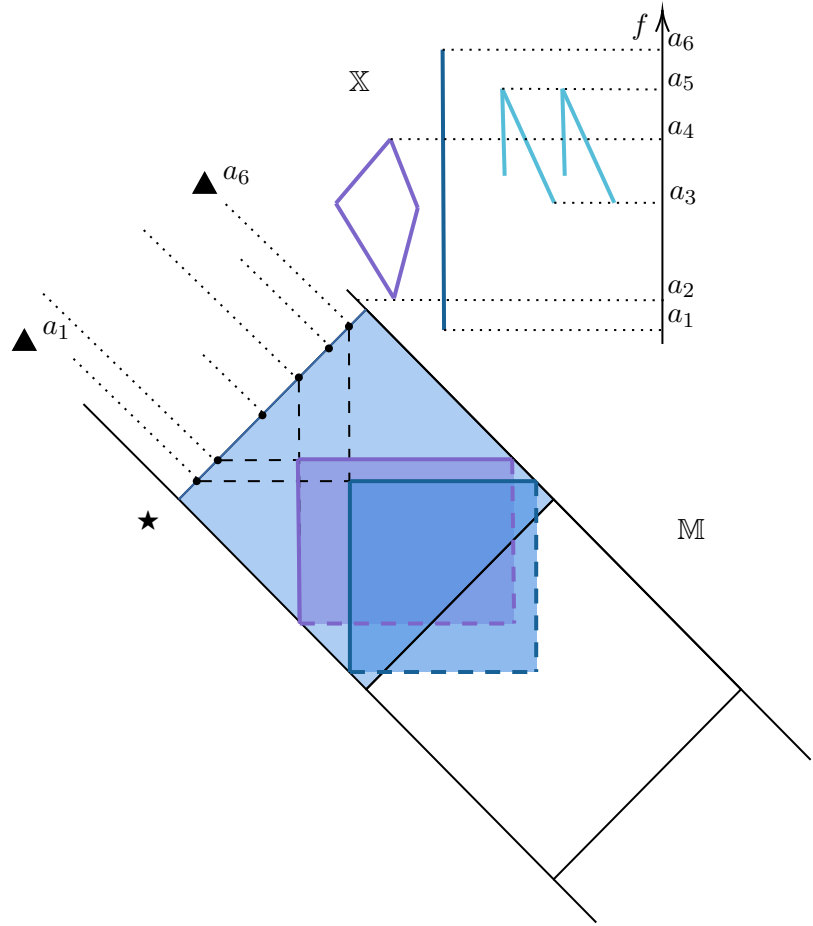

\subsection{Block decomposition}
\label{decomposition}

Let $f:\X\rightarrow\R$ be a continuous function with finite dimensional levelsets co-homology (\textit{i.e.} $\mathrm{dim}_K H^p(f^{-1}((a,b)))<\infty$ for $p\in\N$). As it is the case for the levelsets persistence modules $\mathcal{L}_p(f)$, the functor $h(f)$ is defined over a two-parameters poset, even though originating from a one-parameter filtration. Therefore, one shall identify an algebraic condition that expresses this fact, as it was the case for middle-exactness in the previous section. This will be achieved with the notions of cohomological functor and sequential continuity. 

\begin{definition}(\autocite[Definition~C.1]{bauer2021structure})
A functor $G : \M^{\mathrm{op}} \longrightarrow \mathrm{Vect}_K$ vanishing on $\partial \M$ is cohomological, if for all $u \preceq v \preceq w$, such that $v-u \in (1,0) \cdot \R$ and $w - v \in (0,1) \cdot \R$ (see figure \ref{fig:uvw}), the following sequence is exact:

\[ ... \longrightarrow G(T^{-1}(u)) \longrightarrow G(w) \longrightarrow G(v) \longrightarrow G(u) \longrightarrow G(T(u)) \longrightarrow ... \]

\end{definition}

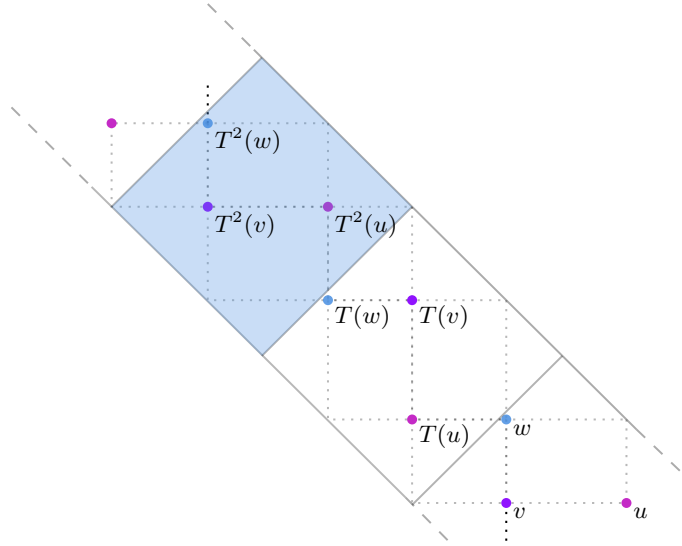
\begin{figure}

\tikzset{every picture/.style={line width=0.75pt}} 

\begin{center}

\tikzset{every picture/.style={line width=0.75pt}} 

\begin{tikzpicture}[x=0.75pt,y=0.75pt,yscale=-1,xscale=1]

\draw [color={rgb, 255:red, 0; green, 0; blue, 0 }  ,draw opacity=0.28 ]   (93,145) -- (158,209.5) ;
\draw [color={rgb, 255:red, 0; green, 0; blue, 0 }  ,draw opacity=0.33 ]   (174,76) -- (365,267) ;
\draw [color={rgb, 255:red, 0; green, 0; blue, 0 }  ,draw opacity=0.28 ]   (232.75,284.25) -- (249.75,301.25) ;
\draw [color={rgb, 255:red, 0; green, 0; blue, 0 }  ,draw opacity=0.33 ] [dash pattern={on 4.5pt off 4.5pt}]  (249.75,301.25) -- (271.25,323.25) ;
\draw [color={rgb, 255:red, 0; green, 0; blue, 0 }  ,draw opacity=0.33 ] [dash pattern={on 4.5pt off 4.5pt}]  (365,267) -- (391,292) ;
\draw [color={rgb, 255:red, 0; green, 0; blue, 0 }  ,draw opacity=0.33 ] [dash pattern={on 4.5pt off 4.5pt}]  (53,105) -- (93,145) ;
\draw [color={rgb, 255:red, 0; green, 0; blue, 0 }  ,draw opacity=0.33 ] [dash pattern={on 4.5pt off 4.5pt}]  (151,53) -- (174,76) ;
\draw  [color={rgb, 255:red, 0; green, 0; blue, 0 }  ,draw opacity=0.28 ][dash pattern={on 0.84pt off 2.51pt}] (300,202) -- (300,262) -- (253,262) -- (253,202) -- cycle ;
\draw  [color={rgb, 255:red, 0; green, 0; blue, 0 }  ,draw opacity=0.28 ][dash pattern={on 0.84pt off 2.51pt}] (360,262) -- (360,304) -- (300,304) -- (300,262) -- cycle ;
\draw [color={rgb, 255:red, 0; green, 0; blue, 0 }  ,draw opacity=0.28 ]   (158,209.5) -- (232.75,284.25) ;
\draw  [color={rgb, 255:red, 0; green, 0; blue, 0 }  ,draw opacity=0.28 ][dash pattern={on 0.84pt off 2.51pt}] (300,262) -- (300,304) -- (253,304) -- (253,262) -- cycle ;
\draw  [color={rgb, 255:red, 0; green, 0; blue, 0 }  ,draw opacity=0.28 ][dash pattern={on 0.84pt off 2.51pt}] (253,202) -- (253,262) -- (211,262) -- (211,202) -- cycle ;
\draw  [color={rgb, 255:red, 0; green, 0; blue, 0 }  ,draw opacity=0.28 ][dash pattern={on 0.84pt off 2.51pt}] (253,155) -- (253,202) -- (211,202) -- (211,155) -- cycle ;
\draw  [color={rgb, 255:red, 0; green, 0; blue, 0 }  ,draw opacity=0.28 ][dash pattern={on 0.84pt off 2.51pt}] (211,155) -- (211,202) -- (151,202) -- (151,155) -- cycle ;
\draw  [color={rgb, 255:red, 0; green, 0; blue, 0 }  ,draw opacity=0.28 ][dash pattern={on 0.84pt off 2.51pt}] (211,113) -- (211,155) -- (151,155) -- (151,113) -- cycle ;
\draw  [color={rgb, 255:red, 0; green, 0; blue, 0 }  ,draw opacity=0.28 ][dash pattern={on 0.84pt off 2.51pt}] (151,113) -- (151,155) -- (103,155) -- (103,113) -- cycle ;
\draw  [dash pattern={on 0.84pt off 2.51pt}]  (151,113) -- (151,94) ;
\draw  [dash pattern={on 0.84pt off 2.51pt}]  (300,323) -- (300,304) ;
\draw  [draw opacity=0][fill={rgb, 255:red, 196; green, 42; blue, 198 }  ,fill opacity=1 ] (357.5,304) .. controls (357.5,302.62) and (358.62,301.5) .. (360,301.5) .. controls (361.38,301.5) and (362.5,302.62) .. (362.5,304) .. controls (362.5,305.38) and (361.38,306.5) .. (360,306.5) .. controls (358.62,306.5) and (357.5,305.38) .. (357.5,304) -- cycle ;
\draw  [draw opacity=0][fill={rgb, 255:red, 74; green, 144; blue, 226 }  ,fill opacity=0.86 ] (148.5,113) .. controls (148.5,111.62) and (149.62,110.5) .. (151,110.5) .. controls (152.38,110.5) and (153.5,111.62) .. (153.5,113) .. controls (153.5,114.38) and (152.38,115.5) .. (151,115.5) .. controls (149.62,115.5) and (148.5,114.38) .. (148.5,113) -- cycle ;
\draw  [draw opacity=0][fill={rgb, 255:red, 74; green, 144; blue, 226 }  ,fill opacity=0.86 ] (208.5,202) .. controls (208.5,200.62) and (209.62,199.5) .. (211,199.5) .. controls (212.38,199.5) and (213.5,200.62) .. (213.5,202) .. controls (213.5,203.38) and (212.38,204.5) .. (211,204.5) .. controls (209.62,204.5) and (208.5,203.38) .. (208.5,202) -- cycle ;
\draw  [draw opacity=0][fill={rgb, 255:red, 74; green, 144; blue, 226 }  ,fill opacity=0.86 ] (297.5,262) .. controls (297.5,260.62) and (298.62,259.5) .. (300,259.5) .. controls (301.38,259.5) and (302.5,260.62) .. (302.5,262) .. controls (302.5,263.38) and (301.38,264.5) .. (300,264.5) .. controls (298.62,264.5) and (297.5,263.38) .. (297.5,262) -- cycle ;
\draw  [draw opacity=0][fill={rgb, 255:red, 144; green, 19; blue, 254 }  ,fill opacity=1 ] (297.5,304) .. controls (297.5,302.62) and (298.62,301.5) .. (300,301.5) .. controls (301.38,301.5) and (302.5,302.62) .. (302.5,304) .. controls (302.5,305.38) and (301.38,306.5) .. (300,306.5) .. controls (298.62,306.5) and (297.5,305.38) .. (297.5,304) -- cycle ;
\draw  [draw opacity=0][fill={rgb, 255:red, 144; green, 19; blue, 254 }  ,fill opacity=1 ] (148.5,155) .. controls (148.5,153.62) and (149.62,152.5) .. (151,152.5) .. controls (152.38,152.5) and (153.5,153.62) .. (153.5,155) .. controls (153.5,156.38) and (152.38,157.5) .. (151,157.5) .. controls (149.62,157.5) and (148.5,156.38) .. (148.5,155) -- cycle ;
\draw  [draw opacity=0][fill={rgb, 255:red, 144; green, 19; blue, 254 }  ,fill opacity=1 ] (250.5,202) .. controls (250.5,200.62) and (251.62,199.5) .. (253,199.5) .. controls (254.38,199.5) and (255.5,200.62) .. (255.5,202) .. controls (255.5,203.38) and (254.38,204.5) .. (253,204.5) .. controls (251.62,204.5) and (250.5,203.38) .. (250.5,202) -- cycle ;
\draw  [draw opacity=0][fill={rgb, 255:red, 196; green, 42; blue, 198 }  ,fill opacity=1 ] (100.5,113) .. controls (100.5,111.62) and (101.62,110.5) .. (103,110.5) .. controls (104.38,110.5) and (105.5,111.62) .. (105.5,113) .. controls (105.5,114.38) and (104.38,115.5) .. (103,115.5) .. controls (101.62,115.5) and (100.5,114.38) .. (100.5,113) -- cycle ;
\draw  [draw opacity=0][fill={rgb, 255:red, 196; green, 42; blue, 198 }  ,fill opacity=1 ] (208.5,155) .. controls (208.5,153.62) and (209.62,152.5) .. (211,152.5) .. controls (212.38,152.5) and (213.5,153.62) .. (213.5,155) .. controls (213.5,156.38) and (212.38,157.5) .. (211,157.5) .. controls (209.62,157.5) and (208.5,156.38) .. (208.5,155) -- cycle ;
\draw  [draw opacity=0][fill={rgb, 255:red, 196; green, 42; blue, 198 }  ,fill opacity=1 ] (250.5,262) .. controls (250.5,260.62) and (251.62,259.5) .. (253,259.5) .. controls (254.38,259.5) and (255.5,260.62) .. (255.5,262) .. controls (255.5,263.38) and (254.38,264.5) .. (253,264.5) .. controls (251.62,264.5) and (250.5,263.38) .. (250.5,262) -- cycle ;
\draw [color={rgb, 255:red, 0; green, 0; blue, 0 }  ,draw opacity=0.3 ]   (178,230) -- (253,155) ;
\draw [color={rgb, 255:red, 0; green, 0; blue, 0 }  ,draw opacity=0.3 ]   (253,304.99) -- (328,230) ;
\draw [color={rgb, 255:red, 0; green, 0; blue, 0 }  ,draw opacity=0.3 ]   (103.01,155) -- (178,80) ;
\draw  [draw opacity=0][fill={rgb, 255:red, 74; green, 144; blue, 226 }  ,fill opacity=0.3 ] (103.02,155.01) -- (178,80) -- (253.01,154.98) -- (178.03,229.98) -- cycle ;

\draw (362,304.9) node [anchor=north west][inner sep=0.75pt]  [font=\footnotesize]  {$u$};
\draw (302,304.9) node [anchor=north west][inner sep=0.75pt]  [font=\footnotesize]  {$v$};
\draw (302,262.9) node [anchor=north west][inner sep=0.75pt]  [font=\footnotesize]  {$w$};
\draw (255,262.9) node [anchor=north west][inner sep=0.75pt]  [font=\footnotesize]  {$T( u)$};
\draw (255,202.9) node [anchor=north west][inner sep=0.75pt]  [font=\footnotesize]  {$T( v)$};
\draw (213,202.9) node [anchor=north west][inner sep=0.75pt]  [font=\footnotesize]  {$T( w)$};
\draw (153,113.9) node [anchor=north west][inner sep=0.75pt]  [font=\footnotesize]  {$T^{2}( w)$};
\draw (153,155.9) node [anchor=north west][inner sep=0.75pt]  [font=\footnotesize]  {$T^{2}( v)$};
\draw (213,155.9) node [anchor=north west][inner sep=0.75pt]  [font=\footnotesize]  {$T^{2}( u)$};

\end{tikzpicture}

\end{center}

    \caption{The subposet generated by the orbits of $u,v,w$. The blue area is a copy of the fundamental domain.}
    \label{fig:uvw}
\end{figure}

\begin{definition}(\autocite[Definition~2.4]{bauer2021structure}) A contravariant functor $G:\M^{\mathrm{op}}\rightarrow\mathrm{Vect}_K$ is \emph{sequentially continuous}, if for any increasing sequence $(m_k)_{k\in\N}\subset\M$ converging to $m\in\M$, the natural map $$G(m)\rightarrow\lim_{\xleftarrow[k]{}} G(m_k)$$ is an isomorphism.

\end{definition}

\begin{proposition}(\autocite[Proposition~2.5]{bauer2021structure})
\label{prop_seq}
Let $f : \X \longrightarrow \R$ be a continuous function. The contravariant functor $$h(f):\M^{\mathrm{op}}\rightarrow\mathrm{Vect}_K$$ is cohomological. Moreover, if it is pointwise finite dimensional, then it is sequentially continuous.
\end{proposition}

\begin{definition}(\autocite[Definition~3.2]{bauer2021structure}) For any $v\in\mathrm{int}(\M)$, the \emph{contravariant block} $B_v$ is defined as
\begin{equation*}
    B_v:\M^{\mathrm{op}}\rightarrow\mathrm{Vect}_K,m\mapsto\begin{cases}
    K \text{ if } m\in (\downarrow m)\cap \mathrm{int}(\uparrow T^{-1}(m)),\\
    \{0\} \text{ otherwise}.
    \end{cases}
\end{equation*}
with identity maps connecting any two non-zero vector spaces.
\end{definition}

\begin{theorem}(\autocite[Corollary~3.5]{bauer2021structure})
\label{decomposition_thm}
 Let $G:\M^{\mathrm{op}}\rightarrow\mathrm{Vect}_K$ be a sequentially continuous, pointwise finite dimensional, cohomological functor. There is a contravariant block decomposition $$G\cong \bigoplus_{v\in\mathrm{int}(\M)} (B_v)^{\oplus\mathrm{deg}(v)},$$
where $\nu=\mathrm{Dgm}(G)$.
\end{theorem}

\begin{corollary}
Let $f:\mathbb{X}\longrightarrow \R$ be a continuous function with pointwise finite dimensional levelsets cohomology. There is a contravariant bock decomposition $$h(f)\cong \bigoplus_{v\in\mathrm{int}(\M)} (B_v)^{\oplus\mu(v)},$$
where $\mu=\mathrm{Dgm}(f)$.
\end{corollary}

\begin{proof}
This follows from Theorem \ref{decomposition_thm} and Proposition \ref{prop_seq}.
\end{proof}

\subsection{Extracting a barcode}
\label{extract}

This section focuses on showing how to extract the \textit{extended persistence barcode} from the information contained in an extended persistence diagram $\mu:\mathrm{int}(\M)\rightarrow\N_0$. The first thing to note is that all the information needed lies in how the injected real line $\bigstar$ and its copies along $\M$ intersect with the fundamental blocks of $\mu$. Now, to complete the reasoning, it remains to naturally extract a barcode from those intersections, which can form four types of intervals.\\

\noindent \textbf{Notation.} Let $2^\R$ be the set of all subsets of $\R$, and for any $m\in\M$, let $\mathrm{deg}(m)$ denote the only integer $p\in\N$ such that $T^p(m)\in D$. Consider the map $\rho=(\rho_1,\rho_2)$ defined in Definition \ref{rho_def} and let $\mathrm{int}(\R)$ denote the set of intervals in $\R$. Define the map 
$$\sigma:\{(X,Y)\in 2^\R\times2^\R\mid Y\subset X\}\longrightarrow 2^\R\text{ }: \text{ }(X,Y)\mapsto X\setminus Y.$$

\begin{definition}
We define the map $I(\cdot):\M\rightarrow 2^\R:m\mapsto (\sigma\circ \rho\circ T^{\mathrm{deg}(m)})(m)$. 
\end{definition}

\begin{remark}
The map $I(\cdot)$ can be splitted into three steps. Consider a point $m\in M$. 
\begin{enumerate}
    \item[$\bullet$] First, it sends $m\in\M$ in the fundamental domain $D$ through repetitively applying $T:\M\rightarrow\M$ (one loses track of the degree $\mathrm{deg}(m)$ of $m$).
    \item[$\bullet$] Second, it attributes a pair of intervals through the map $\rho$. Each of the four regions of the pyramid drawn by $D$ gives a specific kind of pairs of intervals.
    \item[$\bullet$] Third, the map converts an element of $2^\R\times 2^\R$ to an element of $2^\R$ without losing any information. This dimensionality reduction is the key part to extracting a barcode of the form $\Z\times\mathrm{Int}(\R)\rightarrow\N_0$, which is what is looked at next.
\end{enumerate} 
\end{remark}

\begin{proposition}
The map $\Psi:\M\rightarrow \N_0\times\mathrm{Int}(\R):m\mapsto (\mathrm{deg}(m),I(m))$ is a bijection.
\end{proposition}

\begin{proof}
This essentially follows from the fact that, apart from the degree $\mathrm{deg}(\cdot)$, through any of the steps of $I(\cdot)=(\sigma\circ \rho\circ T^{\mathrm{deg}(\cdot)})(\cdot)$, the object considered is entirely determined by two real numbers $a,b\in\R$. Now, if one keeps track of the degree by integrating $\mathrm{deg}(\cdot):\M\rightarrow\Z$ to those steps, a bijection is obtained.
\end{proof}

\noindent Now, this tells us that a point $m\in\M$ corresponds to an interval $I\in\mathrm{Int}(\R)$ together with a dimension $n$. Hence, the information contained in an extended persistence diagram $\mathrm{Dgm}(f):\mathrm{int}(\M)\rightarrow\N_0$ is contained in the corresponding barcode given by 
\begin{equation*}
    \mathbb{B}(f):
    \begin{cases}
    \N_0\times\mathrm{Int}(\R)\rightarrow\N_0\\
    (n,I)\mapsto \mathrm{Dgm}(f)(\Psi^{-1}((n,I))),
    \end{cases}
\end{equation*}
that attributes a multiplicity to an interval $I$ of dimension $n$ through $\mathrm{Dgm}(f)$. 

\begin{figure}
    \centering
    
    \tikzset{every picture/.style={line width=0.75pt}} 

\begin{tikzpicture}[x=0.75pt,y=0.75pt,yscale=-1,xscale=1]

\draw  [color={rgb, 255:red, 74; green, 144; blue, 226 }  ,draw opacity=1 ] (76.02,8.03) -- (134,66.02) -- (76.02,66.02) -- cycle ;
\draw   (134,66.02) -- (191.98,124) -- (134,181.98) -- (76.02,124) -- cycle ;
\draw  [color={rgb, 255:red, 74; green, 144; blue, 226 }  ,draw opacity=1 ][fill={rgb, 255:red, 0; green, 0; blue, 0 }  ,fill opacity=0.2 ] (76.02,8.03) -- (134,66.02) -- (76.02,124) -- (18.03,66.02) -- cycle ;
\draw    (2,50) -- (140,188) ;
\draw    (70,2) -- (210,142) ;
\draw  [color={rgb, 255:red, 74; green, 144; blue, 226 }  ,draw opacity=1 ] (76.02,124) -- (18.03,66.02) -- (76.02,66.02) -- cycle ;
\draw    (187,66) -- (259,66) ;
\draw [shift={(261,66)}, rotate = 180] [color={rgb, 255:red, 0; green, 0; blue, 0 }  ][line width=0.75]    (10.93,-3.29) .. controls (6.95,-1.4) and (3.31,-0.3) .. (0,0) .. controls (3.31,0.3) and (6.95,1.4) .. (10.93,3.29)   ;
\draw [shift={(187,66)}, rotate = 180] [color={rgb, 255:red, 0; green, 0; blue, 0 }  ][line width=0.75]    (0,5.59) -- (0,-5.59)   ;
\draw  [dash pattern={on 0.84pt off 2.51pt}]  (88.02,94.02) -- (88,112) ;
\draw  [dash pattern={on 0.84pt off 2.51pt}]  (88.02,94.02) -- (106,94) ;
\draw  [color={rgb, 255:red, 0; green, 0; blue, 0 }  ,draw opacity=1 ][fill={rgb, 255:red, 0; green, 0; blue, 0 }  ,fill opacity=1 ] (86.52,92.52) .. controls (86.52,91.69) and (87.19,91.02) .. (88.02,91.02) .. controls (88.85,91.02) and (89.52,91.69) .. (89.52,92.52) .. controls (89.52,93.35) and (88.85,94.02) .. (88.02,94.02) .. controls (87.19,94.02) and (86.52,93.35) .. (86.52,92.52) -- cycle ;
\draw  [fill={rgb, 255:red, 0; green, 0; blue, 0 }  ,fill opacity=1 ] (86.5,112) .. controls (86.5,111.17) and (87.17,110.5) .. (88,110.5) .. controls (88.83,110.5) and (89.5,111.17) .. (89.5,112) .. controls (89.5,112.83) and (88.83,113.5) .. (88,113.5) .. controls (87.17,113.5) and (86.5,112.83) .. (86.5,112) -- cycle ;
\draw  [fill={rgb, 255:red, 0; green, 0; blue, 0 }  ,fill opacity=1 ] (104.5,94) .. controls (104.5,93.17) and (105.17,92.5) .. (106,92.5) .. controls (106.83,92.5) and (107.5,93.17) .. (107.5,94) .. controls (107.5,94.83) and (106.83,95.5) .. (106,95.5) .. controls (105.17,95.5) and (104.5,94.83) .. (104.5,94) -- cycle ;
\draw  [color={rgb, 255:red, 74; green, 144; blue, 226 }  ,draw opacity=1 ] (76.02,8.03) -- (134,66.02) -- (76.02,66.02) -- cycle ;
\draw    (351,66) -- (423,66) ;
\draw [shift={(425,66)}, rotate = 180] [color={rgb, 255:red, 0; green, 0; blue, 0 }  ][line width=0.75]    (10.93,-3.29) .. controls (6.95,-1.4) and (3.31,-0.3) .. (0,0) .. controls (3.31,0.3) and (6.95,1.4) .. (10.93,3.29)   ;
\draw [shift={(351,66)}, rotate = 180] [color={rgb, 255:red, 0; green, 0; blue, 0 }  ][line width=0.75]    (0,5.59) -- (0,-5.59)   ;
\draw    (188,45) .. controls (257.65,0.22) and (345.12,0.99) .. (422.83,44.34) ;
\draw [shift={(424,45)}, rotate = 209.43] [color={rgb, 255:red, 0; green, 0; blue, 0 }  ][line width=0.75]    (10.93,-3.29) .. controls (6.95,-1.4) and (3.31,-0.3) .. (0,0) .. controls (3.31,0.3) and (6.95,1.4) .. (10.93,3.29)   ;
\draw [shift={(188,45)}, rotate = 507.26] [color={rgb, 255:red, 0; green, 0; blue, 0 }  ][line width=0.75]    (0,5.59) -- (0,-5.59)   ;

\draw (52,73.4) node [anchor=north west][inner sep=0.75pt]    {$D$};
\draw (142,191.4) node [anchor=north west][inner sep=0.75pt]    {$l_{1}$};
\draw (212,145.4) node [anchor=north west][inner sep=0.75pt]    {$l_{2}$};
\draw (79,74.4) node [anchor=north west][inner sep=0.75pt]    {$m$};
\draw (216,41.4) node [anchor=north west][inner sep=0.75pt]    {$\rho $};
\draw (90,113.9) node [anchor=north west][inner sep=0.75pt]    {$\blacktriangle a$};
\draw (108,95.9) node [anchor=north west][inner sep=0.75pt]    {$\blacktriangle b$};
\draw (270,57.4) node [anchor=north west][inner sep=0.75pt]    {$([ a,b] ,\emptyset )$};
\draw (378,40.4) node [anchor=north west][inner sep=0.75pt]    {$\sigma $};
\draw (431,57.4) node [anchor=north west][inner sep=0.75pt]    {$[ a,b]$};
\draw (288,15.4) node [anchor=north west][inner sep=0.75pt]    {$I( \cdot )$};
\draw (110,132.4) node [anchor=north west][inner sep=0.75pt]    {$T^{-1}( D)$};

\end{tikzpicture}

    \caption{Illustrative example of the map $I(\cdot)$.}
    \label{fig:my_label}
\end{figure}
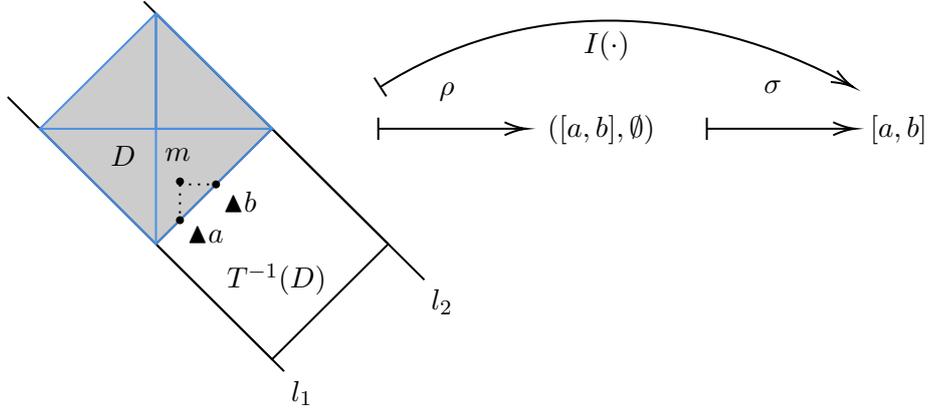

\subsection{Bottleneck distance and universality}

In \cite{bauer2020universality}, the authors propose a way to define a Bottleneck distance between RISC extended persistence diagrams, and prove that in this context, their bottleneck distance satisfy a universality property. \\

\noindent Consider the strip $\M$ built as if $l_1$ intersects the $x$-axis at $(-\pi,0)$ and $l_2$ intersects it at $(\pi,0)$. Let $d_0:\R\times\R\rightarrow\R^+\cup\{\infty\}$ be the unique extended metric satisfying 
\begin{equation*}
    d_0(s,t)=
    \begin{cases}
    \tan t - \tan s \text{ if } [s,t]\cap (\frac{\pi}{2}+\pi\Z)=\emptyset,\\
    \infty \text{ otherwise.}
    \end{cases}
\end{equation*} 

\noindent Now, one can define a metric on $\R\times\R$ based on $d_0$ as follows. 

\begin{equation*}
    d_0^{\R\times\R}:
    \begin{cases}
    (\R\times\R)\times (\R\times\R)\rightarrow \R^+\cup\{\infty\}\\
    ((s_1,s_2),(t_1,t_2))\mapsto \max_{i=1,2} d_0(s_i,t_i)
    \end{cases}
\end{equation*}

\begin{definition}[Metric $d$ on the strip $\M$]
Let $d$ be the following map defined  on $\M\times \M$:
\begin{equation*}
d:\begin{cases}
\M\times\M\rightarrow\R^+\cup\{\infty\}\\
(s,t)\mapsto d_0^{\R\times\R}|_{\M\times\M}(s,t)
\end{cases}
\end{equation*}
\end{definition}

\begin{proposition}[\cite{bauer2020universality}]
The map $d$ is an extended pseudo metric on $\M$.
\end{proposition}

\noindent We present three lemmas that help getting a better understanding of how the metric $d:\M\times\M\rightarrow\R^+\cup\{\infty\}$ acts. Before formulating those results, we introduce a notation that associates to a point of the interior of the strip a specific region of $\M$.\\

\noindent \textbf{Notation.} Let $m\in\M$. We denote by $R_m$ the triangular region in which $m$ is situated. For example, the fundamental domain $D$ consists of four closed triangular regions (the north, east, south and west faces of the pyramid), whose interiors don't intersect. Once again, let $\mathrm{deg}(m)$ denote the only integer $p\in\N$ such that $T^p(m)\in D$.

\begin{lemma}
For all $s,t\in\mathrm{int}(\M)$ such that $\mathrm{deg}(s)=\mathrm{deg}(t)$, we have $$d(s,t)=\infty\iff \mathrm{int}(R_s)\cap \mathrm{int}(R_t)=\emptyset.$$
\end{lemma}

\begin{proof}
This follows from the fact that for $s,t\in\mathrm{int}(\M)$, one has $d(s,t)=\infty$ if and only if either $[s_1,t_1]\cap (\frac{\pi}{2}+\pi\Z)\neq \emptyset$ or $[s_2,t_2]\cap (\frac{\pi}{2}+\pi\Z)\neq \emptyset$. Indeed, the first condition above means that amongst $s$ and $t$, one is in the upper part of its domain and the other one is in the bottom part. Similarly, the second condition means that one of them is in the left part of its domain and the other one is in the right part.
\end{proof}

\begin{lemma}
For all $s,t\in\mathrm{int}(\M)$ such that $|\mathrm{deg}(s)-\mathrm{deg}(t)|\geq 2$, one has $d(s,t)=\infty$.
\end{lemma}

\begin{lemma}
\label{cool_case}
For all $s,t\in\mathrm{int}(\M)$ such that $|\mathrm{deg}(s)- \mathrm{deg}(t)|=1$, one has $d(s,t)\neq\infty$ if and only if $|R_s\cap R_t|>1$, \textit{i.e.} if and only if the closed triangular regions $R_s$ and $R_t$ share a copy of the injected real line $\bigstar=\mathrm{Im}_\blacktriangle(\Bar{R})$.
\end{lemma}

\noindent Based on the introduced metric, one may now express the notion of \textit{Bottleneck distance} between extended persistence diagrams. A $\delta$-matching between two extended persistence diagrams is a partial matching of their vertices such that any two matched vertices $s,t\in\M$ satisfy $d(s,t)\leq\delta$ and unmatched vertices are at distance at most $\delta$ of the boundary $\partial\M$. For $\delta>0$, one writes $\mathcal{M}(\delta)$ for the set of pairs $(\mu_1,\mu_2)$ of extended persistence diagrams for which there exists a $\delta$-matching.

\begin{definition}[Bottleneck distance]
The \emph{Bottleneck distance} between two extended persistence diagrams $\mu_1:\mathrm{int}(\M)\rightarrow \N_0$ and $\mu_2:\mathrm{int}(\M)\rightarrow \N_0$ is defined as $$d_B(\mu_1,\mu_2)=\inf\{\delta>0\mid (\mu_1,\mu_2)\in\mathcal{M}(\delta)\}.$$
\end{definition}

The following is implicitly contained in \cite{bauer2020universality}. 

\begin{proposition}

Let $f,g : \mathbb{X} \longrightarrow \R$ be two continuous functions such that $h(f)$ and $h(g)$ are pointwise finite dimensional. Then:

\[d_B(\mathrm{Dgm}(f),\mathrm{Dgm(g)}) \leq \|f-g\|_\infty.\]
We say that the bottleneck distance is stable.
\end{proposition}

\noindent The main result of \cite{bauer2020universality} states that the Bottleneck distance is universal, meaning that it is the largest possible stable
distance on realizable persistence diagrams.

\begin{definition}(Realizable persistence diagram) A map $\mu:\mathrm{int}(\M)\rightarrow\N_0$ 
is a \emph{realizable persistence diagram} if $\mathrm{Dgm}(f)=\mu$ for some PL function $f:\X\rightarrow\R$. 
\end{definition}

\begin{theorem}(\autocite[Theorem~1.1]{bauer2020universality})
For realizable persistence diagrams $\mu$ and $\nu$ with $d_B(\mu, \nu)<\infty$, there
exists a finite simplicial complex $\X$ and piecewise linear functions $f,g : \X\rightarrow\R$ with
$$\mathrm{Dgm}(f)=\mu, \text{ } \mathrm{Dgm}(g)=\nu, \text{ and } \lVert f-g \rVert_\infty = d_B(\mu, \nu).$$
\end{theorem}

\begin{corollary}
Let $d$ be a stable distance on the set of realizable extended persistence diagrams. Then for all realizable extended persistence diagrams $\mu, \nu$ : $d(\mu, \nu) \leq d_B(\mu, \nu)$.
\end{corollary}

\section{Connection with Levelsets Zigzag (co-)Homology}
\label{connection}

\subsection{Connection with the pyramid}
\label{pyr_connection}
For a pair $(\X,f)$, the strip $\M$ introduced in \cite{bauer2021structure} is closely related to the pyramid associated to $(\X,f)$ presented in Section \ref{pyr1}, which we denote by $\square_f$. If we denote by $H^p(\square_f)$ the diagram obtained by applying the functor $H^p(\cdot)$ to the pyramid $\square$, then $\M$ can be thought as an infinite continuous gluing of all diagrams $H^p(\square_f)$ for $p\in\N$. 

\begin{proposition}
\label{connection_pyramids_prop}
Let $(\X,f)$ be a pair of Morse type. There is an inclusion of posets from the pyramid $\square_f$ to a sub-lattice embedded in the fundamental domain $D$ of $\M$. 
\end{proposition}

\begin{proof}
Let $a_1<...<a_n$ be critical values of $f:\X\rightarrow\R$, with in-between regular values $-\infty<s_0<a_1<s_1<\cdots<s_{n-1}<a_n<s_n<\infty$. Note that the pyramid $\square_f$ can be divided into four triangular regions, which are the faces of the pyramid. Now, each region contains a specific kind of relative or absolute space. Indeed, the south face contains spaces of the form $\X_i^j$ for $i\leq j$, which correspond to intervals of the form $[s_i,s_j]$. The north face contains pairs of the form $(\X_0^n,\X_0^i\cup\X_j^n)$, which correspond to pairs of the form $(\R,\R\setminus [s_i,s_j])$. The east face is made of pairs $(\X_0^i,\X_0^j)$ corresponding to pairs $((-\infty,s_i),(-\infty,s_j))$, and the west face is made of pairs $(\X_i^n,\X_j^n)$ corresponding to pairs $((s_i,\infty),(s_j,\infty))$. We conclude the proof by looking at the schematic functioning of the map $\rho:\M\rightarrow \mathrm{Op}(\mathbb{R})\times\mathrm{Op}(\mathbb{R})$ shown in Figure \ref{fig:rho_map}.
\end{proof}

\subsection{Extracting levelsets zigzag persistence}

Binding the observations of Section \ref{extract} and Section \ref{pyr_connection}, one can formulate the main result, which describes how on can retrieve the levelsets zigzag barcode of a pair $(\X,f)$ from the extended persistence diagram $\mathrm{Dgm}(f)$ defined on the strip $\M$.

\begin{proposition}[Extracting levelsets zigzag persistence]
The barcode map $$\mathbb{B}(f):\N_0\times\mathrm{Int}(\R)\rightarrow\N_0$$ defined in Section \ref{extract} is the levelsets zigzag persistence barcode of $(\X,f)$.
\end{proposition}

\begin{proof}[Idea.]
By construction, the south edge of the pyramid $\square_f$ determines the levelsets zigzag persistence of $(\X,f)$. Moreover, one knows from Proposition \ref{connection_pyramids_prop} that this edge corresponds to one of the copies of the injected line $\bigstar=\mathrm{Im}_\blacktriangle(\Bar{R})$ (depending on the dimension). Let $\bigstar_n$ denote the copy that corresponds to the south edge of the domain $T^{-n}(D)$. Then the information about the dimension-$n$ intervals of the levelsets zigzag barcode of $(\X,f)$ is entirely contained in $\bigstar_n$. Now, what the map $I(\cdot)$ does is mapping a point $m\in T^{-n}(D)$ to an interval in to $\bigstar_n$ that corresponds to an indecomposable block generated by $m$.
\end{proof}

\section{Computational aspects}
\label{algorithm}

In this section, we draw the computational consequences of the several links we have presented between the different flavors of persistence. We also review recent algorithms that achieve state-of-the-art complexity to compute extended persistence.

\subsection{Computing levelsets Zigzag Persistence}
\label{s:levelsetszigzag}

This section is devoted to explicit the bijections between the barcodes obtained with the different flavors of persistence we have presented. Let $a_1<...<a_n$ be the critical values of the Morse type function $f:\X\rightarrow\R$, and in-between regular values $-\infty<s_0<a_1<s_1<\cdots<s_{n-1}<a_n<s_n<\infty$. Recall that for $i \leq j$, we set $\mathbb{X}_i^j := f^{-1}([s_i,s_j])$.

 Instead of dealing with intervals of spaces, one can look at intervals of critical values and reformulate the diagram bijection theorem (\textit{cf} Theorem \ref{bijection_alternative}).\\

\noindent \textbf{Levelsets zigzag persistence.} Regarding the levelsets zigzag persistence, one can define the following correspondence. Every interval is either closed, open or half-open.
\begin{enumerate}
    \item $[\X_{i-1}^i,\X_{j-1}^j]\leftrightarrow [a_i,a_j]$ for $i,j\in \{1,...,n\}$
    \item $[\X_{i-1}^i,\X_{j-1}^{j-1}]\leftrightarrow [a_i,a_j)$ for $i<j\in \{1,...,n+1\}$
    \item $[\X_i^i,\X_{j-1}^j]\leftrightarrow (a_i,a_j]$ for $i<j\in \{0,...,n\}$
    \item $[\X_i^i,\X_{j-1}^{j-1}]\leftrightarrow (a_i,a_j)$ for $i<j\in \{0,...,n+1\}$
\end{enumerate}

\noindent \textbf{Extended persistence.} As for the extended persistence, one can define the notation change as follows. Intervals have four different forms (for \textit{Types I} to \textit{IV}).
\begin{enumerate}
    \item $[\X_0^i,\X_0^{j-1}]\leftrightarrow [a_i,a_j)$ for $i,j\in \{1,...,n\}$
    \item $[(\X_0^n,\X_{j-1}^n),(\X_0^n,\X_i^n)]^+\leftrightarrow [\Bar{a}_j,\Bar{a}_i)^+$ for $i<j\in \{1,...,n+1\}$
    \item $[\X_0^i,(\X_0^n,\X_j^n)]\leftrightarrow [a_i,\Bar{a}_j)$ for $i<j\in \{0,...,n\}$
    \item $[\X_0^j,(\X_0^n,\X_i^n)]^+\leftrightarrow [a_i,\Bar{a}_j)^+$ for $i<j\in \{0,...,n+1\}$
\end{enumerate}

\begin{theorem}[Diagram Bijection, new version]
\label{bijection_alternative}
One has the following correspondence between the intervals of extended persistence (left) and levelsets zigzag persistence (right).

\begin{center}
\begin{tabular}{|c c c|} 
 \hline
 Type & Extended & Levelsets zigzag\\ [0.5ex] 
 \hline
 I $(i<j)$ & $[a_i,a_j)$ & $[a_i,a_j)$ \\ 
\hline
 II $(i< j)$ & $[\Bar{a}_j,\Bar{a}_i)^+$ & $(a_i,a_j]$ \\
 \hline
 III  $(i\leq j)$ & $[a_i,\Bar{a}_j)$ & $[a_i,a_j]$\\
\hline
 IV $(i<j)$ & $[a_j,\Bar{a}_i)^+$ & $(a_i,a_j)$ \\
 \hline
\end{tabular}
\end{center}

\end{theorem}

\begin{proof}
This is another formulation of Theorem \ref{bijection_theorem}.
\end{proof}

\noindent Finally, the algorithm to compute the levelsets zigzag persistence of Morse type function from its extended persistence simply corresponds to  Theorem \ref{bijection_alternative}, as shown in Figure \ref{pipeline}'s pipeline.\\

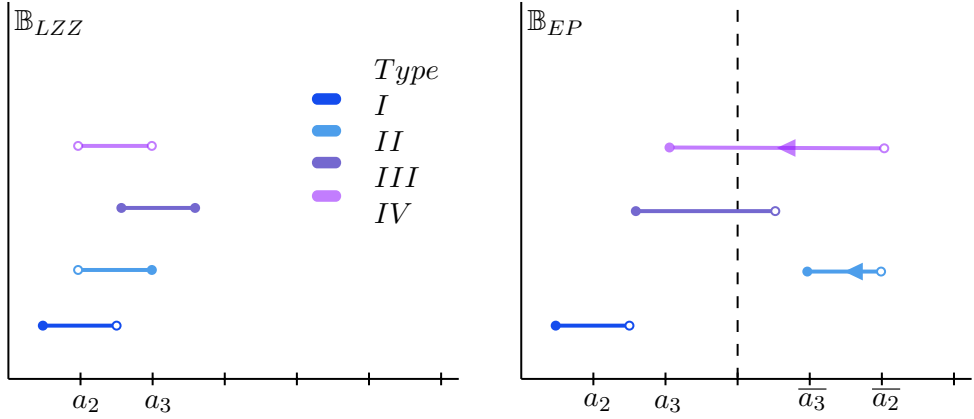
\begin{figure}
    \centering
    
    \tikzset{every picture/.style={line width=0.75pt}} 

\tikzset{every picture/.style={line width=0.75pt}} 

\begin{tikzpicture}[x=0.60pt,y=0.60pt,yscale=-1,xscale=1]

\draw    (44,245) -- (325,245) (89,241) -- (89,249)(134,241) -- (134,249)(179,241) -- (179,249)(224,241) -- (224,249)(269,241) -- (269,249)(314,241) -- (314,249) ;
\draw    (44,245) -- (44,8) ;
\draw    (364,245) -- (645,245) (409,241) -- (409,249)(454,241) -- (454,249)(499,241) -- (499,249)(544,241) -- (544,249)(589,241) -- (589,249)(634,241) -- (634,249) ;
\draw    (364,245) -- (364,8) ;
\draw  [dash pattern={on 4.5pt off 4.5pt}]  (499,13) -- (499,244) ;
\draw [color={rgb, 255:red, 77; green, 158; blue, 235 }  ,draw opacity=1 ][line width=1.5]    (90,176.5) -- (131,176.5) ;
\draw  [color={rgb, 255:red, 77; green, 158; blue, 235 }  ,draw opacity=1 ][fill={rgb, 255:red, 77; green, 158; blue, 235 }  ,fill opacity=1 ] (131,176.5) .. controls (131,175.12) and (132.12,174) .. (133.5,174) .. controls (134.88,174) and (136,175.12) .. (136,176.5) .. controls (136,177.88) and (134.88,179) .. (133.5,179) .. controls (132.12,179) and (131,177.88) .. (131,176.5) -- cycle ;
\draw  [color={rgb, 255:red, 77; green, 158; blue, 235 }  ,draw opacity=1 ] (85,176.5) .. controls (85,175.12) and (86.12,174) .. (87.5,174) .. controls (88.88,174) and (90,175.12) .. (90,176.5) .. controls (90,177.88) and (88.88,179) .. (87.5,179) .. controls (86.12,179) and (85,177.88) .. (85,176.5) -- cycle ;
\draw [color={rgb, 255:red, 116; green, 104; blue, 207 }  ,draw opacity=1 ][fill={rgb, 255:red, 116; green, 104; blue, 207 }  ,fill opacity=1 ][line width=1.5]    (117,137.5) -- (158,137.5) ;
\draw  [color={rgb, 255:red, 116; green, 104; blue, 207 }  ,draw opacity=1 ][fill={rgb, 255:red, 116; green, 104; blue, 207 }  ,fill opacity=1 ] (112,137.5) .. controls (112,136.12) and (113.12,135) .. (114.5,135) .. controls (115.88,135) and (117,136.12) .. (117,137.5) .. controls (117,138.88) and (115.88,140) .. (114.5,140) .. controls (113.12,140) and (112,138.88) .. (112,137.5) -- cycle ;
\draw  [color={rgb, 255:red, 116; green, 104; blue, 207 }  ,draw opacity=1 ][fill={rgb, 255:red, 116; green, 104; blue, 207 }  ,fill opacity=1 ] (158,137.5) .. controls (158,136.12) and (159.12,135) .. (160.5,135) .. controls (161.88,135) and (163,136.12) .. (163,137.5) .. controls (163,138.88) and (161.88,140) .. (160.5,140) .. controls (159.12,140) and (158,138.88) .. (158,137.5) -- cycle ;
\draw [color={rgb, 255:red, 20; green, 76; blue, 237 }  ,draw opacity=1 ][line width=1.5]    (68,211.5) -- (109,211.5) ;
\draw  [color={rgb, 255:red, 20; green, 76; blue, 237 }  ,draw opacity=1 ][fill={rgb, 255:red, 20; green, 76; blue, 237 }  ,fill opacity=1 ] (63,211.5) .. controls (63,210.12) and (64.12,209) .. (65.5,209) .. controls (66.88,209) and (68,210.12) .. (68,211.5) .. controls (68,212.88) and (66.88,214) .. (65.5,214) .. controls (64.12,214) and (63,212.88) .. (63,211.5) -- cycle ;
\draw  [color={rgb, 255:red, 20; green, 76; blue, 237 }  ,draw opacity=1 ] (109,211.5) .. controls (109,210.12) and (110.12,209) .. (111.5,209) .. controls (112.88,209) and (114,210.12) .. (114,211.5) .. controls (114,212.88) and (112.88,214) .. (111.5,214) .. controls (110.12,214) and (109,212.88) .. (109,211.5) -- cycle ;
\draw [color={rgb, 255:red, 20; green, 76; blue, 237 }  ,draw opacity=1 ][line width=1.5]    (388,211.5) -- (429,211.5) ;
\draw  [color={rgb, 255:red, 20; green, 76; blue, 237 }  ,draw opacity=1 ][fill={rgb, 255:red, 20; green, 76; blue, 237 }  ,fill opacity=1 ] (383,211.5) .. controls (383,210.12) and (384.12,209) .. (385.5,209) .. controls (386.88,209) and (388,210.12) .. (388,211.5) .. controls (388,212.88) and (386.88,214) .. (385.5,214) .. controls (384.12,214) and (383,212.88) .. (383,211.5) -- cycle ;
\draw  [color={rgb, 255:red, 20; green, 76; blue, 237 }  ,draw opacity=1 ] (429,211.5) .. controls (429,210.12) and (430.12,209) .. (431.5,209) .. controls (432.88,209) and (434,210.12) .. (434,211.5) .. controls (434,212.88) and (432.88,214) .. (431.5,214) .. controls (430.12,214) and (429,212.88) .. (429,211.5) -- cycle ;
\draw [color={rgb, 255:red, 77; green, 158; blue, 235 }  ,draw opacity=1 ][line width=1.5]    (545,177.5) -- (586,177.5) ;
\draw [shift={(565.5,177.5)}, rotate = 0] [fill={rgb, 255:red, 77; green, 158; blue, 235 }  ,fill opacity=1 ][line width=0.08]  [draw opacity=0] (11.61,-5.58) -- (0,0) -- (11.61,5.58) -- cycle    ;
\draw  [color={rgb, 255:red, 77; green, 158; blue, 235 }  ,draw opacity=1 ][fill={rgb, 255:red, 77; green, 158; blue, 235 }  ,fill opacity=1 ] (540,177.5) .. controls (540,176.12) and (541.12,175) .. (542.5,175) .. controls (543.88,175) and (545,176.12) .. (545,177.5) .. controls (545,178.88) and (543.88,180) .. (542.5,180) .. controls (541.12,180) and (540,178.88) .. (540,177.5) -- cycle ;
\draw [color={rgb, 255:red, 144; green, 19; blue, 254 }  ,draw opacity=0.55 ][line width=1.5]    (90,98.5) -- (131,98.5) ;
\draw  [color={rgb, 255:red, 144; green, 19; blue, 254 }  ,draw opacity=0.55 ] (85,98.5) .. controls (85,97.12) and (86.12,96) .. (87.5,96) .. controls (88.88,96) and (90,97.12) .. (90,98.5) .. controls (90,99.88) and (88.88,101) .. (87.5,101) .. controls (86.12,101) and (85,99.88) .. (85,98.5) -- cycle ;
\draw  [color={rgb, 255:red, 144; green, 19; blue, 254 }  ,draw opacity=0.55 ] (131,98.5) .. controls (131,97.12) and (132.12,96) .. (133.5,96) .. controls (134.88,96) and (136,97.12) .. (136,98.5) .. controls (136,99.88) and (134.88,101) .. (133.5,101) .. controls (132.12,101) and (131,99.88) .. (131,98.5) -- cycle ;
\draw [color={rgb, 255:red, 144; green, 19; blue, 254 }  ,draw opacity=0.55 ][line width=1.5]    (459,99.5) -- (588,100) ;
\draw [shift={(523.5,99.75)}, rotate = 0.22] [fill={rgb, 255:red, 144; green, 19; blue, 254 }  ,fill opacity=0.55 ][line width=0.08]  [draw opacity=0] (11.61,-5.58) -- (0,0) -- (11.61,5.58) -- cycle    ;
\draw  [color={rgb, 255:red, 144; green, 19; blue, 254 }  ,draw opacity=0.55 ][fill={rgb, 255:red, 144; green, 19; blue, 254 }  ,fill opacity=0.55 ] (454,99.5) .. controls (454,98.12) and (455.12,97) .. (456.5,97) .. controls (457.88,97) and (459,98.12) .. (459,99.5) .. controls (459,100.88) and (457.88,102) .. (456.5,102) .. controls (455.12,102) and (454,100.88) .. (454,99.5) -- cycle ;
\draw  [color={rgb, 255:red, 144; green, 19; blue, 254 }  ,draw opacity=0.55 ] (588,100) .. controls (588,98.62) and (589.12,97.5) .. (590.5,97.5) .. controls (591.88,97.5) and (593,98.62) .. (593,100) .. controls (593,101.38) and (591.88,102.5) .. (590.5,102.5) .. controls (589.12,102.5) and (588,101.38) .. (588,100) -- cycle ;
\draw [color={rgb, 255:red, 116; green, 104; blue, 207 }  ,draw opacity=1 ][fill={rgb, 255:red, 116; green, 104; blue, 207 }  ,fill opacity=1 ][line width=1.5]    (438,139.5) -- (520,139.5) ;
\draw  [color={rgb, 255:red, 116; green, 104; blue, 207 }  ,draw opacity=1 ][fill={rgb, 255:red, 116; green, 104; blue, 207 }  ,fill opacity=1 ] (433,139.5) .. controls (433,138.12) and (434.12,137) .. (435.5,137) .. controls (436.88,137) and (438,138.12) .. (438,139.5) .. controls (438,140.88) and (436.88,142) .. (435.5,142) .. controls (434.12,142) and (433,140.88) .. (433,139.5) -- cycle ;
\draw  [color={rgb, 255:red, 20; green, 76; blue, 237 }  ,draw opacity=1 ][fill={rgb, 255:red, 20; green, 76; blue, 237 }  ,fill opacity=1 ] (236.72,66) -- (248.28,66) .. controls (249.78,66) and (251,67.34) .. (251,69) .. controls (251,70.66) and (249.78,72) .. (248.28,72) -- (236.72,72) .. controls (235.22,72) and (234,70.66) .. (234,69) .. controls (234,67.34) and (235.22,66) .. (236.72,66) -- cycle ;
\draw  [color={rgb, 255:red, 77; green, 158; blue, 235 }  ,draw opacity=1 ][fill={rgb, 255:red, 77; green, 158; blue, 235 }  ,fill opacity=1 ] (236.72,86) -- (248.28,86) .. controls (249.78,86) and (251,87.34) .. (251,89) .. controls (251,90.66) and (249.78,92) .. (248.28,92) -- (236.72,92) .. controls (235.22,92) and (234,90.66) .. (234,89) .. controls (234,87.34) and (235.22,86) .. (236.72,86) -- cycle ;
\draw  [color={rgb, 255:red, 116; green, 104; blue, 207 }  ,draw opacity=1 ][fill={rgb, 255:red, 116; green, 104; blue, 207 }  ,fill opacity=1 ] (236.72,106) -- (248.28,106) .. controls (249.78,106) and (251,107.34) .. (251,109) .. controls (251,110.66) and (249.78,112) .. (248.28,112) -- (236.72,112) .. controls (235.22,112) and (234,110.66) .. (234,109) .. controls (234,107.34) and (235.22,106) .. (236.72,106) -- cycle ;
\draw  [color={rgb, 255:red, 144; green, 19; blue, 254 }  ,draw opacity=0.55 ][fill={rgb, 255:red, 144; green, 19; blue, 254 }  ,fill opacity=0.55 ] (236.72,127) -- (248.28,127) .. controls (249.78,127) and (251,128.34) .. (251,130) .. controls (251,131.66) and (249.78,133) .. (248.28,133) -- (236.72,133) .. controls (235.22,133) and (234,131.66) .. (234,130) .. controls (234,128.34) and (235.22,127) .. (236.72,127) -- cycle ;
\draw  [color={rgb, 255:red, 116; green, 104; blue, 207 }  ,draw opacity=1 ] (520,139.5) .. controls (520,138.12) and (521.12,137) .. (522.5,137) .. controls (523.88,137) and (525,138.12) .. (525,139.5) .. controls (525,140.88) and (523.88,142) .. (522.5,142) .. controls (521.12,142) and (520,140.88) .. (520,139.5) -- cycle ;
\draw  [color={rgb, 255:red, 77; green, 158; blue, 235 }  ,draw opacity=1 ] (586,177.5) .. controls (586,176.12) and (587.12,175) .. (588.5,175) .. controls (589.88,175) and (591,176.12) .. (591,177.5) .. controls (591,178.88) and (589.88,180) .. (588.5,180) .. controls (587.12,180) and (586,178.88) .. (586,177.5) -- cycle ;

\draw (46,11.4) node [anchor=north west][inner sep=0.75pt]    {$\mathbb{B}_{LZZ}$};
\draw (366,11.4) node [anchor=north west][inner sep=0.75pt]    {$\mathbb{B}_{EP}$};
\draw (82,253.4) node [anchor=north west][inner sep=0.75pt]    {$a_{2}$};
\draw (535,248.4) node [anchor=north west][inner sep=0.75pt]    {$\overline{a_{3}}$};
\draw (401,252.4) node [anchor=north west][inner sep=0.75pt]    {$a_{2}$};
\draw (262,38.4) node [anchor=north west][inner sep=0.75pt]    {$ \begin{array}{l}
Type\\
I\\
II\\
III\\
IV
\end{array}$};
\draw (127,253.4) node [anchor=north west][inner sep=0.75pt]    {$a_{3}$};
\draw (581,248.4) node [anchor=north west][inner sep=0.75pt]    {$\overline{a_{2}}$};
\draw (445,253.4) node [anchor=north west][inner sep=0.75pt]    {$a_{3}$};

\end{tikzpicture}

\caption{Illustrative example of the diagram bijection theorem. $\mathbb{B}_{LZZ}$ and $\mathbb{B}_{EP}$ denote the levelsets zigzag persistence barcode and the extended persistence barcode respectively. In $\mathbb{B}_{EP}$, the dotted line separates $\R$ from $\R^{\mathrm{op}}$. Intervals marked with an arrow correspond to intervals marked with a "$+$" exponent in Theorem \ref{bijection_alternative}.}
    \label{fig:my_label}
\end{figure}

\noindent We consider the simple, but very illustrative, case of the circle $S^1$ embedded into the real plane $\R^2$, paired with a sub-levelsets filtration induced by the projection onto the horizontal axis (see Figure \ref{circle}).

\begin{example}[The circle $S^1$]
Let $\X=S^1 = \{(x,y) \in \R^2 \mid x^2 + y^2 = 1\}$, and consider the first coordinate projection $f:\X\rightarrow\R$ as defined in Figure \ref{circle}, where $a_1 = -1,a_2 =1$ are the critical values of the function $f$ and $s_0,s_1,s_2$ are regular values of $f$ such that $-\infty<s_0<a_1<s_1<a_2<s_2<\infty$.\\

\noindent \textbf{Levelsets zigzag persistence.} For $p\in\N$, one has levelsets zigzag modules
$$H_p(f^{-1}(s_0))\rightarrow H_p(f^{-1}([s_0,s_1]))\leftarrow   H_p(f^{-1}(s_1))\rightarrow H_p(f^{-1}([s_1,s_2]))\leftarrow H_p(f^{-1}(s_2)).$$

\noindent The case $p=0$ leads to the zigzag module 

$$0\rightarrow K\leftarrow K\oplus K\rightarrow K\leftarrow 0,$$ that decomposes into indecomposable summands as 
$$\left(0\rightarrow K\leftarrow K\rightarrow K\leftarrow 0\right)\oplus (0\rightarrow 0\leftarrow K\rightarrow 0\leftarrow 0),$$

\noindent leading to the levelsets zigzag barcode given by the intervals $[\X_0^1,\X_1^2]_0$ and $[\X_1^1,\X_1^1]_0$, where the interval index is used to keep track of the dimensionality of the features.\\

\noindent \textbf{Extended persistence.} For $p\in\N$, one has the extended persistence modules

\[\xymatrix{
     H_p(f^{-1}(s_0)) \ar[r] & H_p(f^{-1}([s_0,s_1])) \ar[r] & H_p(f^{-1}([s_0,s_2])) \ar[dll]\\
              H_p(f^{-1}(S^1,\{s_2\})) \ar[r] & H_p(f^{-1}(S^1,[s_1,s_2])) \ar[r]& H_p(f^{-1}(S^1,S^1)).
    }\]

\noindent The case $p=0$ gives the indecomposable module 

$$0\rightarrow K\rightarrow K\rightarrow K\rightarrow 0\rightarrow 0,$$ and the case $p=1$ gives the indecomposable module
$$0\rightarrow 0\rightarrow K\rightarrow K\rightarrow K\rightarrow 0,$$
\noindent leading to the extended barcode given by $[\X_0^1,(\X_0^2,\X_2^2)]_0$ and $[\X_0^2,(\X_0^2,\X_1^2)]_1$.\\

\noindent \textbf{Correspondence.} One may refer to the table in Theorem \ref{bijection_alternative} to obtain the matching between the levelsets zigzag and extended barcodes of $f$:

\begin{equation*}
    \begin{cases}
    [\X_1^1,\X_1^1]_0\leftrightarrow [\X_0^2,(\X_0^2,\X_1^2)]_1,\\
    [\X_0^1,\X_1^2]_0\leftrightarrow [\X_0^1,(\X_0^2,\X_2^2)]_0.
    \end{cases}
\end{equation*}

\noindent \textbf{Interpretation.} The levelsets zigzag persistence has a way of detecting cycles that is a bit less natural than ordinary persistence and extended persistence. In fact, the cycle formed by $\X=S^1$ is not represented with dimension-$1$ intervals. Instead, it creates a particular signature entirely encoded in the dimension-$0$ intervals.

\begin{figure}
    \centering
    
    \label{circle}
    
\tikzset{every picture/.style={line width=0.75pt}} 

\begin{tikzpicture}[x=0.75pt,y=0.75pt,yscale=-1,xscale=1]

\draw   (254,105.5) .. controls (254,68.77) and (283.77,39) .. (320.5,39) .. controls (357.23,39) and (387,68.77) .. (387,105.5) .. controls (387,142.23) and (357.23,172) .. (320.5,172) .. controls (283.77,172) and (254,142.23) .. (254,105.5) -- cycle ;
\draw    (200,202) -- (438,202) ;
\draw [shift={(440,202)}, rotate = 180] [color={rgb, 255:red, 0; green, 0; blue, 0 }  ][line width=0.75]    (10.93,-3.29) .. controls (6.95,-1.4) and (3.31,-0.3) .. (0,0) .. controls (3.31,0.3) and (6.95,1.4) .. (10.93,3.29)   ;
\draw  [dash pattern={on 0.84pt off 2.51pt}]  (254,105.5) -- (255,202) ;
\draw  [dash pattern={on 0.84pt off 2.51pt}]  (387,105.5) -- (388,202) ;
\draw    (217,197) -- (217,206) ;
\draw    (325,197) -- (325,206) ;
\draw    (417,197) -- (417,206) ;

\draw (449,192.4) node [anchor=north west][inner sep=0.75pt]    {$f$};
\draw (385,28.4) node [anchor=north west][inner sep=0.75pt]    {$S^{1}$};
\draw (233,204.4) node [anchor=north west][inner sep=0.75pt]    {$a_{1}=-1$};
\draw (360,204.4) node [anchor=north west][inner sep=0.75pt]    {$a_{2}=1$};
\draw (410,204.4) node [anchor=north west][inner sep=0.75pt]    {$s_{2}$};
\draw (303,204.4) node [anchor=north west][inner sep=0.75pt]    {$s_{1}=0$};
\draw (208,204.4) node [anchor=north west][inner sep=0.75pt]    {$s_{0}$};

\end{tikzpicture}
\caption{Height function on the circle $S^1$.}
\end{figure}
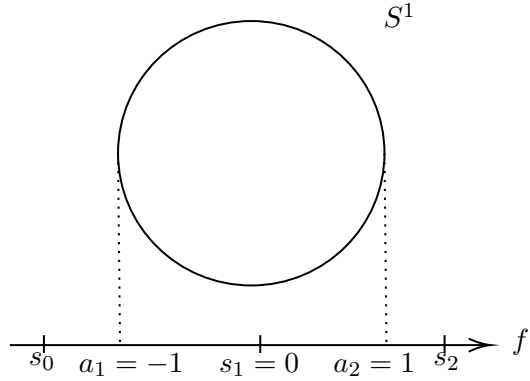

\end{example}

\subsection{Graphs}

Although ordinary persistence computations involve nearly-linear complexity for the case of graphs, this was not the case for zigzag persistence computations on graphs until the recent advances made in \cite{dey2021computing_}, where the authors provide two algorithms for computing $0$-dimensional and $1$-dimensional zigzag persistence barcodes, and extend the first algorithm to an arbitrary dimension $p-1$ for $\R^p$-embedded complexes.

\subsection{Manifold-like complexes}

Shortly after, the same authors look at the more general setting of manifold-like complexes in \cite{dey2021computing}, presenting a polynomial-time
algorithm for each type of levelsets persistence interval, that computes an optimal sequence of so-called \textit{levelsets persistent $p$-cycles} for weak $(p+1)$-pseudomanifolds. This is a consequent step towards efficient computing of zigzag persistence, as general optimal cycle challenges for homology usually involve NP-hard complexity. The idea behind the algorithms is to make use of the one-to-one correspondence between optimal cycles in the complex and minimum-weight cuts in a specific dual graph.

\subsection{Fast computation of zigzag persistence}

In \cite{Dey22}, the authors develop a fast algorithm to convert any zigzag filtration into an up-down filtration, making possible to apply any ordinary persistence algorithm for zigzag persistence. Thanks to this method, any algorithmic progress in ordinary persistence will automatically induce an improvement for zigzag computations.

\section{Applications}

\subsection{Computational sheaf theory and projected barcodes}

In this section, we will freely use the standard terminology of sheaf theory, and refer the reader to \cite{KS90} for an introduction. Let $D^b(K_\R)$ denote the bounded derived category of sheaves of $K$-vector spaces on $\R$ equipped with the standard topology, and let $D^b_{\R c}(K_\R)$ denote its full subcategory whose objects have constructible cohomology. In \cite{KS18}, Kashiwara and Schapira have introduced an interleaving like distance on  $D^b(K_\R)$, called the convolution distance and denoted $d_C$. In the same work, they also proved that  all sheaves $F \in D^b_{\R c}(K_\R)$ decompose as a direct sum of constant sheaves on intervals, not necessarily concentrated in the same degree. This graded collection of intervals, denoted $\mathbb{B}(F)$, has been  later called the graded barcode of $F$ in \cite[Definition 2.13]{BG18}, in which the authors introduce the graded-bottleneck distance $d_B$ between graded barcodes and prove the following isometry theorem:

\begin{theorem}[\cite{BG18}]
Let $F,G\in D^b_{\R c}(K_\R)$, then:

\[d_C(F,G) = d_B(\mathbb{B}(F),\mathbb{B}(G)).\]
\end{theorem}

Given $f : \mathbb{X} \longrightarrow \R$ a continuous function, we denote the derived direct image of the constant sheaf on $X$ by $f$ by $R f_\ast K_X$. It is the sheaf analogue construction to the collection of levelsets persistence modules associated to $f$. This analogy was made precise in \cite{BGO19} where the authors prove that these barcodes determine each others.

First, the authors show that the collection of levelsets persistence modules of a real valued function is redundant, in the following sense. Let $f : \mathbb{X} \to \R$ be such that $\mathcal{L}_p(f)$ is pointwise finite dimensional for all $p \in \Z_{\geq 0}$. We denote by $\mathbb{B}^\textbf{o}(\mathcal{L}_p(f))$ (resp. $\mathbb{B}^{\textbf{c}_2}(\mathcal{L}_p(f))$) the sub-multi-set of $\mathbb{B}(\mathcal{L}_p(f))$ constituted of intervals of type \textbf{o} (resp. $\textbf{c}_2$).

\begin{proposition}[\cite{BGO19}]
There is a bijection  $\phi_f^p : \mathbb{B}^\textbf{o}(\mathcal{L}_p(f)) \longrightarrow \mathbb{B}^{\textbf{c}_2}(\mathcal{L}_{p+1}(f))$, given by $[a,b]_{\mathrm{BL}} \mapsto [b,a]_{\mathrm{BL}}$.
\end{proposition}

We denote by $\mathbb{B}(\mathcal{L}_\ast(f))$ the multi-set $\cup_p \mathbb{B}(\mathcal{L}_p(f)) \times \{p\}$ quotiented by the equivalence relation identifying pairs of intervals of the form $((I,p), (\phi_f^p(I), p+1))$.  

\begin{theorem}[\cite{BGO19}]
Let $X$ be a locally contractible topological space, and $f : \mathbb{X} \longrightarrow \R$ be a continuous map such that $R f_\ast K_X \in D^b_{\mathbb{R} c}(K_X)$. There is a bijection between $\mathbb{B}(R f_\ast K_X)$ and $\mathbb{B}(\mathcal{L}_\ast(f))$ given by: \[\langle a,b \rangle^p \mapsto (\langle a,b \rangle_{\mathrm{BL}},p),\]

where the notation $\langle a,b\rangle^p$, means that the interval $\langle a,b\rangle$ appears in degree $p$ in the graded-barcode of  $R f_\ast K_X$.

\end{theorem}

Combined with the bijection of section \ref{s:levelsetszigzag}, we obtain the following.

\begin{corollary}
Let $a_1<...<a_n$ be the critical values of the Morse type function $f:\X\rightarrow\R$, and in-between regular values $-\infty<s_0<a_1<s_1<\cdots<s_{n-1}<a_n<s_n<\infty$. The collection in all degree of the extended persistence barcode of $f$ is equivalent to the graded-barcode of $R f_\ast K_X$. This equivalence can be computed in linear time with respect to the number of bars in the extended persistence barcode.
\end{corollary}

\begin{remark}
It shall be noted that in this setting, we obtain a mean to compute the derived direct image of a sheaf without the need of computing any injective resolution, as it was done in \cite{Brown21}, which is way more time costly.
\end{remark}

In \cite{BP22}, the authors introduce the projected barcodes of a Piecewise Linear (PL) multi-parameter filtration of a simplicial complex, and prove several important properties this invariant using sheaf theory. More precisely, let $X$ be a finite simplicial complex, $\mathbb{X}$ its geometric realization, and $f : X \to \mathbb{R}^n$ be a PL map, with geometric realization $|f| : \mathbb{X} \to \R^n$. Let $\mathfrak{F}$ be a set of continuous functions from $\R^n$ to $\R$.

\begin{definition}[\cite{BP22}]
The $\mathfrak{F}$-projected barcode of $R |f|_\ast K_\mathbb{X} $ is the collection of graded-barcodes: \[\mathbb{B}^\mathfrak{F}(R |f|_\ast K_\mathbb{X}) := (\mathbb{B}(R |u\circ f|_\ast K_\mathbb{X})_{u \in \mathfrak{F}}.\]
\end{definition}

\begin{corollary}
The $\mathfrak{F}$-projected barcode of $R |f|_\ast K_\mathbb{X} $ is equivalent to the collection in all degree of the extended persistence barcodes of the PL maps $u\circ f$, for $u \in \mathfrak{F}$. This equivalence can be computed in linear time in the number of simplices in $X$.
\end{corollary}

\subsection{Generalized rank invariant of 2-parameters persistence modules}

Given a $2$-parameters persistence module $M : \R^2 \longrightarrow \mathrm{Vect}_K$, its generalized rank invariant is the map associating to each interval $I$ of the poset $\R^2$, the rank of the map
$\textnormal{lim}~ M_{|I} \longrightarrow \textnormal{colim}~ M_{|I}$. It is an incomplete invariant of mutli-parameters persistence modules, and computing it efficiently is an important question for practical applications of multi-parameters persistence. In \cite{Kim21}, the authors show that computing the generalized rank invariant of $M$ over $I$, amounts to computing the generalized invariant of $M$ restricted to a zigzag path tracing the boundary of $I$. This method has a substantial decrease in complexity than previous existing ones, and also permits to define an efficient test for a $2$-parameters persistence module to be interval decomposable.

\subsection{Extended persistent homology transform}

In \cite{Tur22}, the authors introduce the Extended Persistent Homology Transform, as the extended version of the well-known Persistent Homology Transform \cite{Tur13}. They provide evidence of its usefulness for image classification.




\section{Conclusion}
\label{ccl}

Extended, zigzag and levelsets persistence have been introduced more than ten years ago now, and offer substantial generalization of ordinary persistence, while still having nice computational properties. In addition, they are at the heart of recent advances for multi-parameter persistence. Given the recent advances regarding this literature, we felt the need for a self-contained and unified introduction presenting its main results and constructions, and showing their implications for active research topics, such as computational sheaf theory and multi-parameter persistence.

\subsection*{Acknowledgements.}
The authors are grateful to Kathryn Hess for her support in the writing of this article. This work was supported by the Swiss Innovation Agency (Innosuisse project 41665.1
IP-ICT).

\printbibliography

\end{document}